\documentclass[reqno,12pt]{preprint}
\usepackage[marginparwidth=1in]{geometry}
\usepackage[full]{textcomp}
\usepackage[osf]{newtxtext}
\usepackage{cabin}
\usepackage{enumerate}
\usepackage{enumitem}
\usepackage{bbm}
\usepackage{dsfont}

\usepackage[varqu,varl]{inconsolata}
\usepackage[cal=boondoxo]{mathalfa}

\usepackage{comment}
\usepackage{hyperref}
\usepackage{breakurl}
\usepackage{mathrsfs}
\usepackage{booktabs}
\usepackage{caption}

\usepackage{tikz}
\usetikzlibrary{decorations.pathreplacing,decorations.markings}
\usepackage{amsmath} 
\usepackage{mhequ} 
\usepackage{mhsymb} 
\usepackage{mhenvs} 
\usepackage{microtype}
\usepackage{amssymb} 
\usepackage{eufrak}

\usepackage{wasysym}
\usepackage{centernot}

\usepackage{oldgerm}

\newcommand{\Exp}{\mathbf{E}}

\newcommand{\hffNn}{\hat{\ff}^{N,n}}
\newcommand{\fFN}{{\fF}^{N}}

\renewcommand{\R}{\mathbb{R}}

\renewcommand{\N}{\mathbb{N}}

\newcommand{\T}{\mathbb{T}}
\renewcommand{\Z}{\mathbb{Z}}
\renewcommand{\E}{\mathbb{E}}
\renewcommand{\P}{\mathbb{P}}

\newcommand{\J}{\mathbb{J}}

\newcommand{\cA}{\mathcal{A}}

\newcommand{\cD}{\mathcal{D}}
\newcommand{\cF}{\mathcal{F}}

\newcommand{\cJ}{\mathcal{J}}
\newcommand{\cK}{\mathcal{K}}
\newcommand{\cL}{\mathcal{L}}
\newcommand{\cM}{\mathcal{M}}
\newcommand{\cN}{\mathcal{N}}
\newcommand{\cS}{\mathcal{S}}
\newcommand{\cV}{\mathcal{V}}

\newcommand{\dd}{\mathrm{d}}      %%%%%% d

\newcommand{\SH}{\mathscr{H}}

\newcommand{\fFNn}{\mathfrak{F}^{N,n}}

\newcommand{\fb}{\mathfrak{b}}
\newcommand{\fc}{\mathcal{c}}

\newcommand{\ff}{\mathfrak{f}}
\newcommand{\ft}{\mathfrak{t}}
\newcommand{\fp}{\mathfrak{p}}
\newcommand{\fg}{\mathfrak{g}}

\newcommand{\fs}{\mathfrak{s}}
\newcommand{\fh}{\mathfrak{h}}
\newcommand{\nf}{\mathfrak{n}^N}

\newcommand{\ffNn}{\ff^{N,n}}
\newcommand{\tffNn}{\tilde{\ff}^{N,n}}
\newcommand{\fbNn}{\fb^{N,n}}
\newcommand{\tfbNn}{\tilde{\fb}^{N,n}}
\newcommand{\fhNn}{\fh^{N,n}}
\newcommand{\tfhNn}{\tilde{\fh}^{N,n}}
\newcommand{\sym}{\mathrm{sym}}

\newcommand{\bv}{\mathbf{v}}

\def\one{\mathrm{(I)}}
\def\two{\mathrm{(II)}}

\newcommand{\1}{\mathds{1}}

\def\restr{\mathord{\upharpoonright}}

\def\Wick#1{\mathopen{{:}}#1\mathclose{{:}}}

\newcommand{\vphi}{\varphi}
\def\sym{{\mathrm{sym}}}
\newcommand{\Di}{\mathrm{Diag}}
\newcommand{\oD}{\mathrm{off}}
\newcommand{\oDi}{\mathrm{off_1}}
\newcommand{\ooDi}{\mathrm{off_2}}
\newcommand{\bulk}{\mathrm{bulk}}
\newcommand{\Dbulk}{\mathrm{D}_\bulk}

\newcommand{\nueff}{\nu_{\rm eff}}

\colorlet{darkblue}{blue!90!black}
\colorlet{darkred}{red!90!black}
\colorlet{darkgreen}{green!50!black}
\colorlet{darkyellow}{yellow!90!black}

\def\one{\mathrm{(I)}}
\def\two{\mathrm{(II)}}

\newcommand{\half}{\frac{1}{2}}

\def\indN#1{{	\J^N_{#1}	}}
\newcommand{\fock}{\Gamma L^2}	% Fock space

\def\sint{I}					% Stochastic integral
\def\wc{\SH}					% Wiener Chaos
\def\swn{{	\eta	}}
\newcommand{\gen}{\cL^N}

\newcommand{\gensy}{\cL_0}

\newcommand{\gena}{\cA^{N}}
\newcommand{\genap}{\cA^{N}_+}
\newcommand{\genam}{\cA^{N}_-}

\newcommand{\nonlin}{\cK^{N}}
\newcommand{\Nonlin}{\cN^{N}}

\newcommand{\Op}{\mathcal{H}^N}

\usetikzlibrary{shapes.misc}
\usetikzlibrary{shapes.symbols}
\usetikzlibrary{shapes.geometric}
\usetikzlibrary{snakes}
\usetikzlibrary{decorations}
\usetikzlibrary{decorations.markings}

\tikzstyle{tinydots}=[dash pattern=on \pgflinewidth off 2*\pgflinewidth]

\def\drawx{\draw[-,solid] (-3pt,-3pt) -- (3pt,3pt);\draw[-,solid] (-3pt,3pt) -- (3pt,-3pt);}
\tikzset{
	n/.style={circle,fill=black!15,inner sep=0pt, minimum size=0.7mm},
	l/.style={inner sep=2pt,label=center:{...}},
	r/.style={circle,fill=red,inner sep=0pt, minimum size=1mm},
	simple/.style={circle,inner sep=0pt, minimum size=0.5mm},
	d/.style={circle,fill=black,inner sep=0pt, minimum size=0.7mm},
	c/.style={circle,draw=black,fill=white,inner sep=0pt, minimum size=1.2mm},
	root/.style={circle,fill=testcolor,inner sep=0pt, minimum size=2mm},
	dot/.style={circle,fill=black,inner sep=0pt, minimum size=1mm},
	var/.style={circle,fill=black!10,draw=black,inner sep=0pt, minimum size=2mm},
	dotred/.style={circle,fill=black!50,inner sep=0pt, minimum size=2mm},
	generic/.style={semithick,shorten >=1pt,shorten <=1pt},
	dist/.style={ultra thick,draw=testcolor,shorten >=1pt,shorten <=1pt},
	testfcn/.style={ultra thick,testcolor,shorten >=1pt,shorten <=1pt,<-},
	testfcnx/.style={ultra thick,testcolor,shorten >=1pt,shorten <=1pt,<-,
		postaction={decorate,decoration={markings,mark=at position 0.6 with {\drawx}}}},
	kprime/.style={semithick,shorten >=1pt,shorten <=1pt,densely dashed,->},
	kprimex/.style={semithick,shorten >=1pt,shorten <=1pt,densely dashed,->,
		postaction={decorate,decoration={markings,mark=at position 0.4 with {\drawx}}}},
	kernel/.style={semithick,shorten >=1pt,shorten <=1pt,->},
	multx/.style={shorten >=1pt,shorten <=1pt,
		postaction={decorate,decoration={markings,mark=at position 0.5 with {\drawx}}}},
	kernelx/.style={semithick,shorten >=1pt,shorten <=1pt,->,
		postaction={decorate,decoration={markings,mark=at position 0.4 with {\drawx}}}},
	kernel1/.style={->,semithick,shorten >=1pt,shorten <=1pt,postaction={decorate,decoration={markings,mark=at position 0.45 with {\draw[-] (0,-0.1) -- (0,0.1);}}}},
	kernel2/.style={->,semithick,shorten >=1pt,shorten <=1pt,postaction={decorate,decoration={markings,mark=at position 0.45 with {\draw[-] (0.05,-0.1) -- (0.05,0.1);\draw[-] (-0.05,-0.1) -- (-0.05,0.1);}}}},
	kernelBig/.style={semithick,shorten >=1pt,shorten <=1pt,decorate, decoration={zigzag,amplitude=1.5pt,segment length = 3pt,pre length=2pt,post length=2pt}},
	rho/.style={dotted,semithick,shorten >=1pt,shorten <=1pt},
	renorm/.style={shape=circle,fill=white,inner sep=1pt},
	labl/.style={shape=rectangle,fill=white,inner sep=1pt},
	cumu2n/.style={inner sep=3pt},
	cumu2/.style={draw=red!50,fill=red!20},
	cumu3/.style={regular polygon, regular polygon sides=3,draw=red!50,rounded corners=3pt,fill=red!20,minimum size=5mm},
	cumu4/.style={regular polygon, regular polygon sides=4,draw=red!50,rounded corners=3pt,fill=red!20,minimum size=7mm},
	cumu5/.style={regular polygon, regular polygon sides=5,draw=red!50,rounded corners=3pt,fill=red!20,minimum size=5mm},
	xi/.style={circle,fill=symbols!10,draw=symbols,inner sep=0pt,minimum size=1.2mm},
	xix/.style={crosscircle,fill=symbols!10,draw=symbols,inner sep=0pt,minimum size=1.2mm},
	xib/.style={circle,fill=symbols!10,draw=symbols,inner sep=0pt,minimum size=1.6mm},
	xibx/.style={crosscircle,fill=symbols!10,draw=symbols,inner sep=0pt,minimum size=1.6mm},
	not/.style={circle,fill=symbols,draw=symbols,inner sep=0pt,minimum size=0.5mm},
	>=stealth,
}

\tikzset{
	on each segment/.style={
		decorate,
		decoration={
			show path construction,
			moveto code={},
			lineto code={
				\path [#1]
				(\tikzinputsegmentfirst) -- (\tikzinputsegmentlast);
			},
			curveto code={
				\path [#1] (\tikzinputsegmentfirst)
				.. controls
				(\tikzinputsegmentsupporta) and (\tikzinputsegmentsupportb)
				..
				(\tikzinputsegmentlast);
			},
			closepath code={
				\path [#1]
				(\tikzinputsegmentfirst) -- (\tikzinputsegmentlast);
			},
		},
	},
	mid arrow/.style={postaction={decorate,decoration={
				markings,
				mark=at position .5 with {\arrow[#1]{stealth}}
	}}},
}
\makeatletter
\def\DeclareSymbol#1#2#3{\expandafter\gdef\csname MH@symb@#1\endcsname{\tikzsetnextfilename{symbol#1}\tikz[baseline=#2,scale=0.15,draw=symbols]{#3}}\expandafter\gdef\csname MH@symb@#1s\endcsname{\scalebox{0.7}{\tikzsetnextfilename{symbol#1s}\tikz[baseline=#2,scale=0.15,draw=symbols]{#3}}}}
\def\<#1>{\csname MH@symb@#1\endcsname}
\makeatother

\definecolor{Red}{rgb}{1,0,0}
\definecolor{Blue}{rgb}{0,0,1}
\definecolor{Olive}{rgb}{0.41,0.55,0.13}
\definecolor{Yarok}{rgb}{0,0.5,0}
\definecolor{Green}{rgb}{0,1,0}
\definecolor{MGreen}{rgb}{0,0.8,0}
\definecolor{DGreen}{rgb}{0,0.65,0}
\definecolor{Yellow}{rgb}{1,1,0}
\definecolor{Cyan}{rgb}{0,1,1}
\definecolor{Magenta}{rgb}{1,0,1}
\definecolor{Orange}{rgb}{1,.5,0}
\definecolor{Violet}{rgb}{.5,0,.5}
\definecolor{Purple}{rgb}{.75,0,.25}
\definecolor{Brown}{rgb}{.75,.5,.25}
\definecolor{Grey}{rgb}{.7,.7,.7}
\definecolor{Black}{rgb}{0,0,0}
\definecolor{dr}{rgb}{0.8,0,0}
\definecolor{db}{rgb}{0,0,0.8}

\newcounter{list_size}
\newcommand{\connect}[4][dr]{
	\edef\myleftx{10000}
	\edef\myrightx{-10000}
	\edef\mycentery{0}
	\setcounter{list_size}{0}
	\foreach \j in {#2}{\stepcounter{list_size}};
	\edef\n{\arabic{list_size}}
	\ifthenelse{\n=1}
	{
		\draw (#3#2) -- (#4#2);
	}
	{
		\foreach \j in {#2} {
			\path (#3\j); \pgfgetlastxy{\XCoord}{\YCoord};
			\pgfmathsetmacro{\lx}{min(\myleftx,\XCoord)};
			\pgfmathsetmacro{\rx}{max(\myrightx,\XCoord)};
			\pgfmathsetmacro{\cy}{\mycentery+\YCoord/(2*\n)};
			\path (#4\j); \pgfgetlastxy{\XCoord}{\YCoord};
			\pgfmathsetmacro{\cy}{\cy+\YCoord/(2*\n)};
			\global\let\myleftx=\lx
			\global\let\myrightx=\rx
			\global\let\mycentery=\cy
		}
		\foreach \j in {#2} {
			\draw (#3\j) -- (#4\j);
		}
		\draw[thick,#1] (\myleftx pt,\mycentery pt) -- (\myrightx pt,\mycentery pt);
		\foreach \j in {#2} {
			\path (#3\j); \pgfgetlastxy{\XCoord}{\YCoord};
			\node at (\XCoord,\mycentery pt) [simple,fill=#1] {};
		}
	}
}

\newcommand{\const}{\lambda_N}

\def\energy{\CE^N}
\newcommand{\Ll}{\mathrm{L}^N}

\title{Weak coupling limit of the Anisotropic KPZ equation}

\begin{document}

\maketitle

\vspace{-2cm}

\noindent{\large \bf Giuseppe Cannizzaro$^1$,  Dirk Erhard$^2$, Fabio Toninelli$^3$}
\newline

\noindent{\small $^1$University of Warwick, UK\\%
    $^2$Universidade Federal da Bahia, Brazil\\%
    $^3$Technical University of Vienna, Austria\\}

\noindent\email{giuseppe.cannizzaro@warwick.ac.uk, 
erharddirk@gmail.com, \\
fabio.toninelli@tuwien.ac.at}
\newline

\begin{abstract}
In the present work, we study the two-dimensional anisotropic KPZ
equation (AKPZ), which is formally given by
\begin{equation*}
\partial_t h=\tfrac12 \Delta h + \lambda ((\partial_1 h)^2)-(\partial_2 h)^2) +\xi\,,
\end{equation*}
where $\xi$ denotes a space-time white noise and $\lambda>0$ is the so-called coupling constant. 
The AKPZ equation is a {\it critical} SPDE, meaning that not only it is analytically ill-posed but also the
breakthrough path-wise techniques for singular SPDEs [M. Hairer, Ann. Math. 2014] and 
[M. Gubinelli, P. Imkeller and N. Perkowski, Forum of Math., Pi, 2015] are not applicable. 
As shown in [G. Cannizzaro, D. Erhard, F. Toninelli, arXiv, 2020], 
the equation regularised at scale $N$ has a diffusion coefficient 
that diverges logarithmically as the regularisation is removed in the limit $N\to\infty$. 
Here, we study the \emph{weak coupling limit} where $\lambda=\lambda_N=\hat\lambda/\sqrt{\log N}$: 
this is the correct scaling that guarantees that the nonlinearity has a still non-trivial but non-divergent effect. 
In fact, as $N\to\infty$ the sequence of equations converges to the linear stochastic heat equation 
\begin{equ}
\partial_t h =\tfrac{\nueff}{2} \Delta h + \sqrt{\nueff}\xi\,,
\end{equ}
where $\nueff >1$ is explicit and depends non-trivially on $\hat\lambda$. 
This is the first full renormalization-type result for a
critical, singular SPDE which cannot be linearised via Cole-Hopf or any other transformation. 
%and mapped to a polymer in random environment. \fabioText{do you like calling ours a ``renormalization result? I think it can be called so. Maybe we can use this terminology also in the introduction!}\fabioText{I removed the sentence about proof techniques, as it sounded like an anti-climax and was unclear anyway} %% The proof
%        relies on a refined analysis of the
        %% generator of the equation and builds upon the fact that it can
        %% be well approximated by a suitably in Fourier space modulated
        %% version of the generator of the linear equation.
\end{abstract}


\setcounter{tocdepth}{2}       % Put subsubsections in the table of contents
\tableofcontents

\section{Introduction}

%\newpage

The KPZ equation is a (singular) stochastic partial differential equation (SPDE) of evolution type
which models the time-growth of a random surface $h$. Its (formal) expression in $d$-dimensions reads
\begin{equation}\label{eq:KPZ}
	\partial_t h= \tfrac12 \Delta h + \lambda \langle \nabla h, Q\nabla h\rangle +  \xi\,,%\quad t\geq 0,\, x\in\R^d\,.
\end{equation}
where $h=h(t,x)$ for $t\geq 0,\, x\in\T^d$ (the $d$-dimensional unit torus), $\xi$ is a space-time white noise on $\R\times\T^d$, 
i.e. a centred Gaussian process whose covariance is formally given by
$\mathbb E({\xi(x,t)\xi(y,s)})=\delta (x-y)\delta(t-s)$, $Q$ is a $d\times d$ symmetric matrix encoding how 
the growth mechanism depends on the slope  
and $\lambda\geq 0$ is the so-called {\it coupling constant}, tuning the strength of such dependence. 

Introduced in the seminal work~\cite{Kardar},~\eqref{eq:KPZ} owes its importance to the fact that 
a wide variety of dynamical surface growth phenomena such as growth of combustion fronts or bacterial colonies, 
crystal growth on thin films, growth of vicinal surfaces and many others, are conjectured to share the same 
large-scale fluctuations as those of its solution. 

From a mathematical viewpoint, the first problem one faces when dealing with~\eqref{eq:KPZ} is that 
it is analytically ill-posed in any dimension $d\geq 1$ as the space-time white noise $\xi$ is way too rough 
for the nonlinearity to be canonically defined. 
Nevertheless, in the last years significant progress has been made in the study of its large-scale properties and behaviour, 
at least in the {\it sub-} and {\it super-critical} regimes, 
corresponding to dimensions $d=1$ and $d\geq 3$ respectively. 
The case $d=1$ is the only one for which a full (local) solution theory is expected and 
can be obtained via either the breakthrough pathwise theories of  
regularity structures~\cite{Hai} and paracontrolled calculus~\cite{KPZreloaded} (inspired by  
rough paths~\cite{KPZ}), or 
via the energy solutions approach based on martingale arguments~\cite{GubinelliJara2012,GPuni}. 
Concerning the large-scale behaviour, it was shown that the solution $h$ is {\it superdiffusive}, 
meaning that its fluctuations evolve non-trivially on time scales shorter than those of the stochastic 
heat equation (SHE), given by~\eqref{eq:KPZ} with $\lambda=0$. This can be quantified, e.g. 
by means of  the bulk diffusion coefficient $\Dbulk$, 
which measures how the correlations of a process spread in space as a function 
of time and formally is such that the correlation length behaves like $\ell(t)\sim \sqrt{t \times \Dbulk(t)}$\footnote{see~\eqref{def:Dbulk} below, for the definition we will be using and~\cite[Appendix A]{CET} for 
a heuristic connecting the latter and $\ell(t)$}.  For the 1-dimensional KPZ equation, $\Dbulk(t)$ grows like $t^{1/3}$ for $t$ large (see~\cite{BQS}), 
as opposed to SHE whose bulk diffusion coefficient is constant. On a quantitative 
level, these fluctuations have been fully characterised and the scaling limit of~\eqref{eq:KPZ} 
determined~\cite{QSkpz,Vir}. 

In the supercritical regime instead, a recent series of works~\cite{Magnen, Gu2018b, CCM1, LZ, CNN} 
has confirmed, at least under the assumption that $Q$ is the identity matrix (but the result 
should hold for any $Q$), the prediction
first implicitly made in~\cite{Kardar}, that for $\lambda$ small enough the solution $h$ of~\eqref{eq:KPZ} 
is {\it diffusive} and its scaling limit is the solution of a stochastic heat equation with {\it renormalised} coefficients.  

Yet, the {\it critical} case $d=2$ %, by far the most challenging, 
remains poorly understood. 
Formally, in $2$ dimensions~\eqref{eq:KPZ} is scale invariant under diffusive scaling and therefore it is expected 
that finer features of the equation, and in particular the nature of the slope dependence determined by the matrix $Q$, 
might qualitatively influence its properties. %large-scale properties.
Via non-rigorous one-loop Renormalisation Group computations, Wolf conjectured in~\cite{W91} 
that its large scale behaviour will depend on the sign of $\det Q$ - in the Isotropic case, corresponding to 
$\det Q>0$, $\Dbulk(t)\sim t^\beta$ for some   universal $\beta>0$ (known only numerically), while in 
the Anisotropic case, corresponding to $\det Q\leq 0$, $\Dbulk(t)\sim t^\beta$ for $\beta=0$. 
%
%the existence of two distinct large-scale behaviours for~\eqref{eq:KPZ} depending on the sign of $\det Q$ - 
%the Isotropic for $\det Q>0$, and the Anisotropic for $\det Q\leq 0$. 

The present paper focuses on the latter, and more specifically on the case of $Q=Q_{\rm AKPZ}={\rm Diag}(1,-1)$. 
We will refer to~\eqref{eq:KPZ} with $Q_{\rm AKPZ}$ as the Anisotropic KPZ (AKPZ) equation, which reads
\begin{equ}[e:AKPZ1]
\partial_t h =\tfrac{1}{2}\Delta h +\lambda \big((\partial_1 h)^2-(\partial_2 h)^2\big) +\xi\,.
\end{equ}
Let us stress that~\eqref{e:AKPZ1} is {\it critical} as $d=2$ is the
dimension at which Hairer's theory of Regularity Structures~\cite{Hai}
and the other pathwise approaches break down for~\eqref{eq:KPZ} and a
local solution theory is not even expected to hold.  One is therefore
naturally led to first regularise the equation and consequently (try
to) determine the large-scale properties of its solution as the
regularisation is removed.  Contrary to the folklore
belief~\cite{BCT}, in~\cite{CET} we showed that~\eqref{e:AKPZ1} is not
diffusive but {\it logarithmically superdiffusive}.  Translating the
result therein to the torus $\mathbb T^2$, we proved that there exists
$\delta\in(0,1/2]$, bounded away from zero as $\lambda\to 0$, such
  that the bulk diffusion coefficient $\Dbulk^N$ of the solution $h^N$
  of~\eqref{e:AKPZ1} regularised at level $N\in\mathbb N$ (see~\eqref{e:NonLin}
  for the precise form of the regularisation) satisfies 
\begin{equ}[e:SuperDiff]
(\log N)^\delta \,\lesssim\, \Dbulk^N(t)\,\lesssim\, (\log N)^{1-\delta}\,,\qquad \text{for $N$ large and $t$ fixed}
\end{equ}
%for large $t$, 
and we conjectured the exact value of $\delta$ to be $1/2$. See~\cite[Appendix B]{CET} for a heuristic and~\cite{CHT} 
for another system in the same universality class for which such conjecture was recently established up to lower order corrections. We emphasise that the limit $N\to\infty$ corresponds to removing the regularisation.

As the linear part of~\eqref{e:AKPZ1} is clearly diffusive, 
the origin of the observed logarithmically superdiffusive behaviour 
must lie in the slope dependence. In order to find a regime in which the nonlinearity 
has a non-trivial but not divergent effect, % the logarithmic divergence it causes and derive a scaling limit for~\eqref{e:AKPZ1}, 
it is therefore natural to tune 
the coupling constant $\lambda=\const$, which modulates the strength of the nonlinearity, 
together with the regularisation parameter $N$. 
In other words, the present paper is devoted to the study the AKPZ equation~\eqref{e:AKPZ1} 
under the {\it weak coupling scaling}, i.e. 
\begin{equ}\label{e:AKPZintro}
\partial_t h^N = \tfrac{1}{2} \Delta h^N
+
\const \Nonlin[h^N] + \xi\,,%\qquad h^N(0,\cdot)=\chi
\end{equ}
where $\Nonlin$ is a regularisation of the nonlinearity in~\eqref{e:AKPZ1} at level $N$ %has been smoothened with a Fourier cut-off 
%$\Pi_N$ that removes all the Fourier modes bigger than $N$,  
%\begin{equ}[e:NonLin]
%\Nonlin[h^N]\eqdef \Pi_N \Big((\Pi_N \partial_1 h^N)^2 - (\Pi_N \partial_2 h^N)^2\Big)
%\end{equ}
and the coupling constant $\const$ has been chosen, in light of~\eqref{e:SuperDiff} (together with the conjecture $\delta=1/2$), as 
\begin{equ}[e:const]
\const\eqdef \frac{\hat\lambda}{\sqrt{\log N}}
\end{equ}
for $\hat\lambda>0$. Our specific goal is twofold.

\begin{itemize}[ leftmargin=*]
\item In Theorem~\ref{thm:Dbulk}, we will prove that, 
as $N\to\infty$, the solution $h^N$ of~\eqref{e:AKPZintro} is {\it asymptotically diffusive} for {\it any} value 
of $\hat\lambda$. To do so, we will determine the large $N$ limit of the bulk diffusion coefficient 
and show that it converges to an explicit constant $\nueff=\nueff(\hat\lambda)$ (see~\eqref{e:nueff}). 
Let us point out that, {\it for every $\hat\lambda>0$}, the resulting $\nueff$ is strictly bigger than $1$, 
which in turn is the bulk diffusion coefficient 
of the SHE obtained by~\eqref{eq:KPZ} simply setting $\lambda=0$, 
and therefore the choice of~\eqref{e:const} 
is meaningful in that it does not wash away the nonlinearity $\Nonlin$ but only ``cures'' its divergence. 

\item In Theorem~\ref{thm:Conv}, we will derive the {\it full scaling limit} of~\eqref{e:AKPZintro} and 
show that, as $N\to\infty$, 
$h^N$ converges in law to the solution of the stochastic heat equation with {\it renormalised coefficients} 
given by 
\begin{equ}[e:SHE1]
\partial_t h=\tfrac{\nueff}{2}\Delta h + \sqrt{\nueff} \xi
\end{equ}
where $\nueff$ is the $\hat\lambda$-dependent constant determined above. This is saying that 
even though the nonlinearity is tuned down by a logarithmic factor, not only it does not vanish but 
actually it produces a new noise (and a new Laplacian) in the limit where the regularisation is removed.
\end{itemize}

The scaling regime in~\eqref{e:const} and the phenomenon observed above 
have already appeared for the Isotropic case in~\cite{CD,CSZ,Gu2020}. 
In all of these works, the matrix $Q$ is chosen to be the identity matrix
so that the nonlinearity in~\eqref{eq:KPZ} becomes $|\nabla h|^2$. 
We will refer to such equation as the Isotropic KPZ (IKPZ) equation. 
The work~\cite{CSZ} shows, under a different regularisation, that the weak coupling limit 
of IKPZ is given again by a SHE with renormalised coefficients 
but that this holds true {\it only} for $\hat\lambda$ smaller than a critical threshold $\hat\lambda_{\rm c}$ 
and the coefficients of the limiting 
SHE explode as $\hat\lambda\to\hat\lambda_{\rm c}$. 
The absence of such phase transition in the AKPZ case is already implied by~\cite{CES} 
(and it is further strengthened by the results in the present paper), 
thus providing a clear indication of the different nature of the Isotropic and Anisotropic settings 
and therefore supporting Wolf's prediction~\cite{W91}. 
\medskip

Let us emphasise that there is a major difference between our work and~\cite{CD,CSZ,Gu2020}
together with all those available in the supercritical regime. 
As mentioned above, the equation considered in therein is IKPZ, 
i.e.~\eqref{eq:KPZ} with $Q$ given by the identity matrix. Such an equation 
can be {\it linearised} via a nonlinear transformation, the so-called Cole-Hopf transform, 
which turns IKPZ into the linear stochastic heat equation with multiplicative noise. 
Once in possession of a linear SPDE, it is then possible to obtain an explicit 
representation of its solution thanks to the Feynman-Kac formula and therefore 
reduce the analysis of IKPZ to a problem of directed polymers in random environment. 
For the AKPZ equation there exists neither a transformation that linearises the equation 
nor (up to the authors' knowledge) an explicit representation for its solution, 
so that we need to resort to a completely alternative set of tools. 
What we can and will exploit is that, as shown in~\cite{CES},~\eqref{e:AKPZ} admits an invariant measure 
which is a Gaussian Free Field $\chi$. 
Our approach, partly inspired by~\cite{GPGen} (which though focuses on the {\it subcritical} KPZ equation), 
is based  on a thorough analysis of the action of the generator 
$\gen$ of the solution of~\eqref{e:AKPZintro} on the (infinite dimensional) 
space $L^2(\chi)$ of square integrable functions with respect to $\chi$. 
The idea is that the distribution of the solution of~\eqref{e:SHE1} is fully determined 
by a number of observables which is a (very small!) subset $S$ of $L^2(\chi)$. 
Loosely speaking, our goal is therefore to identify a ($N$-dependent) subset $S^N$
of $L^2(\chi)$ big enough to be able to characterise the fluctuations of $h^N$ and  
such that for each $b\in S$  
there exists $b^N\in S^N$ for which both $b$ is well-approximated by $b^N$ and 
the action of the generator of~\eqref{e:SHE1} on $b$ 
is well-approximated by $\gen b^N$. 
As we will see, while $S$ admits an easy description, the choice of $S^N$ is rather subtle 
and $S^N\cap S=\emptyset$! In other words, the structure of the $b^N$'s needs to be sufficiently rich 
to be able to capture the roughness of the nonlinearity which is encoded in $\gen$ (see Section~\ref{s:strategy} 
for a more detailed and precise explanation of the idea and its subtleties). 

Before delving into the proofs, in the next section we will rigorously introduce the 
quantities under study and state the main results.

\subsection{The equation  and the main result}
\label{Sec:1.1}

For $N\in\N$, let us begin by recalling the AKPZ equation regularised at level $N$ that is the object of the present work,
\begin{equ}\label{e:AKPZ}
\partial_t h^N = \frac{1}{2} \Delta h^N
+
\const \Nonlin[h^N] + \xi\,,\qquad h^N(0,\cdot)=\chi\,.
\end{equ}
Above $h^N=h^N(t,x)$, for $t\ge0,\, x\in \T^2$, the two-dimensional torus of size $2\pi$, 
$\chi$ is the initial condition, 
$\const$ is as in~\eqref{e:const} and the nonlinearity $\Nonlin$ is defined according to
\begin{equ}[e:NonLin]
\Nonlin[h^N]\eqdef \Pi_N \Big((\Pi_N \partial_1 h^N)^2 - (\Pi_N \partial_2 h^N)^2\Big)
\end{equ}
where $\Pi_N$ is the cut-off operator acting in Fourier space by removing the modes larger than $N$, 
i.e., for any function $w\colon \T^2\mapsto \R$,
\begin{equ}[eq:PiN]
\widehat{  \Pi_N w}(k)\eqdef \hat w(k)\mathds{1}_{|k|\le N}
\end{equ}
and $\hat w(k)$ is the $k$-th Fourier mode of $w$. 

Let us introduce the notion of solution for~\eqref{e:AKPZ} we will be working with throughout the 
rest the paper. 
%
%As mentioned above, in~\cite{CES} it was shown that the invariant measure of $h^N$ is the Gaussian Free Field 
%for all Fourier modes apart from the $0$-th. 
%Throughout the rest of the paper, we will be working at stationarity so that we now introduce the notion of 
%almost-stationary solution for~\eqref{e:AKPZ}. 

\begin{definition}\label{def:QSsol}
We say that $h^N$ is an almost-stationary solution to~\eqref{e:AKPZ} with coupling constant $\const$ if 
$h^N$ solves~\eqref{e:AKPZ} with initial condition $h(0,\cdot)\eqdef\chi$ and 
$\chi$ is such that $\eta\eqdef(-\Delta)^{1/2}\chi$ 
is distributed according to a spatial white noise on $\T^2$ (see~\eqref{e:fLapla} below for the definition of $(-\Delta)^{1/2}$), 
i.e. it is a mean zero Gaussian distribution on $\T^2$ such that 
\begin{equ}[e:CovS]
\E[\eta(\phi)\eta(\psi)]=\langle\phi,\psi\rangle_{L^2(\T^2)}\,,\qquad \text{for all $\phi,\,\psi\in L^2_0(\T^2)$}
\end{equ}
where $L^2_0(\T^2)$ is the space of zero-average square integrable functions on $\T^2$. 
\end{definition}
\begin{remark}
The reason why we call such solution ``almost-stationary'' (following the nomenclature of \cite{GonJara}) is that, 
as we will see (cf. Lemma \ref{lem:gen}), the law of $h^N$ with such initial condition is indeed stationary, 
except for the zero-mode $\frac1{2\pi}\int_{\mathbb T^2}h^N(t,x)\dd x$, 
whose law is time-dependent. 
Note also that the condition that $(-\Delta)^{1/2}\chi$ is a space white-noise does not impose any restriction 
on the zero-mode $\hat \chi(0)\eqdef\frac1{2\pi}\int_{\mathbb T^2}\chi(x)\dd x$ of the initial condition, 
that can be either deterministic or random. 
The fact that the almost-stationary solution exists and is unique (for fixed $N$) follows from soft arguments, 
see \cite{CES}.
\end{remark}

At first, we will show that the almost-stationary solution $h^N$ of~\eqref{e:AKPZ} 
is asymptotically diffusive for large $N$. 
To precisely formulate the corresponding statement, we need to introduce the {\it bulk diffusion coefficient} of $h^N$. 
Similar to \cite[eq. (1.6)]{CET}, let $\Dbulk$ be defined as
\begin{equ}[def:Dbulk]
\Dbulk^N(t)\eqdef 1 + 2\frac{\const^2}{t}\int_0^t\int_0^s \int_{\T^2}\Exp\Big[\Nonlin[h^N](r,x)\Nonlin[h^N](0,0)\Big]\dd x\dd r\dd s
\end{equ} 
where $\Nonlin$ is given by~\eqref{e:NonLin} and $\mathbf E$ is the expectation with respect 
to the law $\mathbf P$ of the process $h^N$. 
For future reference, we will instead denote by $\mathbb E$ the expectation 
taken with respect to the law $\mathbb P$ of its stationary measure.

The expression~\eqref{def:Dbulk} for the bulk diffusion coefficient is the analog of that
adopted in \cite{BQS} for the 1-dimensional KPZ equation and in \cite{spohn2012large} for interacting particle systems 
(see \cite[Appendix A]{CET} for a discussion). 

As mentioned above, recall that for the usual stochastic heat equation, corresponding to
$\lambda=0$ in~\eqref{eq:KPZ}, the bulk diffusion coefficient is identically equal to $1$.  
In the next theorem, we prove instead that 
$\Dbulk^N$ in~\eqref{def:Dbulk} converges to a time-independent
constant, which is however {\it strictly bigger than $1$} for all $\hat\lambda>0$.

\begin{theorem}\label{thm:Dbulk}
  For $N\in\N$ and $\hat\lambda>0$, let $h^N$ be the unique almost-stationary solution (see Definition~\ref{def:QSsol})
to~\eqref{e:AKPZ} with coupling constant 
$\const$ given by~\eqref{e:const}. 
For any $\hat\lambda>0$ the bulk diffusion coefficient $\Dbulk^N$ of $h^N$, defined as in~\eqref{e:Dbulk}, is such that 
for all $t\geq 0$
\begin{equ}[e:Dbulk]
\lim_{N\to\infty} \Dbulk^N(t)=\nueff
\end{equ}
where $\nueff$ is the effective diffusion coefficient 
\begin{equ}[e:nueff]
\nueff\eqdef \sqrt{2\fc(\hat\lambda)+1}\qquad\text{and}\qquad \fc=\fc(\hat\lambda)\eqdef \frac{\hat\lambda^2}{\pi}
\end{equ}
for any $\hat\lambda>0$.
%\begin{equ}[e:c]
%\fc=\fc(\hat\lambda)\eqdef \frac{\hat\lambda^2}{\pi}\,.
%\end{equ}
\end{theorem}

The proof of the previous theorem, provided in Section~\ref{sec:Dbulk} and Appendix~\ref{a:Approx}, 
is based on a significant refinement of the iterative scheme introduced in~\cite{Yau} and 
further developed in~\cite{CET}. Up to the authors' knowledge, this iterative procedure has so 
far only been used to derive upper and lower bounds  on the 
diffusivity or on any quantity that can be similarly analysed. In turn, this is the first time 
in which it is exploited to derive a full convergence statement. 

Heuristically, such a scheme identifies a sequence of operators 
which are expressed as nonlinear functions of the generator $\gensy$ of the stochastic heat equation
and are sufficiently close to the generator $\gen$ of $h^N$ (for $N$ large). 
As we will see, an interesting point is that the operators in the sequence 
do not converge {\it per se} to $\nueff\gensy$, generator of~\eqref{e:SHE1}, 
but they do so only when applied to suitable observables 
(see Proposition~\ref{p:Hj}). 
\medskip

The previous theorem strongly suggests that in the large $N$ limit, 
the solution $h^N$ converges to a stochastic heat equation 
and explicitly identifies its expected renormalised limiting coefficients.
The next is the main result of the present paper and establishes the scaling limit of~\eqref{e:AKPZ} 
under the weak coupling regime. 

\begin{theorem}\label{thm:Conv}
% For $N\in\N$ and $\hat\lambda>0$, let $h^N$ be the unique almost-stationary solution (see Definition~\eqref{def:QSsol})
% to~\eqref{e:AKPZ} with coupling constant 
% $\const$ given by~\eqref{e:const} and initial condition $\chi$.
Let $T>0$. 
Then, the almost-stationary solution $h^N$ of~\eqref{e:AKPZ} under the weak coupling scaling 
converges in distribution in $C([0,T],\cD'(\T^2))$ to $h$ which solves 
\begin{equ}[e:SHEfinal]
\partial_t h=\tfrac{\nueff}{2}\Delta h + \sqrt{\nueff} \xi\,,\qquad h_0=\chi\,.
\end{equ}
\end{theorem}

\begin{remark}[Extensions]
With little extra effort and the same techniques developed below, 
the previous results can be translated to the full $\R^2$. 
%
%Via an easy scaling argument (also used in~\cite{CET}), it is clearly possible to translate our results 
%to the full plane $\R^2$. In this case, one would start from the equation regularised at scale $N=1$, 
%then rescale space and time diffusively ($t\mapsto t/\eps^2$, $x\mapsto x/\eps$) and simultaneously 
%tune $\lambda=\lambda_\eps=\hat\lambda/\sqrt{\log \eps^{-1}}$ as in~\eqref{e:const}

Further, exploiting the estimates in~\cite{GPabs,GPGen}, it would be possible to weaken the 
(almost) stationarity assumption and request instead the initial condition to be absolutely 
continuous with respect to $\chi$ in Definition~\ref{def:QSsol}. 
\end{remark}

%\begin{remark}
%The proof of the previous Theorem could be simplified if we omitted the 
%study of the $0$-th Fourier mode. Indeed, if 
%\end{remark}

Let us conclude this introduction by pointing out that the present paper 
(and especially the previous Theorem) represents the {\it first}
instance of a {\it critical} equation, out of the scope of the theory
of Regularity Structures~\cite{Hai} and with no Cole-Hopf
transformation, for which a full scaling limit (in the weak coupling
regime), corresponding to a sharp renormalisation-type statement, is obtained.  
Further, the techniques we adopt, based on the
analysis of the generator and martingale arguments, are bound to be
applicable to other SPDEs at criticality, e.g. the viscous stochastic Navier-Stokes
equation in~\cite{GT, CK}, or to statistical mechanics models, e.g. ASEP
in two dimensions with weak asymmetry (similarly to the AKPZ equation,  $2$-dimensional ASEP is
  logarithmically super-diffusive if the asymmetry is of order $1$
  \cite{Yau}), weak coupling version of diffusions in the curl of GFF
considered in~\cite{CHT}, and many others.

\subsection*{Organization of the article}
The rest of this work is organised as follows. 
Below we introduce the notations and conventions which will be used throughout. 
In Section~\ref{sec:prelim}, we recall some results from previous works, 
introduce the Burger's version of our equation, and prove some preliminary estimates. 
Section~\ref{sec:Dbulk} is devoted to the proof of Theorem~\ref{thm:Dbulk}, 
and, along the way, we derive a replacement lemma, 
which basically allows one to replace the generator of~\eqref{e:AKPZ} 
by the appropriately modulated (in Fourier space) generator of the stochastic heat equation. 
In Sections~\ref{sec:truncated} and~\ref{s:theconvergence}, we prove Theorem~\ref{thm:Conv}. 
First, we introduce the martingale problem and provide an outline of the proof. 
Then, we carry out a thorough analysis of the generator $\gen$ of $h^N$ and identify the 
class of observables needed to establish the statement, 
whose proof is completed at the beginning Section~\ref{s:theconvergence}. 
In the rest of the section, we show that the martingales associated to the observables 
determined in the Section~\ref{sec:truncated} have quadratic variation which is asymptotically deterministic. 
At last, in the appendix we obtain some technical results needed for the proof of the replacement lemma 
(Appendix~\ref{a:Approx}), control of the quadratic variation (Appendix~\ref{app:S}) and provide 
an example of the estimates to perform in Section~\ref{s:theconvergence} (Appendix~\ref{sec:example}). 

\subsection*{Notations and Function Spaces}

%For $N>0$, let $\Z_N\eqdef \Z/N$ and 
We let $\T^2$ be the two-dimensional torus of side length $2\pi $. % If $N=1$ then we simply write $\T^2$ instead of $\T_N^2$.
We denote by $\{e_k\}_{k\in\Z^2}$ the Fourier basis defined via 
$e_k(x) \eqdef \frac{1}{2\pi} e^{i k \cdot x}$ which, for all $j,\,k\in\Z^2$, satisfies 
$\langle e_k, e_{-j}\rangle_{L^2(\T^2)}= \mathds{1}_{k=j}  $. 
% {\color{red} I think the following sentence can be omitted: The basis functions $e_k$ can be decomposed in their real and imaginary part, so that $e_k=a_k+i b_k$ 
% and the system $\{a_k\}_{k \in \Z_{\mathrm{diag}}^2} \sqcup \{b_k\}_{k \in \Z_{\mathrm{diag}}^2\setminus\{0\}}$ 
% forms a real valued orthogonal basis of $L^2(\T_L^2)$, where $\Z_{\mathrm{diag}}^2= \{(k_1,k_2)\in\Z^2:\, k_1\geq k_2\}$. }

The Fourier transform of a given function $\phi\in L^2(\T^2)$ will be represented as 
$\cF(\phi)$ or by $\hat\phi$ and, for $k\in\Z^2$
is given by the formula
	\begin{equation}\label{e:FT}
	\cF(\varphi)(k) =\hat\varphi(k)\eqdef  \int_{\T^2} \varphi(x) e_{-k}(x)\dd x\,, 
	\end{equation}
so that in particular
\begin{equ}[e:FourierRep]
\varphi(x) = % \frac{1}{N^2}
\sum_{k\in\Z^2} \hat\varphi(k) e_k(x)\,,\qquad\text{for all $x\in\T^2$. }
\end{equ}
Let $\cD(\T^2)$ be the space of smooth functions on $\T^2$ and $\cD'(\T^2)$ 
the space of real-valued distributions given by the dual of $\cD(\T^2)$. 
For any $\eta\in\cD'(\T^2)$ and $k\in\Z^2$, 
we will denote its Fourier transform by 
\begin{equation}\label{e:complexPairing}
\hat \eta(k)\eqdef \eta(e_{-k})\,.%=\eta(a_k)-i \eta(b_k)
\end{equation} 
Note that $\overline{\hat\eta(k)}=\hat\eta(-k)$. 
Moreover, we recall that the Laplacian $\Delta$ on $\T^2$ has eigenfunctions $\{e_k\}_{k \in \Z^2}$ 
with eigenvalues $\{-|k|^2\,:\,k\in\Z^2\}$, so that, for $\theta>0$, we can define the operator $(-\Delta)^\theta$
by its action on the basis elements 
	\begin{equation}\label{e:fLapla}
	(-\Delta)^\theta e_k\eqdef |k|^{2\theta}e_k\,,\qquad k\in\Z^2\,.
	\end{equation}	
 In particular, $(-\Delta)^\theta$ is an invertible linear bijection on distributions with null $0$-th Fourier mode. 
This allows us to make the following observation:
given $\Psi\in\cD(\T^2)$, set
\begin{equs}[e:Psi]
\psi_0\eqdef \hat\Psi(0)\,,\qquad \tilde\Psi(x)\eqdef \Psi(x)-\psi_0 e_0(x) \,,\qquad \psi(x)\eqdef (-\Delta)^{-1/2} \tilde\Psi (x)
\end{equs}
for all $x\in\T^2$. Since $\psi\in \cD(\T^2)$, in particular $\psi\in H^1_0(\T^2)$, 
the space of mean zero functions such that
\[
\|\psi\|_{H^1(\T^2)}^2\eqdef \sum_{k\in\Z^2}|k|^2|\hat\psi(k)|^2<\infty\,,
\]  %\giuseppeText{define the $H^1_0(\T^2)$ space and norm...}
and the following equality holds
\begin{equ}[e:L2H1]
\|\Psi\|^2_{L^2(\T^2)}=|\psi_0|^2 + \|\psi\|_{H^1(\T^2)}^2\,.
\end{equ}
 Now, if $h\in\cD'(\T^2)$ and $u\eqdef (-\Delta)^{1/2}h\in\cD'(\T^2)$, then
\begin{equ}[e:hu]
h(\Psi)= u(\psi)+\psi_0 \hat h(0)
\end{equ}
which ensures that $h$ is fully characterised by $u$ and its spatial
average given by $\hat h(0)$.  
\medskip

Throughout the paper, we will write $a\lesssim b$ if there exists a constant $C>0$ such that $a\leq C b$ and $a\sim b$ if $a\lesssim b$ and $b\lesssim a$. We will adopt the previous notations only in case in which the hidden constants do not depend on any quantity which is relevant for the result. When we write $\lesssim_T$ for some quantity $T$, it means that the constant $C$ implicit in the bound depends on $T$.

\section{Preliminaries}
\label{sec:prelim}

As in~\cite{CES}, it turns out to be convenient to work with 
the stationary Stochastic Burgers equation, which can be derived from the almost-stationary 
AKPZ~\eqref{e:AKPZ} started from $\chi$ (see Definition~\ref{def:QSsol}) by setting 
$u^N\eqdef(-\Delta)^{\frac{1}{2}}h^N$ 
so that $u^N$ solves 
\begin{equation}\label{e:AKPZ:u}
\partial_t u^N = \tfrac{1}{2} \Delta u^N
+
\const \cM^N[u^N] + (-\Delta)^{\frac{1}{2}}\xi, \quad u^N(0)=\eta\eqdef(-\Delta)^{\frac12}\chi
\end{equation}
where $u^N=u^N(t,x)$, $t\geq 0$, $x\in\T^2$, $\const$ is as in~\eqref{e:const}
and the nonlinearity $\cM^N$ is 
\begin{equation}\label{e:nonlin}
\cM^N[u^N]\eqdef (-\Delta)^{\frac{1}{2}}\Pi_N \Big((\Pi_N \partial_1(-\Delta)^{-\frac{1}{2}} u^N)^2 - (\Pi_N \partial_2 (-\Delta)^{-\frac{1}{2}} u^N)^2\Big)\,.
\end{equation}
Before proceeding further, in the next section we recall  basic properties of the 
zero-average spatial white noise $\eta$ on $\mathbb T^2$ which will be crucial in our subsequent analysis
(for more on it, we refer to~\cite[Chapter 1]{Nualart2006}, or~\cite{GPnotes,GPGen} and~\cite[Section 2]{CES}). 

\subsection{Elements of Wiener space analysis}
\label{S:Malliavin}

Let $(\Omega,\cF,\P)$ be a complete probability space and 
$\eta$ be a mean-zero spatial white noise on the two-dimensional torus $\T^2$, i.e. 
$\eta$, defined in $\Omega$, is a centred isonormal Gaussian process 
(see~\cite[Definition 1.1.1]{Nualart2006}) on $H\eqdef L^2_0(\T^2)$, 
the space of square-integrable functions with $0$ total mass,
whose covariance function is given by 
\begin{equ}\label{eq:spatial:white:noise}
\E[\swn(\vphi)\swn(\psi)]=\langle \vphi, \psi\rangle_{L^2(\T^2)}
\end{equ}
where $\varphi,\psi\in H$ and $\langle\cdot,\cdot\rangle_{L^2(\T^2)}$ is the usual scalar product in $L^2(\T^2)$. 
For $n\in\N$, let $\SH_n$ be the {\it $n$-th homogeneous Wiener chaos}, i.e. the closed linear subspace of 
$L^2(\eta)\eqdef L^2(\Omega)$ generated by the random variables $H_n(\eta(h))$, 
where $H_n$ is the $n$-th Hermite polynomial, and 
$h\in H$ has norm $1$. By~\cite[Theorem 1.1.1]{Nualart2006}, $\wc_n$ and $\wc_m$ are orthogonal whenever 
$m\neq n$ and $L^2(\eta)=\bigoplus_{n}\SH_n$. 
Moreover, there exists a canonical contraction $\sint{}:\bigoplus_{n\ge 0} L^2(\T^{2n}) \to L^2(\eta)$, 
which restricts to an isomorphism $\sint{}:\fock \to L^2(\eta)$ on the Fock space $\fock:=\bigoplus_{n \ge 0} \fock_n$, 
where $\fock_n$ denotes the space $L_\sym^2(\T^{2n})$ of functions in $L^2_0(\T^{2n})$ which are symmetric
with respect to permutation of variables. The restriction $\sint_n$ of $\sint$ to $\fock_n$, 
called $n$-th (iterated) Wiener-It\^o integral with respect to $\eta$, is itself an isomorphism from $\fock_n$ to $\SH_n$
so that by~\cite[Theorem 1.1.2]{Nualart2006}, for every $F\in L^2(\eta)$ there exists a family of kernels 
$(f_n)_{n\in\N}\in\fock$ such that $F=\sum_{n\geq 0} I_n(f_n)$ and 
\begin{equ}[e:Isometry]
\E[F^2]\eqdef\|F\|^2 = \sum_{n\geq 0} n!\|f_n\|_{L^2(\T^{2n})}^2
\end{equ}
and we take the right hand side as the definition of the scalar product on $\fock$, i.e. 
\begin{equ}[e:ScalarPFock]
 \langle f,g\rangle=\sum_{n\geq 0} \langle f_n,g_n\rangle \eqdef \sum_{n\geq 0} n!\langle f_n,g_n\rangle_{L^2(\T^{2n})}.
\end{equ}

\begin{remark}\label{rem:Ident}
In view of the isometry~\eqref{e:Isometry}, we will abuse notations throughout the rest of the paper and 
implicitly identify a random variable in $\SH_n$  with its kernel $f_n$ in $\fock_n$. 
% In the same spirit, we will identify linear operators acting on $L^2(\eta)$ with the corresponding
% linear operators acting on $\bigoplus_n L^2_{sym}(\mathbb R^{2n})$, and
% we will denote them using the same symbol.
\end{remark}

We conclude this paragraph by mentioning that we will mainly work with the Fourier representation $\{\hat \eta(k)\}_k$ of 
$\eta$, which is a family of complex valued, centred Gaussian random variables such that 
\begin{equs}[e:NoiseFourier]
\hat \eta(0)=0\,,\qquad \overline{\hat \eta(k)}=\hat\eta({-k})\qquad\text{and}\qquad \E[\hat \eta(k)\hat \eta(j)]=\1_{k+j=0}\,.
\end{equs}
Given a random variable $f=f(\eta)$ in $\SH_n$ whose kernel $f_n$ has a Fourier transform $\hat f_n$, we write
\begin{equs}
  f(\eta)=\sum_{p_1,\dots,p_n\in\mathbb Z^2}\hat f_n(p_1,\dots,p_n):\hat\eta(p_1)\dots\hat \eta(p_n):
\end{equs}
where $:\dots:$ denotes the Wick product. A few basic properties of Wick polynomials (in the specific case where $\eta$ is a white noise) are collected at the beginning of Section~\ref{sec:case>} (see \cite{janson1997gaussian} for additional details).

\subsection{Stochastic Burgers equation and its Generator}
\label{S:Properties}

In the present section, we summarise the main results, properties and notations concerning 
the solutions $h^N$ and $u^N$ of~\eqref{e:AKPZ} 
and~\eqref{e:AKPZ:u} respectively, proven in~\cite{CES}. 
% For a complete overview and the proof of the statements, we address the reader to the above mentioned reference. 

As can be directly checked, in Fourier variables,~\eqref{e:AKPZ:u} corresponds to an infinite system of 
(complex-valued) SDEs whose $k$-th component, $k\in\Z^2\setminus\{0\}$, reads
\begin{equation}\label{e:kpz:u}
\dd \hat u^{N}(k) =
\Big(-\tfrac{1}{2}|k|^2 \hat u^{N}(k)
+
 \const\cM^N_k[u^{N}]\Big) \dd t +  |k|\dd B_t(k)\,,\quad \hat u^N_0(k)=\hat\eta(k)
\end{equation}
where $\eta$ is as in~\eqref{e:NoiseFourier}, the $B(k)$'s are complex valued Brownian motions defined via  
$B_t(k)\eqdef \int_0^t \hat \xi(s,k)\, ds$, $\hat \xi(\cdot,k)=\xi(\cdot,e_{-k})$, so that %(recalling that $\xi$ is a space-time white noise)
$\overline{B(k)}= B(-k)$ and $\dd\langle B(k), B(\ell) \rangle_t =  \mathds{1}_{\{k+\ell=0\}}\, \dd t$, and 
$\cM^N_k$ is the $k$-th Fourier component of the nonlinearity in~\eqref{e:nonlin}. 
More explicitly, the latter is given by
\begin{equation}\label{e:nonlinF}
\cM_k^N[u^{N}]\eqdef\cM^N[u^{N}](e_{-k})= |k|\sum_{\ell+m=k}\nonlin_{\ell,m}\hat  u^{N}(\ell) \hat u^{N}(m)\,,
\end{equation}
the sum running over $\ell,m\in \Z^2\setminus\{0\}$, and, for $x,y,z,w\in \R^2\setminus\{0\}$
\begin{align}
 \nonlin_{x,y} &\eqdef \frac{1}{2\pi} \frac{c(x,y)}{|x||y|} \indN{x/N,y/N}\,,\qquad c(x,y) \eqdef x_2y_2- x_1y_1 \label{e:nonlinCoefficient}\\
 \indN{z,w}&\eqdef
\mathds{1}_{1/N \le |z|\le 1, 1/N\le|w|\le 1, |z+w|\le 1}\,. \label{e:Jlm}
\end{align}

Since eventually we are interested in $h^{N}$ rather than $u^{N}$, we note that by~\eqref{e:hu} 
we can recover all the non-zero Fourier components of $h^{N}$ while  
the zero-mode $\hat h^{N}(0)$ satisfies 
\begin{equ}[eq:modozero]
\dd \hat h^{N}(0)= \const \cN^N_0[u^N]\dd t+ \dd B_t(0)\,,\qquad  \hat h^{N}_0(0)=\hat\chi(0)
\end{equ}
where $\chi$ is as in Definition~\ref{def:QSsol} and
\begin{equ}[e:modozeroN]
  \cN^N_0[u^N]= \sum_{\ell+m=0}\nonlin_{\ell,m}\hat u^{N}(\ell)\hat u^{N}(m)\,,
\end{equ}
so that also $\hat h^N(0)$ is a function of $u^{N}$ (and of an independent Brownian motion $B_\cdot(0)$). 
\medskip 

By \cite[Proposition 3.4]{CES}, for any $N\in\N$, there exists a unique global in time solution 
$t\mapsto \hat u^{N}(t)=\{\hat u^{N}(t,k)\}_{k\in\Z^2\setminus\{0\}}$ to~\eqref{e:kpz:u} (and consequently 
the same holds for~\eqref{e:AKPZ}), 
and such a solution is a strong Markov process, whose law is translation invariant and 
whose generator will be denoted by $\gen$. 
Let $F$ be a cylinder function on $\CD'(\T^2)$, i.e. $F$ is such that such that there exists a polynomial
$f=f((x_k)_{k\in\Z^2\setminus\{0\}})$ (depending 
only on finitely many variables), 
for which $F(\eta)=f((\hat \eta(k))_{k\in\Z^2\setminus\{0\}})$. 
Then, writing $\gen\eqdef\gensy+\gena$, the action of $\gensy$ and $\gena$ on $F$ is 
\begin{align}
(\gensy F)(v) &\eqdef \sum_{k \in \Z^2} \half|k|^2 (-\hat v({-k}) D_k  +  D_{-k}D_k )F(v) \label{e:gens}\\
(\gena F)(v) &= \const \sum_{m,\ell \in \Z^2\setminus\{0\}} |m+\ell|\nonlin_{m,\ell} \hat v(m) \hat v(\ell) D_{-m-\ell} F(v)\,. \label{e:gena}
\end{align}
Here, for $k\in\Z^2$, $D_k F$ is defined as\footnote{For more on the actual definition of cylinder function and Malliavin derivative, see~\cite[Section 2 and Lemma 2.1]{CES}}
\begin{equ}[e:Malliavin]
D_{k} F\eqdef (\partial_{x_k} f)((\hat \eta({k}))_{k\in\Z^2\setminus\{0\}})\,.
\end{equ}

\begin{lemma}\textnormal{\cite[Lemma 3.1]{CES}}\label{lem:gen}
For every $N\in\N$, the spatial white noise $\eta$ on $\T^2$ defined in \eqref{eq:spatial:white:noise} 
is an invariant measure for the solution $\hat u^{N}$ of~\eqref{e:kpz:u} and, with respect to $\eta$, 
$\gensy$ and $\gena$ are respectively 
the symmetric and anti-symmetric part of $\gen$. 
\end{lemma}

Our goal  is to determine the limit of $h^N$ as $N\to\infty$. 
A first result in this direction was obtained in~\cite{CES}, and it ensures that the sequence $\{h^N\}_N$ is 
tight. 

\begin{theorem}\textnormal{\cite[Theorem 1.1]{CES}}\label{thm:Tightness}
% For $N\in\N$ and $\hat\lambda>0$, let $h^N$ be the unique almost-stationary solution to~\eqref{e:AKPZ} with 
% $\const$ given by~\eqref{e:const} and initial condition $\chi$. 
% Then, f
For every $\hat\lambda>0$, the sequence $\{h^N\}_N$ is tight in $C([0,T],\CD'(\T^2))$. 
\end{theorem}

\begin{remark}
 Theorem~\cite[Theorem 1.1]{CES} actually shows more, namely that for the choice of $\const$ in~\eqref{e:const}, 
$\{h^N\}_N$ is tight in $C([0,T],\CC^\alpha(\T^2))$ for any $\alpha<0$, where $\CC^\alpha$ is the 
Besov-H\"older space of distribution defined in~\cite[eq. (1.11)]{CES}. This is the optimal regularity 
one can expect for the solution to~\eqref{e:AKPZ}. We will not need this fact here. % but this fact will play no role in the subsequent analysis. 
%we will not use this fact in what follows. 
\end{remark}

\subsection{The action of Burgers' generator on $L^2(\eta)$}

In this section, we want to deepen our knowledge of the generator $\gen$ and understand how 
it acts on random variables in $L^2(\eta)$. 
To lighten the exposition, we will use the following convention throughout the rest of the paper. 
In line with Remark~\ref{rem:Ident}, we will identify linear operators acting on 
$L^2(\eta)$ with the corresponding linear operators acting on $\bigoplus_n L^2_{\sym}(\T^{2n})$, 
and we will denote them using the same symbol. 

\begin{lemma}\textnormal{\cite[Lemma 3.5]{CES}}\label{lem:generator}
For $N\in\N$, let $\gensy$ and $\gena$ be defined according to~\eqref{e:gens} and~\eqref{e:gena} respectively. 
For all $n \in \N$ the operator $\gensy$ is such that $\gensy(\wc_n)\subset\wc_n$ and 
on $\fock$ it satisfies $\gensy=\half\Delta$, i.e. for $\phi_n\in\wc_n$ we have
\begin{equ}[e:gens:fock]
\CF(\gensy \varphi_n) (k_{1:n}) = -\tfrac{1}{2} |k_{1:n}|^2 \hat\phi_n(k_{1:n}). \label{e:gens:fock}
\end{equ}
The operator $\gena$ instead can be written as the sum of two operators $\genap$ and $\genam$ which, 
for every $n\in\N$, map $\wc_n$ respectively into $\wc_{n+1}$ and $\wc_{n-1}$ and 
are such that $-\genap$ is the adjoint of $\genam$. Further, for $\phi_n\in\wc_n$, 
$\genap$ and $\genam$ are given by
\begin{align}
\CF(\genap \phi_n)(k_{1:n+1}) 
&= n \const 
|k_1+k_2|\nonlin_{k_1, k_2} \hat\phi_n(k_1 + k_2, k_{3:n+1})  \label{e:genap:fock}
\\
\CF(\genam \phi_n)(k_{1:n-1})
&=
2n(n-1) \const
\sum_{\ell+m = k_1}|m|
\nonlin_{k_1, -\ell} \hat \phi_n(\ell,m, k_{2:n-1}), \label{e:genam:fock}
\end{align}
where all the variables belong to $\Z^2\setminus\{0\}$ and 
we adopted the short-hand notations $k_{1:n}\eqdef (k_1,k_2,\dots,k_n)$ and 
$|k_{1:n}|^2\eqdef |k_1|^2+\dots+|k_n|^2$. 
Strictly speaking the functions on the right hand side of~\eqref{e:genap:fock} and \eqref{e:genam:fock} 
need to be symmetrised with respect to all permutations of their arguments.
\end{lemma}

Before proceeding, we recall one of the main tools in the 
proof of Theorem~\ref{thm:Tightness} which is going to be useful also in our analysis, namely 
the so-called It\^o trick. First devised in~\cite{GubinelliJara2012} and since then exploited in a number 
of references, we provide its statement below in the form that will be used later on.  

\begin{lemma}\label{l:ItoTrick}
For $N\in\N$, let $u^N$ be the stationary solution to~\eqref{e:AKPZ:u}. Then, 
for any $p\geq 2$, $T>0$ and $F\in L^2(\eta)$ with finite chaos expansion, the following estimate holds 
\begin{equ}[e:ItoTrick]
\Exp\Big[\sup_{t\in[0,T]}\left|\int_0^t F(u^N_s)\dd s\right|^p\Big]^{1/p}\lesssim_p T^{1/2}\|(-\gensy)^{-1/2}F\|\,.
\end{equ}
\end{lemma}
\begin{proof}
In view of Lemma~\ref{lem:generator}, $(-\gensy)$ is strictly positive and invertible on $\fock$, 
hence, for $F$ as in the statement, $G\eqdef (-\gensy)^{-1} F$ is well-defined. 
Applying~\cite[Lemma 4.1, eq. (4.7)]{CES} to $G$ with 
$\nu_N\equiv 1$, we get 
\begin{equs}[e:ItoStep1]
\Exp\Big[\sup_{t\in[0,T]}\left|\int_0^t F(u^N_s)\dd s\right|^p\Big]^{\frac{1}{p}}&=\Exp\Big[\sup_{t\in[0,T]}\left|\int_0^t \gensy G(u^N_s)\dd s\right|^p\Big]^{\frac{1}{p}}\\
&\lesssim T^{1/2} \E\Big[|\energy(G)|^{\frac{p}{2}}\Big]^{\frac{1}{p}}\lesssim_p T^{1/2} \E\Big[\energy(G)\Big]^{\frac{1}{2}}
\end{equs}
where the energy functional $\energy$ is defined as 
\begin{equ}[e:Energy]
\energy(G)=\sum_{k\in\Z^2}|k|^2|D_kG|^2
\end{equ}
with $D_k$ as in~\eqref{e:Malliavin}, and in the last inequality we used Gaussian 
hypercontractivity~\cite[Theorem 1.4.1]{Nualart2006} which holds 
as $F$, and therefore $G$, has a finite chaos expansion. 
Now, in order to simplify the right hand side of~\eqref{e:ItoStep1}, we apply 
the integration by parts formula~\cite[eq. (2.11)]{CES}, so that we obtain
\begin{equs}[e:EnergyMart]
\E\Big[\energy(G)\Big]&= %\sum_{k\in\Z^2} |k|^2 \E[D_k G D_{-k} G]\\
\sum_{k\in\Z^2} |k|^2 \E\big[\big(-D_{-k} D_{k} G(\eta) +\hat\eta(-k) D_{k} G(\eta)\big)\,G(\eta)\big]\\
&=2\E\Big[\Big(\sum_{k\in\Z^2} \frac{1}{2}|k|^2 (\big(-D_{-k} D_{k} G(\eta) +\hat\eta(-k) D_{k} G(\eta)\big)\Big) \,G(\eta)\Big]\\
&=2\E\Big[G(\eta)(-\gensy)G(\eta) \Big]=2\|(-\gensy)^{1/2} G\|^2
\end{equs}
where we used also the definition of $\gensy$ in~\eqref{e:gens}. 
Since we chose $G=(-\gensy)^{-1} F$,~\eqref{e:ItoTrick} follows at once. 
\end{proof}

According to Lemma~\ref{lem:generator}, $\gensy$ is a {\it diagonal operator} - for every $n$, 
it maps $\wc_n$ into itself and it acts in Fourier space as multiplication by a Fourier multiplier. 
In the following statement, $\cS$ will denote a generic diagonal operator with 
multiplier $\sigma$, $\sigma$ corresponding to the collection $(\sigma_n)_n$, $\sigma_n$ being the multiplier on $\wc_n$. 

\begin{lemma}\textnormal{\cite[Section 3.1.1]{CET}}\label{l:DiagOffDiag}
Let $\cS$ be a diagonal operator with Fourier multiplier $\sigma$. 
Then, for every $\phi_1,\,\phi_2\in\wc_n$, the following decomposition holds
\begin{equ}
\langle \cS\genap\phi_1,\genap\phi_2\rangle=\langle \cS\genap\phi_1,\genap\phi_2\rangle_{\Di} +\sum_{i=1}^2 \langle \cS\genap\phi_1,\genap\phi_2\rangle_{\oD_i}
\end{equ}
where the first summand will be referred to as the ``diagonal part'' and is given by 
\begin{equs}
\langle \cS\genap\phi_1,\genap\phi_2\rangle_{\Di}\eqdef& n!\, n \,2 \const^2 \times\label{e:Diag}\\
& \times\sum_{k_{1:n}}|k_1|^2\hat\phi_1(k_{1:n})\overline{\hat\phi_2(k_{1:n})}\sum_{\ell+m=k_1}\sigma_{n+1}(\ell,m,k_{2:n}){(\nonlin_{\ell,m})^2}
\end{equs}
while the other two terms will be referred to as the ``off-diagonal part of type $1$ and $2$'' and are 
respectively given by 
\begin{equs}
\langle \cS\genap\phi_1,\genap\phi_2&\rangle_{\oDi}\eqdef n! \,c_{\oDi}(n)\,\const^2\times\label{e:OffDiag1}\\
&\times\sum_{k_{1:n+1}} \sigma_{n+1}(k_{1:n+1})   |k_1+k_2|\nonlin_{k_1,k_2}|k_1+k_3|\nonlin_{k_1,k_3}\times\\
&\qquad\qquad\times
\hat{\phi}_1(k_1+k_2,k_3,k_{4:n+1})\overline{\hat{\phi}_2(k_1+k_3,k_2,k_{4:n+1})}\,,\\
\langle \cS\genap\phi_1,\genap\phi_2&\rangle_{\ooDi}\eqdef n! \,c_{\ooDi}(n)\,\const^2\times\label{e:OffDiag2}\\
&\times \sum_{k_{1:n+1}} \sigma_{n+1}(k_{1:n+1})   |k_1+k_2|\nonlin_{k_1,k_2}|k_3+k_4|\nonlin_{k_3,k_4} \times\\
&\qquad\qquad\times\hat{\phi}_1(k_1+k_2,k_3,k_4,k_{5:n+1})\overline{\hat{\phi}_2(k_3+k_4,k_1,k_2,k_{5:n+1})}
\end{equs}
where, for $i=1,\,2$, $c_{\oD_i}(n)$ is an explicit positive constant only depending on $n$ and such that $c_{\oD_i}(n)= O(n^{i})$. 
\end{lemma}
\begin{proof}
See~\cite[eq.'s (3.28) and (3.34)]{CET} for~\eqref{e:Diag},~\cite[eq.'s (3.29)-(3.30)]{CET} for~\eqref{e:OffDiag1}-\eqref{e:OffDiag2} and~\cite[eq. (3.31)]{CET} for the magnitude of $c_{\oD_i}(n)$, $i=1,\,2$. 
\end{proof}

Before stating the next lemma, let us introduce the so-called number operator. 

\begin{definition}\label{def:NoOp}
We define the {\it number operator} $\cN\colon \fock\to\fock$ as the linear diagonal operator acting on $\phi\in\fock$ 
as $(\cN\phi)_n\eqdef n \phi_n$. 
\end{definition}

In force of the results and notations above, we are ready to derive
the first bounds on the operators $\genap$ and 
$\genam$ when the operator $\cS$ in Lemma~\ref{l:DiagOffDiag} is $\gensy$. 

\begin{lemma}\label{l:A+A-}
Let $\cN$ be the number operator of Definition~\ref{def:NoOp}. Then, 
for every $\mu\ge0$ and $\phi\in \Gamma L^2$ one has
\begin{equs}
\|(\mu-\gensy)^{-1/2}\genap \phi\|&\lesssim \|\cN(-\gensy)^{1/2}\phi\|\,,\label{b:A+}\\
\|(\mu-\gensy)^{-1/2}\genam \phi\|&\lesssim \|\cN(-\gensy)^{1/2}\phi\|\,,\label{b:A-}
\end{equs}
while 
\begin{equ}
\|(\mu-\gensy)^{-1}\genap \phi\|\lesssim \frac{1}{\sqrt{\log N}}\|\cN\phi\|\label{b:LA+}\,.
\end{equ}
\end{lemma}
\begin{proof}
Let $\phi=(\phi_n)_n\in \Gamma L^2$. The square of  the left hand side of~\eqref{b:A+} is
  \begin{equ}
   \|(\mu-\gensy)^{-1/2}\genap \phi\|^2= \langle (\mu-\gensy)^{-1}\genap \phi, \genap \phi\rangle=\sum_n\langle (\mu-\gensy)^{-1}\genap \phi_n, \genap \phi_n\rangle\,.
 \end{equ}
  Thanks to Lemma~\ref{l:DiagOffDiag}, for each $n$,  we can decompose the scalar product at the right hand side 
  into diagonal and
  off-diagonal terms, defined according to~\eqref{e:Diag}--\eqref{e:OffDiag2}. For
  our purposes, we will just upper bound the off-diagonal terms by the
  diagonal ones, using the Cauchy-Schwarz inequality and the bounds $c_{\oD_i}(n)\lesssim n^2$, $i=1,\,2$. 
  Therefore we get
  \begin{equs}
  \langle &(\mu-\gensy)^{-1}\genap \phi_n, \genap \phi_n\rangle\lesssim n^2 \langle (\mu-\gensy)^{-1}\genap \phi_n, \genap \phi_n\rangle_{\Di}\\
  &= n!\, n^3\,\frac{2\hat\lambda^2}{\log N} \sum_{k_{1:n}}|k_1|^2|\hat\phi_n(k_{1:n})|^2\sum_{\ell+m=k_1}\frac{(\cK^N_{\ell,m})^2}{\mu+\tfrac12(|\ell|^2+|m|^2+|k_{2:n}|^2)}\,.
  \end{equs}
  Now, the inner sum can be controlled via  
  \begin{equs}
    \label{daillog}
    \sum_{\ell+m=k_1}\frac{(\cK^N_{\ell,m})^2}{\mu+\tfrac12 (|\ell|^2+|m|^2+|k_{2:n}|^2)}
    \lesssim\sum_{0<|\ell|\le N}\frac1{|\ell|^2}\lesssim \log N
  \end{equs}
  since, by~\eqref{e:nonlinCoefficient} $|\cK^N_{\ell,m}|\lesssim \mathds{1}_{1\le |\ell|\le N, 1\le|m|\le N, |\ell+m|\le N}$. Therefore, we obtain 
  \begin{equs} 
  \|(\mu-\gensy)^{-1/2}\genap \phi\|^2&\lesssim \sum_n n! \, n^3 \sum_{k_{1:n}}\tfrac12 |k_1|^2|\hat\phi_n(k_{1:n})|^2\\
  &= \sum_n n!\, n^2\sum_{k_{1:n}}\tfrac12 |k_{1:n}|^2|\hat\phi_n(k_{1:n})|^2=\| \cN(-\gensy)^{1/2}\phi\|^2
  \end{equs}
  where in the passage from the first to the second line we used that $\hat\phi$ is symmetric in its arguments.
  
  The bound \eqref{b:LA+} follows similarly by noting that if the
  denominator in \eqref{daillog} is squared, then the corresponding sum in~\eqref{daillog} is bounded
  uniformly in $N$.  
  
  As for \eqref{b:A-}, arguing as above we can focus on the case $\phi\in\fock_n$ for $n\in\N$. 
  If $\phi\in \fock_2$ then \cite[Lemma 3.10]{CET} gives the desired estimate: 
  just choose $G\equiv 1$ in that statement, so that $g(\cdot)\lesssim \log N$, and replace $\lambda^2$ by
  $\const^2$ in~\eqref{e:const}.
  If instead $\phi\in \Gamma\L^2_n, n>2$ note first that, by symmetrizing~\eqref{e:genam:fock} one has
  \begin{equ}[eq:Fourier_asym_genam]
    \cF(\genam\phi_n)(k_{1:n-1})={2n}\frac{\hat\lambda}{\sqrt{\log N}}\sum_{j=1}^{n-1}\sum_{\ell+m=k_j}|m|\nonlin_{k_j,-\ell}\hat\phi_n(\ell,m,k_{\{1:n-1\}\setminus \{j\}})
\end{equ}
so that, using the Cauchy-Schwarz inequality, the square of the left hand side of \eqref{b:A-} is upper bounded by
\begin{equ}[3c]
   n!\,4n^3\frac{\hat\lambda^2}{\log N}\sum_{k_{1:n-1}}\frac1{\mu+\tfrac12|k_{1:n-1}|^2}\Big(\sum_{\ell+m=k_1}|m|\nonlin_{k_1,-\ell}\hat \phi_n(\ell,m,k_{2:n-1})\Big)^2\,
 \end{equ}
(the factorial comes from the conventions in \eqref{e:ScalarPFock}).
As in the proof of \cite[Lemma 3.10]{CET}, the sum within the parenthesis is upper bounded by an absolute constant times
\begin{equs}
  |k_1|\sum_{\substack{\ell+m=k_1\\ 0<|\ell|,|m|\le N}}|\hat\phi_n(\ell,m,k_{2:n-1})|\,.
\end{equs}
Plugging this into \eqref{3c} one obtains an upper bound proportional to
\begin{equs}
  n!\,n^3&\frac{\hat\lambda^2}{\log N}\sum_{k_{1:n-1}}\Big(\sum_{\substack{\ell+m=k_1\\ 0<|\ell|,|m|\le N}}\frac{\sqrt{|m|^2+|\ell|^2}|\hat\phi_n(\ell,m,k_{2:n-1})|}{\sqrt{|m|^2+|\ell|^2}}\Big)^2\\
  &\leq n!\, n^3\hat\lambda^2\sum_{k_{1:n}}(|k_1|^2+|k_2|^2)|\hat\phi_n(k_{1:n})|^2\lesssim
 n!\, n^2  \hat \lambda^2\sum_{k_{1:n}}|k_{1:n}|^2 |\hat\phi_n(k_{1:n})|^2
\end{equs}
where in the first step we used Cauchy-Schwarz and~\eqref{daillog}, while in the second  
the symmetry of $\hat\phi$. The bound \eqref{b:A-} is then proven.
\end{proof}

\section{The effective diffusion coefficient}\label{sec:Dbulk}

The aim of this section is to determine the limit as $N\to\infty$ of the bulk diffusion coefficient $\Dbulk^N$
in~\eqref{def:Dbulk} and consequently deduce that $h^N$ is asymptotically diffusive as stated in Theorem~\ref{thm:Dbulk}. 
As a first step, we will derive an alternative representation for $\Dbulk^N$, thanks to 
which we will be able to connect its Laplace transform to the inverse of the generator $\gen$ of 
$u^N$ in~\eqref{e:AKPZ:u}. 

\begin{lemma}\label{l:Dbulk}
For $N\in\N$, let $u^N$ be the stationary solution of~\eqref{e:AKPZ:u} with coupling constant $\const$ in~\eqref{e:const} 
and $\gen$ its generator. 
Then, for all $t\geq 0$, one has 
% the bulk diffusion coefficient $\Dbulk^N$ in~\eqref{def:Dbulk} satisfies the following equality
\begin{equ}[e:DbulkNew]
t \,\Dbulk^N(t)=t + \Exp\Big[\Big(\int_0^t \const\cN^N_0[u_s^N]\dd s\Big)^2\Big]
\end{equ}
where $\cN^N_0$ is the $0$-th Fourier mode 
of the nonlinearity~\eqref{e:nonlin} of $h^N$ and is given by~\eqref{e:modozeroN}.

Moreover, for all $\mu>0$, the Laplace transform $\CD^N_\bulk$ of $t\mapsto t\Dbulk^N(t)$ is given by 
\begin{equ}[e:DbulkLaplace]
\CD^N_\bulk(\mu) \eqdef \int_0^\infty e^{-\mu t}\, t\,\Dbulk^N(t) \,\dd t =\frac{1}{\mu^2} + \frac{2}{\mu^2} \langle \nf_0,\,(\mu- \gen)^{-1}\nf_0\rangle
\end{equ}
where $ \nf_0$ is the unique kernel in $\fock_2$ such that 
\begin{equ}[eq:explFock]
\const\cN_0^N[\eta]=\sint_2(\nf_0)  \,.
\end{equ}
Its explicit Fourier transform can be read off~\eqref{e:modozeroN}.
\end{lemma}
\begin{proof}
Despite the different scaling, the proof of~\eqref{e:DbulkNew} follows 
the same steps as that of~\cite[Lemma 4.1]{CET}. 
On the other hand,~\eqref{e:DbulkLaplace} is a consequence of the 
representation~\eqref{e:DbulkNew} and~\cite[Lemma 5.1]{CES}. 
\end{proof}

For the previous Lemma to be of any use, we need to control the large $N$ behaviour of 
the scalar product at the right hand side of~\eqref{e:DbulkLaplace}. 
This, in turn, seems to require us to 
determine (sufficiently explicitly!) the solution $\tilde\fh^N$ of the {\it generator equation}  
\begin{equ}[e:GenEq1]
(\mu-\gen)\tilde\fh^N = \nf_0\,.
\end{equ}
Despite being linear,~\eqref{e:GenEq1} is extremely hard to study as $\gen$ is non-diagonal, 
in its chaos decomposition and in Fourier space, so that even though $\nf_0$ lives in 
$\wc_2$, $\tilde\fh^N$ will have non-trivial components in  {\it arbitrarily high} chaoses. 
As noted in~\cite[Section 3]{CET}, a way around it is to consider instead the {\it truncated generator equation}, 
an equation involving only finitely many chaoses and obtained by replacing 
the operator $\gen$ in~\eqref{e:GenEq1}  with $\gen_n\eqdef \sint_{\leq n}\gen \sint_{\leq n}$, 
where, for $n\in\N$, $I_{\leq n}$ is the
projection onto $\fock_{\leq n}\eqdef\bigoplus_{k=0}^n \fock_k$. 
More explicitly, we are led to study the solution $\tfhNn$ of 
\begin{equ}[e:TrGenEq]
(\mu-\gen_n)\tfhNn = \nf_0
\end{equ} 
which, by Lemma~\ref{lem:generator}, is equivalent to the hierarchical system 
\begin{equation}\label{e:System}
\begin{cases}
\big(\mu-\gensy\big)\tfhNn_n-\genap\tfhNn_{n-1}=0,\\
\big(\mu-\gensy\big)\tfhNn_{n-1}-\genap\tfhNn_{n-2}-\genam\tfhNn_{n}=0,\\
\dots\\
\big(\mu-\gensy\big)\tfhNn_2-\genap\tfhNn_1-\genam\tfhNn_3=\nf_0,\\
\big(\mu-\gensy\big)\tfhNn_1-\genam\tfhNn_2=0\,.
\end{cases}
\end{equation}

\begin{lemma}\label{l:TrGenEq}
For any $N,n\in\N$, the solution $\tfhNn$ of~\eqref{e:System} is such that $\tfhNn_1=0$ and 
\begin{equation}\label{e:SysDiff}
\begin{cases}
\tfhNn_2=(\mu-\gensy+\Op_n)^{-1}\nf_0& % \text{for $i=1$ or $2$}
\\
\tfhNn_j=(\mu-\gensy+\Op_{n+2-j})^{-1}\genap\tfhNn_{j-1}\,, & \text{for $j=3,\dots, n$\,}
\end{cases}
\end{equation}
where, for $j\geq 2$ the $\Op_j$'s are positive definite operators such that, for all $n\in\N$, $\Op_k(\wc_n)\subset\wc_n$, 
and are recursively defined via 
\begin{equation}\label{def:OpH}
\begin{cases}
\Op_2= 0\,, &\\
\Op_j = -\genam[\big(\mu-\gensy\big)+\Op_{j-1}]^{-1}\genap\,, & \text{for $j\geq 3$\,}.
\end{cases}
\end{equation}
\end{lemma}
\begin{proof}
As shown in~\cite[Section 3]{CET}, the expression~\eqref{e:SysDiff} can be obtained 
by solving~\eqref{e:System} recursively expressing 
$\tfhNn_j$ in terms of $\tfhNn_{j-1}$ starting from $j=n$.  
Concerning the first chaos component of $\tfhNn$,~\eqref{e:modozeroN} and~\eqref{e:nonlinCoefficient} 
state that the Fourier transform of 
$\nf_0$ is concentrated on those $\ell$ and $m\in\Z^2$ such that 
$\ell+m=0$, and the same holds for $\tfhNn_2$. 
In fact, one readily sees from Lemma \ref{lem:generator} that if $\hat\phi_n$ is concentrated on 
the set $A^{(n)}_0\eqdef\{k_{1:n}\in\Z^{2n}\colon \sum_{i\le n}k_i=0\}$ 
then $\cF(\mathcal A_\pm^N\phi_n)$ are concentrated on $A^{(n\pm1)}_0$. 
As a consequence, $\cF(\Op_j \phi_n)$ is concentrated on $A^{(n)}_0$, because $\Op_j$ is defined through $\genap,\genam$ and 
the diagonal operator $(-\gensy)$.
But then, by~\eqref{e:genam:fock}, $\genam\tfhNn_2=0$ 
which in turn implies that $\tfhNn_1$ solves $(\mu-\gensy) \tfhNn_1=0$ and hence $\tfhNn_1=0$. 
At last, for the properties of the operators $\Op_j$'s we refer to~\cite[Lemma 3.2]{CET}. 
\end{proof}

In order to be able to exploit the previous Lemma, we need to analyse the operators $\Op_j$'s. 
In the next section, we will show that, for each $j$, $\Op_j$ can be replaced with {\it diagonal} 
operators at a cost which is negligible in the $N\to\infty$ limit. The advantage of dealing with (positive) diagonal 
operators is that their inverse is {\it explicit} so that the terms at the right hand side of~\eqref{e:SysDiff} 
(and in particular $\tfhNn_2$) can be more easily controlled.

\subsection{The replacement argument}\label{sec:Repl}

Before stating the main result of this section, we need to introduce some notation. 
For $N\in\N$, let $\Ll$ be the non-negative function defined on $[1/2,\infty)$ as
\begin{equ}[e:L]
\Ll(x)\eqdef \frac{\fc}{\log N^2} \log \left(1+\frac{N^2}{x}\right)\,,
\end{equ}
where $\fc$ is the constant in~\eqref{e:nueff}. 

In the following proposition, we derive an approximation of the operators $\Op_j$ in terms of diagonal operators 
given by a nonlinear transformation of $\gensy$. 

\begin{proposition}\label{p:Hj}\
Let $\mu\geq 0$, $n\in\N$ and $\phi_1,\,\phi_2\in\fock_{n}$ be such that $(-\gensy)^{1/2}\phi_i\in\fock$, $i=1,2$.
Then, for every $j\in\N$, $j\geq 2$, there exists a constant $C=C(n,j)$ such that 
\begin{equs}[e:Hj]
|\langle \phi_1,\big[\Op_j+&G_j(\Ll(\mu-\gensy))\gensy\big]\phi_2\rangle|\\
&\leq C\eps_N \|(-\gensy)^{1/2}\phi_1\|\|(-\gensy)^{1/2}\phi_2\|
\end{equs}
where $\eps_N\to 0$ as $N\to\infty$ uniformly in $n$ and $j$, $\Op_j$ is defined in~\eqref{def:OpH} 
and the functions $G_j$'s on $[0,\infty)$ are such that for all $j$ and all $x$
\begin{enumerate}[noitemsep, label=(\roman*)]
\item $G_j(0)=0$ and $G_j(x)\geq 0$,  
\item $|G_j(x)|\leq x$, $|G'_j(x)|\leq 1$ and $|G''_j(x)|\leq 1$,
\end{enumerate}
and are recursively defined via the relation
\begin{equ}[e:Gj]
G_j(x)\eqdef
\begin{cases}
0 & j=2\,,\\
\int_0^x\frac{1}{1+G_{j-1}(y)}\dd y & j\geq 3\,.
\end{cases}
\end{equ}
\end{proposition}

The proof of the previous statement relies on the following Replacement Lemma which represents one of the  
main technical tools of the present paper. Its formulation, more general than what needed for the 
Proposition~\ref{p:Hj}, is chosen to be able to approximate a broader class of operators and 
it will play a crucial role also in the upcoming sections. 

\begin{lemma}[Replacement Lemma]\label{l:GenBound}
Let $H$ and $H^+$ be real-valued differentiable functions on $[0,\infty)$ such that 
\begin{enumerate}[noitemsep]
\item for every compact subset $\cK$ of $[0,\infty)$ there exists a finite constant $K>0$ for which 
  \begin{equ}
    \label{e:KAPPA}
\sup_{y\in\cK}\{|H(y)|,|H'(y)|,|H^+(y)|, |(H^+)'(y)|\}\leq K
\end{equ}
\item for all $y\ge0$, $H(y)\geq 1$ and  $H^+(y)\geq 0$.
\end{enumerate}
Let $\mu\geq 0$, $N\in\N$ and $\cJ^N$ be the operator defined by 
\begin{equs}[def:JN]
\cJ^N\eqdef %&(\mu-H(\Ll(\mu-\gensy))\gensy)^{-1}\times\\
(\mu-H^+(\Ll(\mu-\gensy))\gensy)\times (\mu-H(\Ll(\mu-\gensy))\gensy)^{-2}
\end{equs}
where $\Ll$ is as in~\eqref{e:L}. 
Let $n\in\N$ and $\phi_1,\,\phi_2\in\fock_{n}$ be such that $(-\gensy)^{1/2}\phi_i\in\fock$, $i=1,2$.
Then, there exists a constant $C=C(n)$, possibly depending on $K$, for which
\begin{equ}[e:JN]
|\langle (-\genam\cJ^N\genap+\tilde H(\Ll(\mu-\gensy))\gensy)\phi_1,\phi_2\rangle|\leq C\eps_N \|(-\gensy)^{1/2}\phi_1\|\|(-\gensy)^{1/2}\phi_2\|
\end{equ}
where $\eps_N\to 0$ as $N\to\infty$ uniformly in $n$ and $K$, and $\tilde H$ is a non-negative 
bounded differentiable function on $[0,\infty)$ with first derivative bounded on compact sets, defined by 
\begin{equ}[e:Htilde]
\tilde H(x)=\int_0^x \frac{H^+(y)}{H(y)^2}\dd y\,.
\end{equ}
\end{lemma}
\begin{proof}
In order to estimate the left hand side of~\eqref{e:JN}, we first decompose it in diagonal and 
off-diagonal terms as in Lemma~\ref{l:DiagOffDiag} and obtain
\begin{equs}
|\langle (-\genam\cJ^N\genap&+\tilde H(\Ll(\mu-\gensy))\gensy)\phi_1,\phi_2\rangle|\\
\leq&  \Big|\langle\cJ^N\genap \phi_1,\genap\phi_2\rangle_{\Di}+ \langle \tilde H(\Ll(\mu-\gensy))\gensy\phi_1,\phi_2\rangle_{\fock_n}\Big|\label{e:Di}\\
&%\qquad\qquad\qquad\qquad\qquad\qquad
+\sum_{i=1,2}|\langle (\cJ^N\genap\phi_1,\genap\phi_2\rangle_{\mathrm{off}_i}|\,.\label{e:oDi}
\end{equs}
We will separately bound~\eqref{e:Di} and~\eqref{e:oDi}, starting with the former. 
By~\eqref{def:JN}, one sees that $\cJ^N$ is a diagonal operator whose Fourier multiplier on $\wc_n$ is 
$J^N$, given by 
\begin{equ}[e:kernelJN]
J^N(\mu,k_{1:n})
\eqdef \frac{\mu+\tfrac12|k_{1:n}|^2H^+(\Ll(\mu+\tfrac12|k_{1:n}|^2))}{[\mu +\tfrac12|k_{1:n}|^2H(\Ll(\mu+\tfrac12|k_{1:n}|^2))]^2}\,.
\end{equ}
Hence, thanks to~\eqref{e:Diag} and the definition of $P^N$ in~\eqref{e:J},~\eqref{e:Di} equals
\begin{equs}
n!\Big|n&\sum_{k_{1:n}}\tfrac12|k_1|^2\hat\phi_1(k_{1:n})\overline{\hat\phi_2(k_{1:n})} \Big(P^N(\mu,k_{1:n})-\tilde H(\Ll(\mu+\tfrac12|k_{1:n}|^2))\Big)\Big|\\
&\leq \tilde\eps_N n! \,n \sum_{k_{1:n}}\tfrac12 |k_1|^2|\hat\phi_1(k_{1:n})||\hat\phi_2(k_{1:n})|= \tilde\eps_N n! \sum_{k_{1:n}} \tfrac12 |k_{1:n}|^2|\hat\phi_1(k_{1:n})||\hat\phi_2(k_{1:n})|\\
&\lesssim_n\tilde\eps_N \|(-\gensy)^{1/2}\phi_1\|\|(-\gensy)^{1/2}\phi_2\|
\end{equs}
where $\tilde\eps_N$ is given by 
\begin{equ}[e:epsN]
\tilde\eps_N\eqdef \sup_{\substack{k_{1:n}\in\Z^{2n}\setminus\{0\}\\ \mu\geq 0}}\Big|P^N(\mu,k_{1:n})-\tilde H(\Ll(\mu+\tfrac12|k_{1:n}|^2))  \Big| 
\end{equ}
and in the last two steps we respectively exploited the symmetry of $\phi_1$, $\phi_2$, and the Cauchy-Schwarz inequality. 
At this point, we only need to show that $\tilde\eps_N$ converges to $0$ as $N\to\infty$, 
which in turn is a direct consequence of Proposition~\ref{p:Approx}. 
In view of~\eqref{e:Approx} we further deduce the shape of $\tilde H$ in~\eqref{e:Htilde}. 
 so that the estimate in~\eqref{e:JN} on~\eqref{e:Di} follows.  

We now treat the off-diagonal terms for which we proceed similarly to what done in the proof of~\cite[Lemma 3.9]{CET}.  
We start with the off-diagonal terms of the first type defined in~\eqref{e:OffDiag1}. 
Relabelling the variables, we need to control
\begin{equs}
n! c_{\oD_1}(n)\frac{\hat\lambda^2}{\log N}\sum_{j_{1:3},k_{3:n}} J^N(\mu,j_{1:3},k_{3:n})& |j_1+j_2|\nonlin_{j_1,j_2}|j_1+j_3|\nonlin_{k_1,k_3}\times\\
&\times
\hat{\phi}_1(j_1+j_2,j_3,k_{3:n})\overline{\hat{\phi}_2(j_1+j_3,j_2,k_{3:n})}\,.
\end{equs}
%Note that $\tilde{H}(\Ll(\mu-\gensy))\gensy$ is a diagonal operator, so that it vanishes on the off-diagonal. Moreover, 
%By symmetry it is sufficient to bound the term in~\eqref{e:OffDiag1} for which $\underline{i}= (1,1)$ and $\underline{j}= (2,3)$. 
where $J^N$ is as in~\eqref{e:kernelJN}. Defining
\begin{equ}
\Phi_i(k_{1:n})=\hat{\phi}_i(k_{1:n})\prod_{j=1}^{n}|k_j|\,,\qquad i=1,2	
\end{equ}
applying Cauchy-Schwarz and estimating $|\nonlin_{j_1,j_2}|,\,|\nonlin_{j_1,j_3}|$ by $1$, 
we see that the expression above can be bounded by
\begin{equs}
&n! c_{\oD_1}(n)\frac{\hat\lambda^2}{\log N}\sum_{j_{1:3},k_{3:n}}\frac{|\Phi_1(j_1+j_2, j_3,k_{3:n})||\Phi_2(j_1+j_3,j_2,k_{3:n})|}{|j_3||j_2|\prod_{j=3}^n|k_j|^2} J^N(\mu,j_{1:3}, k_{3:n})\\
&\leq\frac{\hat\lambda^2}{\log N}\prod_{i=1}^{2}\Big( n! c_{\oD_1}(n)\sum_{j_{1:3},k_{3:n}}\frac{|\Phi_i(j_1+j_2, j_3,k_{3:n})|^2}{|j_3||j_2|\prod_{j=3}^n|k_j|^2} J^N(\mu,j_{1:3}, k_{3:n})\Big)^{\frac12}\\
&=\frac{\hat\lambda^2}{\log N}\prod_{i=1}^2\Big( n! c_{\oD_1}(n)\sum_{k_{1:n}}|\hat{\phi}_i(k_{1:n})|^2|k_1|^2|k_2|\sum_{j_1+j_2=k_1}\frac{1}{|j_2|} J^N(\mu,j_{1:2}, k_{2:n})\Big)^{\frac12}\,.\label{e:offi1}
\end{equs}
Note that the argument of $H^{+}$ in \eqref{e:kernelJN} is at most $\fc+1$, so that $H^+\le K$ since it is  bounded on compact sets. Also,  $H\geq 1$ and $(\mu+K x)/(\mu+x)\leq K\vee 1$, so we have
\begin{equ}
J^N(\mu,j_{1:2}, k_{2:n})\leq \frac{K\vee 1}{\mu +\tfrac12(|j_{1:2}|^2+|k_{2:n}|^2)}\,.
\end{equ}
Using further that $|j_{1:2}|^2+|k_{2:n}|^2\geq \tfrac14 (|j_2|^2+|k_{1:n}|^2)$ (whenever $j_1+j_2=k_1$), 
we obtain an upper bound for the inner sum in~\eqref{e:offi1} of the form
\begin{equs}
K\vee 1\sum_{j_1+j_2=k_1}\frac{1}{|j_2|(\mu +\tfrac{1}{8}(|j_2|^2+|k_{1:n}|^2))}	
\lesssim \frac{1}{N}\int_{\mathbb R^2} \frac{\dd x}{|x|(|k_{1:n}/N|^2 + |x|^2)}
\lesssim \frac{1}{|k_{1:n}|}\,,
\end{equs}
where, in  the last two steps, we passed to polar coordinates and computed the remaining integral. 
Hence, we conclude
\begin{equs}
|\langle (\cJ^N\genap\phi_1,\genap\phi_2\rangle_{\mathrm{off}_1}|&\lesssim \frac{\hat\lambda^2}{\log N}\prod_{i=1}^2\Big(n!c_{\oD_1}(n)\sum_{k_{1:n}}|\hat{\phi}_i(k_{1:n})|^2|k_1|^2\Big)^{\frac12}\\
&\lesssim \frac{1}{\log N}\|(-\gensy)\phi_1\|\|(-\gensy)\phi_2\|
\end{equs}
where the constant hidden into $\lesssim$ depends on $n$ and $K$. 

For the off-diagonal terms of the second type we proceed similarly. With the same notations as above, 
by~\eqref{e:OffDiag2} we have 
\begin{equs}
&|\langle (\cJ^N\genap\phi_1,\genap\phi_2\rangle_{\mathrm{off}_2}|\\
&\leq n! c_{\oD_2}(n)\frac{\hat\lambda^2}{\log N}\sum_{j_{1:4},k_{4:n}}\frac{|\Phi_1(j_1+j_2, j_{3:4},k_{3:n})||\Phi_2(j_3+j_4,j_{1:2},k_{4:n})|}{|j_1||j_2||j_3||j_4|\prod_{j=4}^n|k_j|^2} J^N(\mu,j_{1:3}, k_{3:n})\\
&\leq n! c_{\oD_2}(n)\frac{\hat\lambda^2}{\log N}\prod_{i=1}^2\Big( 	\sum_{j_{1:4},k_{4:n}}\frac{|\Phi_i(j_1+j_2,j_{3:4},k_{4:n})|^2}{|j_1||j_2||j_3||j_4|\prod_{j=4}^{n}|k_j|^2} J^N(\mu,j_{1:4},k_{4:n})\Big)^{1/2}\,.
\end{equs}
The inner sum can then be controlled as 
\begin{equs}
&\sum_{k_{1:n}} |\hat{\phi}_1(k_{1:n})|^2|k_1|^2|k_2||k_3|\sum_{j_1+j_2=k_1}\frac{|J^N(\mu,j_{1:2}, k_{2:n})|}{|j_1||j_2|}\\
&\lesssim \sum_{k_{1:n}}|\hat{\phi}_1(k_{1:n})|^2|k_1||k_2||k_3|\sum_{j_1+j_2=k_1}\frac{|J^N(\mu,j_{1:2}, k_{2:n})|}{|j_2|}\,\lesssim \sum_{k_{1:n}} |\hat{\phi}_1(k_{1:n})|^2\frac{|k_1||k_2||k_3|}{|k_{1:n}|}
\end{equs}
from which we can argue as for the off-diagonal terms of the first type 
(see the proof of~\cite[Lemma 3.9]{CET} for more details). 
\end{proof}

We are now ready to prove Proposition~\ref{p:Hj}. 
  
  \begin{proof}[of Proposition~\ref{p:Hj}]
  The proof goes via induction. If $j=2$, then $\Op_j=G_j=0$ and there is
  nothing to prove.  Assume that the result is true for
  some $j\geq 2$. We write
\begin{equs}
\langle \big[\Op_{j+1} + G_{j+1}(\Ll(\mu-\gensy))\gensy\big]\phi_1,\phi_2\rangle 	= \one+\two\,,
	\end{equs}
where 
\begin{equs}
	\one&=\langle [(\mu-\gensy+\Op_j)^{-1}-(\mu-[1+G_j(\Ll(\mu -\gensy))]\gensy)^{-1}]\genap \phi_1,\genap\phi_2\rangle\,\\
	\two&= \langle [-\genam (\mu-[1+G_j(\Ll(\mu -\gensy))]\gensy)^{-1}\genap + G_{j+1}(\Ll(\mu-\gensy))\gensy]\phi_1,\phi_2\rangle\,.	
\end{equs}
We first treat $\one$.
To that end let $A=(\mu-\gensy+\CC_A)$ and $B=(\mu-\gensy+\CC_B)$ where $\CC_A\eqdef \Op_j$ while 
$\CC_B\eqdef -G_j(\Ll(\mu -\gensy))\gensy$, and use the relation
\begin{equs}[e:OpId]
A^{-1}-B^{-1}= A^{-1}[B-A]B^{-1}
	\end{equs}
to deduce that
\begin{equs}[e:(I)]
|\one|&=\left|\langle (-G_j(\Ll(\mu-\gensy))\gensy - \Op_j)B^{-1}\genap \phi_1,A^{-1}\genap\phi_2\rangle\right|\\
&\leq C \eps_N \|(-\gensy)^{1/2}B^{-1}\genap \phi_1\|\,\|(-\gensy)^{1/2}A^{-1}\genap \phi_2\|\,,
	\end{equs}
the last passage being a consequence of the induction hypothesis which holds as 
$\CC_A,\CC_B$ are diagonal in the chaos, 
so that $A^{-1}\genap\phi_2,B^{-1}\genap\phi_1\in \wc_{n+1}$ and $C=C(n+1,j)$.
Note that both the operators $\CC_A$ and $\CC_B$ are non-negative definite, and 
this is the only property we will need in order to bound the two terms at the right hand side. 
Hence, let $\CC$ be either $\CC_A$ or $\CC_B$ and $i=1$ or $2$. Since $\mu\geq 0$, we have 
\begin{equs}[e:Positivity]
\|&(-\gensy)^{1/2}(\mu-\gensy+\CC)^{-1}\genap \phi_i\|^2\\
&=\langle (-\gensy)(\mu-\gensy+\CC)^{-1}\genap\phi_i, (\mu-\gensy+\CC)^{-1}\genap\phi_i\rangle\\
&\leq \langle (\mu-\gensy +\CC)(\mu-\gensy+\CC)^{-1}\genap\phi_i, (\mu-\gensy+\CC)^{-1}\genap\phi_i\rangle\\
&= \langle \genap\phi_i, (\mu-\gensy+\CC)^{-1}\genap\phi_i\rangle\\
&\leq \langle \genap\phi_i, (\mu-\gensy)^{-1}\genap\phi_i\rangle\lesssim \|(-\gensy)^{1/2}\phi_i\|^2
\end{equs}	
where in the last step we applied Lemma~\ref{l:A+A-}. 
Hence, we have the desired estimate for $\one$. We turn to $\two$, for which we will apply Lemma~\ref{l:GenBound}. 
In the notations therein, set $H^+\equiv H\equiv 1+G_j$ and note that, 
since by the induction hypothesis $G_j$ satisfies (i)-(ii),
$H^+$ and $H$  satisfy assumptions 1. and 2., from which~\eqref{e:Hj} follows at once with $G_{j+1}$ 
defined according to~\eqref{e:Gj}. Finally, as (i)-(ii) hold for $G_j$, it is straightforward to verify that they also 
hold for $G_{j+1}$. 
\end{proof}

\subsection{The diffusivity}\label{s:TheDiff}

The aim of this section is to prove Theorem~\ref{thm:Dbulk}. First, we will determine the large $N$ limit 
 of~\eqref{e:DbulkLaplace} when the operator $\gen$ is replaced by $\gen_n$ and $n$ is fixed.

\begin{proposition}\label{e:Dbulkn}
Let $n\in\N$, $\mu>0$ and, for $N\in\N$, $\tfhNn$ be the solution of the truncated generator equation~\eqref{e:TrGenEq}. 
Then, the following holds 
\begin{equ}[e:nDbulk]
\lim_{N\to\infty} \langle \nf_0,(\mu-\gen_n)^{-1}\nf_0\rangle=\lim_{N\to\infty} \langle \nf_0,\tfhNn\rangle=  \frac12G_{n+1}(\fc)
\end{equ}
where the functions $G_n$'s are defined in~\eqref{e:Gj} and $\fc$ is as in~\eqref{e:nueff}. 
\end{proposition}
\begin{proof}
By orthogonality of Wiener chaoses and the definition of $\tfhNn$ in~\eqref{e:TrGenEq}, we have 
\begin{equs}[e:FirstBDiff]
\langle\nf_0,&(\mu-\gen_n)^{-1}\nf_0\rangle=\langle \nf_0,\tfhNn_2\rangle=\langle\nf_0,\big(\mu-\gensy+\Op_n\big)^{-1}\nf_0\rangle\\
=&\langle\nf_0,\big[(\mu-\gensy+\Op_n)^{-1}-(\mu-\gensy[1+G_n(\Ll(\mu-\gensy))])^{-1}\big]\nf_0\rangle\\
& \qquad\qquad\qquad\qquad\qquad+ \langle\nf_0,(\mu-\gensy[1+G_n(\Ll(\mu-\gensy))])^{-1}\nf_0\rangle\,.
\end{equs}
As in the proof of Proposition~\ref{p:Hj}, 
set $A=(\mu-\gensy+\CC_A)$ and $B=(\mu-\gensy+\CC_B)$ where $\CC_A\eqdef \Op_n$ while 
$\CC_B\eqdef -G_n(\Ll(\mu -\gensy))\gensy$. Exploiting~\eqref{e:OpId} 
we see that absolute value of the first summand at the right hand side of~\eqref{e:FirstBDiff} equals 
\begin{equs}
|\langle (-G_n(\Ll(\mu-\gensy))\gensy - \Op_n)&B^{-1}\nf_0,A^{-1}\nf_0\rangle|\\
&\lesssim_n \eps_N \|(-\gensy)^{1/2}B^{-1}\nf_0\|\,\|(-\gensy)^{1/2}A^{-1}\nf_0\|
\end{equs}
where the bound is a consequence of~\eqref{e:Hj}. 
Denoting by $\CC$ either $\CC_A$ or $\CC_B$, and arguing as in~\eqref{e:Positivity}
we get 
\begin{equ}
\|(-\gensy)^{1/2}(\mu-\gensy+\CC)^{-1}\nf_0\|^2\lesssim \|(-\gensy)^{-1/2}\nf_0\|^2\lesssim\frac{1}{\log N}\sum_{\ell\neq 0} (\nonlin_{\ell,-\ell})^2\frac{1}{|\ell|^2}\lesssim 1\,.
\end{equ}
Therefore, the first summand in~\eqref{e:FirstBDiff} goes to $0$ as $N\to\infty$ and 
we only need to focus on the second. Note that
\begin{equs}
\langle&\nf_0,(\mu-\gensy[1+G_n(\Ll(\mu-\gensy))])^{-1}\nf_0\rangle\\
&=\frac{2\hat\lambda^2}{\log N}\sum_{\ell\neq 0} (\nonlin_{\ell,-\ell})^2\frac{1}{\mu+|\ell|^2[1+G_n(\Ll(\mu+|\ell|^2))]}\\
&= \frac{2\hat\lambda^2}{\log N}\int_{\mathbb R^2}\frac{(\nonlin_{xN ,-xN})^2}{\mu_N+|x|^2[1+G_n(\Ll(N^2(\mu_N+|x|^2)))]}\dd x+o(1)\\
&=\frac{2\hat\lambda^2}{\log N}\int_{\mathbb R^2}\frac{(\nonlin_{xN,-xN})^2}{(\mu_N+|x|^2)(\mu_N+|x|^2+1)[1+G_n(\Ll(N^2(\mu_N+|x|^2)))]}\dd x+o(1)
\end{equs}
where  $\mu_N\eqdef \mu/N^2$ and 
$o(1)$ goes to $0$ as $N\to\infty$. 
The second step follows by Riemann-sum approximation: by defining $x=\ell/N$ and recalling the definition \eqref{e:nonlinCoefficient} of $\nonlin_{\ell,m}$, we replace the sum by an integral which grows like $\log N$ (see below), plus a $O(1)$ error that gives $o(1)$ once divided by $\log N$. 
The third step is a consequence of~\eqref{e:Repl1.1} (with $H\equiv 1+G_n,\,k_{1:n}=0$) 
and~\eqref{e:Repl3} (with $H^+\equiv H\equiv 1+G_n$).
For the last integral, we proceed similarly to the proof of Proposition~\ref{p:Approx}, 
i.e. we pass to polar coordinates and get that it equals
\begin{equs}
% &\frac{2\hat\lambda^2}{\log N}\int_{\mathbb R^2}\frac{(\nonlin_{x,-x})^2}{(\mu_N+|x|^2)(\mu_N+|x|^2+1)[1+G_n(\Ll(N^2(\mu_N+|x|^2)))]}\dd x\\
&\frac{\hat\lambda^2}{\log N^2} \frac{1}{2\pi^2}\int_0^{2\pi} \cos(2\theta)^2\dd\theta\int_{\mu_N+\tfrac{1}{N^2}}^{\mu_N+1}\frac{\dd \rho}{\rho(\rho+1)[1+G_n(\Ll(N^2\rho))]}+o(1)\\%\int_{1/N}^1\frac{r \dd r}{(\mu_N+\tfrac12 r^2)(\mu_N+\tfrac12 r^2+1)[1+G_n(\Ll(N^2(\mu_N+\tfrac12r^2)))]}\\
&=\frac{\fc}{2 \log N^2}\int_{\mu_N+1/N^2}^{\mu_N+1}\frac{\dd \rho}{\rho(\rho+1)[1+G_n(\Ll(N^2\rho))]}+o(1)\,,
\end{equs}
where we used the definition of $\fc$ in~\eqref{e:nueff}.
%% the $o(1)$ being due to 
%% \begin{equs}
%% \frac{\fc}{2\log N^2}\int_{\mu_N}^{\mu_N+\tfrac{1}{N^2}}&\frac{\dd \rho}{\rho(\rho+1)[1+G_j(\Ll(\rho))]}\\
%% &\leq \frac{\fc}{2\log N^2}\int_{\mu_N}^{\mu_N+\tfrac{1}{N^2}}\frac{\dd \rho}{\rho}
%% =\frac{\fc}{2\log N^2}\log\Big(1+\frac{1}{\mu}\Big)
%% \end{equs}
%% which goes to $0$ as $N\to\infty$.
% he $1/N^2$ in the lower extremum of integration comes from the fact that by definition, 
% $\nonlin_{x,-x}$ is zero when the norm of its argument is smaller than $1/N$.
By~\eqref{e:Gj}, the last expression becomes
\begin{equs}
% \frac{\fc}{2\log N^2}&\int_{\mu_N+1/N^2}^{\mu_N+1}\frac{\dd \rho}{\rho(\rho+1)[1+G_n(\Ll(N^2\rho))]}=
\frac12\int_{\Ll(N^2(\mu_N+1))}^{\Ll(\mu+1)}\frac{\dd y}{1+G_n(y)}+o(1)
=\frac12 G_{n+1}(\Ll(\mu+1))+o(1)\,.
\end{equs}
Summarising, all the estimates above give
\begin{equ}[e:ApproxDbulk]
\langle\nf_0,(\mu-\gensy[1+G_n(\Ll(\mu-\gensy))])^{-1}\nf_0\rangle=\frac12 G_{n+1}(\Ll(\mu+1))+o(1)
\end{equ}
from which, by definition of $\Ll$ and the continuity of $G_{n+1}$ (see Proposition~\ref{p:Hj}(ii)), 
\begin{equs}
\lim_{N\to\infty}\langle \nf_0,\tilde\fh^{N,n}\rangle = \lim_{N\to\infty} \frac12 G_{n+1}\Big( \frac{\fc}{\log N^2}\log\Big(1+\frac{N^2}{\mu+1}\Big)\Big)=\frac12 G_{n+1}(\fc)\,.
\end{equs}
and~\eqref{e:nDbulk} follows at once. 
\end{proof}

Thanks to the previous proposition, we are ready to complete the proof of Theorem~\ref{thm:Dbulk}. 
\begin{proof}[of Theorem~\ref{thm:Dbulk}]
At first, we will determine the limit as $N\to\infty$ of $\cD_{\bulk}^N$, 
the Laplace transform of $t\mapsto t \,\Dbulk^N(t)$. 
As mentioned in Section~\ref{sec:Dbulk}, this amounts to characterise the behaviour of the 
scalar product at the right hand side of~\eqref{e:DbulkLaplace}. To do so, we will exploit~\cite[Lemma 3.1]{CET} 
which holds {\it mutatis mutandis} in the present context as the operators we are considering are the same 
apart from the fact that $\lambda$ therein is replaced by $\const$. 
It states that, for every $\mu>0$, $n\in\N$ and $N\in\N$ the following holds 
\begin{equ}[e:Sandwich]
\langle \nf_0,\tilde \fh^{N,2n+1}\rangle\leq \langle \nf_0,(\mu-\gen)^{-1}\nf_0\rangle\leq \langle \nf_0,\tilde\fh^{N,2n}\rangle\,.
\end{equ}
As a consequence, we deduce that for all $n\in\N$ and $\mu>0$
\begin{equs}
\limsup_{N\to\infty}\langle \nf_0,(\mu-\gen)^{-1}\nf_0\rangle\leq \limsup_{N\to\infty}\langle \nf_0,\tilde \fh^{N,2n}\rangle=\frac12 G_{2n+1}(\fc)
\end{equs}
and 
\begin{equs}
\liminf_{N\to\infty}\langle \nf_0,(\mu-\gen)^{-1}\nf_0\rangle\geq \liminf_{N\to\infty}\langle \nf_0,\tilde\fh^{N,2n+1}\rangle=\frac12 G_{2n+2}(\fc).
\end{equs}
so that it suffices to argue that the sequence $G_n(\fc)$ converges as $n\to\infty$. 
Notice that, in view of Proposition~\ref{p:Hj} (ii), the functions $G_n$'s are uniformly bounded (in $n$) on $[0,\fc]$ and 
have uniformly bounded first and second derivative. Hence, by Arzel\`a-Ascoli's theorem, 
the sequence $\{G_n\}_n$ is relatively compact in $\CC^1([0,K], \R)$ for any $K<\infty$, the space of real-valued 
bounded continuously differentiable functions, and therefore converges along subsequences in $\CC^1([0,K])$. 
Thanks to Proposition~\ref{p:Hj} (i), any limit point 
must be non-negative, $0$ at $0$ and, in view of the recursive relation in~\eqref{e:Gj}, it must satisfy 
the following Cauchy-value problem
\begin{equ}
{G}'=\frac{1}{1+G}\,,\qquad G(0)=0\,,%\quad G\geq 0
\end{equ}
which has a unique solution given by 
\begin{equ}[e:LimG]
G(x)\eqdef \sqrt{2x+1}-1\,. %,\qquad x\in[0,\fc]\,.
\end{equ}
In conclusion, the sequence $\{G_n\}_n$ converges in $\CC^1([0,K])$ to $G$ for any $K< \infty$, which in particular means that
\begin{equ}[e:GjConv]
\lim_{j\to\infty} G_j(\fc)=G(\fc)=\sqrt{2\fc+1}-1=\nueff-1
\end{equ}
and consequently by~\eqref{e:DbulkLaplace}, we obtain
\begin{equ}[e:DiffLaplace]
\lim_{N\to\infty}\CD^N_\bulk(\mu)=\frac{\sqrt{2\fc +1}}{\mu^2}=\frac{\nueff}{\mu^2} \,.
\end{equ}
In order to translate the result from Laplace to real time (the only
point to be justified is the interchange between the inverse Laplace
transform and the $N\to\infty$ limit) and conclude the proof of
Theorem~\ref{thm:Dbulk}, we argue as follows. For $N\in\N$, let 
$x^N$ be the map on $[0,\infty)$ defined by the right hand side
of~\eqref{e:DbulkNew}, i.e.
\begin{equ}
x^N(t)\eqdef t\Dbulk^N(t)=t + \Exp\big[Q^N(t)^2\big]\,,\qquad\text{where}\qquad  Q^N(t)\eqdef \int_0^t \const\cN^N_0[u_s^N]\dd s\,.
\end{equ}
In the proof of~\cite[Theorem 4.8]{CES} it was shown that there exists $c>0$ such that 
for all $N\in\N$
\begin{equ}[e:linGrowth]
x^N(t)\leq ct\,,\qquad \text{for all $t\geq 0$.}
\end{equ}
By stationarity, we also deduce that 
\begin{equs}[e:Equicont]
|x^N(t)-&x^N(s)|\leq |t-s|+\Big|\Exp[Q^N(t)^2-Q^N(s)^2]\Big|\\
&=|t-s|+\Big|\Exp[(Q^N(t)-Q^N(s))(Q^N(t)+Q^N(s))]\Big|\\
&\leq |t-s|+ \Exp[|Q^N(t)-Q^N(s)|^2]^{1/2}\Big(\Exp[Q^N(t)^2]^{1/2}+\Exp[Q^N(s)^2]^{1/2}\Big)\\
&\lesssim |t-s|+ |t-s|^{1/2} (t\vee s)^{1/2}\,. 
\end{equs}
Thanks to~\eqref{e:linGrowth} and~\eqref{e:Equicont}, the assumptions of~\cite[Corollary 10.12]{D} are satisfied. 
Hence, the sequence $\{x^N\}_N$ is relatively compact with respect to the uniform on compacts topology on $C(\R_+,\R)$. 
Let  $x\in C(\R_+,\R)$ be a limit point and, slightly abusing notations, let 
$\{x^N\}_N$ be the subsequence converging to it uniformly 
on compacts (and therefore pointwise). Note that, since the constant $c$ in~\eqref{e:linGrowth} is independent of $N$ 
then also $x$ satisfies the same bound.
Therefore, by the Dominated Convergence Theorem and~\eqref{e:DiffLaplace}, we obtain that, for all $\mu>0$,
\begin{equs}
 \frac{\sqrt{2\fc +1}}{\mu^2}=\lim_{N\to\infty} \CD^N_\bulk(\mu)&=\lim_{N\to\infty}\int_0^\infty e^{-\mu t} x^N(t)\dd t = \int_0^\infty e^{-\mu t} x(t)\dd t
\end{equs}
which, by uniqueness of the Laplace transform, implies that the sequence $x^N$ 
converges to a unique limit. Therefore, for all $t>0$, we have
\begin{equ}
\lim_{N\to\infty}  \Dbulk^N(t) =\lim_{N\to\infty}  \frac{x^N(t)}{t}=\sqrt{2\fc +1}=\nueff
\end{equ}
and the proof is concluded. 
\end{proof}

\section{The truncated generator equation}
\label{sec:truncated}

\subsection{The martingale problem}

The goal of the present and the next sections is to prove the main
result of the paper, namely Theorem~\ref{thm:Conv}.  Since by
Theorem~\ref{thm:Tightness} we already know that, for every $T>0$, the
sequence $\{h^N\}_N$ is tight in the space $C([0,T], \cD'(\T^2))$, it
remains to verify that any limit point $h$ equals in law the solution to
the almost-stationary stochastic heat equation~\eqref{e:SHEfinal} that
we here recall
\begin{equ}[e:SHEakpz]
\partial_t h = \frac{\nueff}{2}\Delta h +\sqrt{\nueff} \xi\,,\qquad h_0=\chi
\end{equ}
where $\nueff$ is defined according to~\eqref{e:Dbulk}, $\xi$ is a space-time white noise on $\R_+\times\T^2$ 
and, as in Definition~\ref{def:QSsol}, $\chi$ is such that $(-\Delta)^{1/2}\chi=\eta$, for $\eta$ a 
spatial white noise on $\T^2$.

Heuristically, in order to show convergence in law of  $h^N$ 
to the solution $h$ of ~\eqref{e:SHEakpz}, 
%we need to be able to determine the large $N$ limit of 
%a wide class ``observables'' of $h^N$, i.e. of functionals of the form $\ff(h^N)$. Heuristically, 
it suffices to prove that $\lim_N \ff(h^N)$ is distributed as 
$\ff(h)$ for a sufficiently rich family of observables $\ff$'s.
%As we will see in a moment  (see in particular Definition~\ref{def:MPSHE} and Theorem
%\ref{thm:WellPosedMP} below), 
The family of functionals identified by the following martingale problem associated to~\eqref{e:SHEakpz} 
is made of those of the form $\ff(h)=h(\Psi)$ for $\Psi\in\cD(\T^2)$. 
%A family which is wide enough 
%for our purposes is formed by linear functionals of the type $\ff(h)=h(\Psi)$ for $\Psi\in\cD(\T^2)$, 
%as the following martingale characterisation of the law of~\eqref{e:SHEakpz} shows. 

\begin{definition}\label{def:MPSHE}
Let $T>0$, $\Omega=C([0,T], \cD'(\T^2))$ and $\CG=\CB(C([0,T], \cD'(\T^2)))$ the canonical Borel $\sigma$-algebra on it. 
Let $\chi$ be such that $(-\Delta)^{1/2}\chi=\eta$ for $\eta$ a zero-average space white noise on $\T^2$. 
We say that a probability measure $\P$ on $(\Omega,\CG)$ {\it solves the martingale problem for 
$\gensy^{\rm eff}\eqdef  {\nueff}\gensy$
with initial distribution $\chi$}, if for all $\Psi\in\cD(\T^2)$, the canonical process $h$ under $\P$ is such that 
\begin{equs}
&\cM^\Psi_t\eqdef h_t(\Psi) - \chi(\Psi) -\frac{\nueff}{2} \int_0^t h_s(\Delta \Psi)\dd s\label{e:Mart1}\\
&\Gamma^\Psi_t\eqdef(\cM^\Psi_t)^2-t\,\nueff  \|\Psi\|^2_{L^2(\T^2)}\label{e:Mart2}
\end{equs}
are local martingales.
\end{definition} 

In the next theorem (see e.g.~\cite{MW}) it is shown that the family of observables mentioned above is 
wide enough in that the martingale problem is well-posed (i.e. it has a unique solution) and 
its unique solution is the law of~\eqref{e:SHEakpz}.
%
%The previous martingale problem is indeed well-posed as shown, e.g., in~\cite{MW}, and 
%uniquely characterises the law of~\eqref{e:SHEakpz}.
\begin{theorem}\textnormal{\cite[Theorem D.1]{MW}}\label{thm:WellPosedMP}
The martingale problem for $\gensy^{\rm eff}$ with initial distribution $\chi$ in Definition~\ref{def:MPSHE} is well-posed 
and uniquely characterises the law of the solution to~\eqref{e:SHEakpz} on $C([0,T], \cD'(\T^2))$.
%, i.e. there exists a unique probability 
%measure $\P$ on $C([0,T], \cD'(\T^2))$ such that the processes in~\eqref{e:Mart1} and~\eqref{e:Mart2} are 
%local martingales. 
\end{theorem}

In view of the previous statement, the proof of Theorem~\ref{thm:Conv}
boils down to verify that any limit point $h$ of the sequence
$\{h^N\}_N$ is such that for all $\Psi\in\cD(\T^2)$, the
processes~\eqref{e:Mart1} and~\eqref{e:Mart2} are (local)
martingales. In the next section, we give an overview of the strategy we will follow and introduce the necessary notations.
The actual technical work starts in Section \ref{s:RecEst}.

\subsection{Theorem \ref{thm:Conv}: strategy of proof}
\label{s:strategy}

%Recall that our main goal is to show that for   any limit point $h$ of the sequence
%$\{h^N\}_N$, the
%processes~\eqref{e:Mart1} and~\eqref{e:Mart2} are (local)
%martingales. 
As argued in Section~\ref{sec:prelim}, 
we can fully recover $h^N$ from $u^N$ and~\eqref{eq:modozero} so that we can (and will) directly work with $u^N$, and we will come back to $h^N$ only in Section \ref{sec:ProofMain}.
 
Translated in terms of $u$,  Theorem \ref{thm:WellPosedMP} tells us that we need to show that, 
for any limit point $u=(u_t)_{t\ge0}$
of the sequence $\{u^N\}_N=\{(u^N_t)_{t\ge0}\}_N$ (known to exist by~\cite[Theorem 4.5]{CES}), 
and for every test function
$\psi\in H^1_0(\T^2)$, the stochastic processes
\begin{equ}[idea1]
M^\psi_t\eqdef u_t(\psi)-u_0(\psi)-\int_0^t \gensy^{\rm eff}u_s(\psi)\dd s \,,\qquad \gamma^\psi_t\eqdef (M^\psi_t)^2-t \nueff\|\psi\|_{H^1(\T^2)}^2
\end{equ}
are martingales. 
What we know for free, as $u^N$ is a Markov process with generator $\gen$, is that, 
for any reasonable functional $b^N=b^N(\eta)$, allowed to explicitly depend on $N$,
\begin{equs}
  \label{idea4}
M^{N}_t(b^N)\eqdef  b^N(u^N_t)-b^N(u^N_0)-\int_0^t \mathcal L^N b^N(u^N_s)\dd s
\end{equs}
and $(M^{N}_t(b^N))^2-\langle M^{N}(b^N)\rangle_t$ are martingales. Therefore, if we were able to 
show that for all $\psi$ there exists $b^N$ such that the difference 
\begin{equ}[e:idea5]
\Big(u^N_t(\psi)-u^N_0(\psi)-\int_0^t \gensy^{\rm eff}u^N_s(\psi)\dd s\Big)- \Big( b^N(u^N_t)-b^N(u^N_0)-\int_0^t \gen b^N(u^N_s)\dd s \Big) 
\end{equ}
is small for $N$ large in, say, mean square, then by standard martingale convergence theorems and 
tightness of $u^N$ we could conclude that the sequence $\{M^N(b^N)\}_N$ converges (at least along subsequences) 
and that the limit, $M^\psi$, is indeed a martingale (for a rigorous proof see Section~\ref{sec:ProofMain}). 
Proving that $\gamma^\psi_t$ is also a martingale additionally
requires the quadratic variation of
$M^{N}(b^N)$ to be asymptotically non-random:
this is a very challenging technical step, deferred to Sections~\ref{sec:Feynman} and \ref{sec:tourdeforce}. 

Hence, the goal is to determine $b^N$ in such a way that~\eqref{e:idea5} is small in mean square, 
which in turn requires $b^N$ to be such that 
\begin{equ}[e:point1]
\| \eta(\psi)-b^N(\eta)\|\quad\overset{N\to\infty}{\longrightarrow} \quad 0\,,
\end{equ}
thus ensuring that $u^N_\cdot(\psi)-b^N(u^N_\cdot)$ vanishes, and
\begin{equ}[e:point2]
\| (-\gensy)^{-\tfrac12}\Big(\gensy^{\rm eff}\eta(\psi)-\gen b^N(\eta)\Big)\|\quad\overset{N\to\infty}{\longrightarrow} \quad 0
\end{equ}
which, by the It\^o trick (see~\eqref{e:ItoTrick}), guarantees that the difference of integral terms in~\eqref{e:idea5} 
also goes to $0$.

As a first attempt (motivated by the observables needed in Definition~\ref{def:MPSHE}), 
one might be led to take $b=b^N$ independent of $N$ and 
of the form $b(\eta)\eqdef \eta(\psi)$, 
so that~\eqref{e:point1} is trivially satisfied. 
That said, as it can be readily checked by the definition of $\gen$ and $\cM^N$ 
in~\eqref{e:gens}-\eqref{e:gena} and~\eqref{e:nonlin} respectively, 
and the orthogonality of different Wiener chaoses, with this choice of $b$, the left hand side of~\eqref{e:point2} reduces to 
\begin{equs}
\|(-\gensy)^{-\tfrac12}&\Big((\nueff-1)\gensy\eta(\psi)- \const\cM^N[\eta]\Big)\|\\
&=(\nueff-1)\|(-\gensy)^{\tfrac12}\eta(\psi)\|+\const\|(-\gensy)^{-\tfrac12} \cM^N[\eta]\|
\end{equs}
and the right hand side, even though bounded, certainly {\it does not vanish as $N\to\infty$}! 
The problem is that functionals of the form $\ff(\eta)=\eta(\psi)$ live only in the first chaos 
and consequently are not able to capture the fluctuations coming from higher chaoses which are
induced by the nonlinearity. 

Ideally, one would like to take instead $b=b^N$ {\it explicitly depending on $N$} in such a way that 
the left hand side of~\eqref{e:point2} is $0$, i.e. so that 
\begin{equ}[idea2]
-\gen b^N(\eta)=
  \frac{\nueff}2\eta(-\Delta\psi)= -\gensy^{\rm eff}\eta(\psi)\,,
\end{equ}
and subsequently prove that such $b^N$ also satisfies~\eqref{e:point1}. 
As noted in Section~\ref{sec:Dbulk}, it is very difficult to analyse
the generator equation \eqref{idea2} directly since its solution will have non-trivial components in {\it every} chaos. 
To overcome this issue, we resort instead to introduce a truncation parameter $n$ 
and consider the truncated generator equation
where $\gen$ is replaced by $\gen_n$ (defined in Section \ref{sec:Dbulk}). 
More precisely,  for $N\in\N$ and $n\in\N$ we analyse 
 \begin{equ}[e:GenEq]
-\gen_n\,\tfbNn=-\gensy^{\rm eff}\eta(\psi)\,,
\end{equ}
and, to control the zero-mode of $h^N$, 
\begin{equ}[e:GenEq2]
  \qquad -\gen_n\,\tfhNn=\nf_0,
\end{equ}
where $\nf_0$ is defined in~\eqref{eq:explFock}. 
In view of the the results in Section~\ref{sec:Dbulk} 
we can (and will) even do better. As Proposition~\ref{p:Hj} hints at, 
we can approximate observables as $\tfbNn$ and $\tfhNn$ by (almost) explicit functionals, $\fbNn$ and $\fhNn$, 
whose definition involves 
{\it perturbations of the operator $\gensy$} (see $\ffNn_a$ in \eqref{e:Fjn} below, with   $a=1,2$) {\it which are converging to $\gensy^{\rm eff}$ in the large $n$ limit} (see the proof of Theorem~\ref{thm:Dbulk} in Section~\ref{s:TheDiff}). 
The observables $\fbNn$ are those for which we will establish ~\eqref{e:point1} and~\eqref{e:point2}, though
as a result of {\it two} limits, {\it first take the limit $N\to\infty$} and then let $n$ tend to infinity. 
\medskip

Let us describe more in detail the implementation of this program, rigorously introduce the 
observables mentioned above and state the main results of this section. 
To do so let us begin by looking at the solutions $\tfbNn$, $\tfhNn$ of \eqref{e:GenEq} and \eqref{e:GenEq2}. 
As can be seen by arguing as in the proof of Lemma~\ref{l:TrGenEq}, 
$\tfbNn$ and $\tfhNn$ have a similar structure and we will therefore treat them likewise. 
Let 
\begin{equ}[e:InputGenEq]
\fg_1 \eqdef -\gensy^{\rm eff}\eta(\psi)=\frac{\nueff}{2}\eta(-\Delta\psi)\in\fock_1\qquad\text{and}\qquad \fg_2=\nf_0\in\fock_2
\end{equ}
and $\tffNn_a=(\tffNn_{a,j})_{j=a,\dots,n}$ be the solution to \eqref{e:GenEq} for $a=1$ and 
to \eqref{e:GenEq2} for $a=2$ (in the latter case, Lemma~\ref{l:TrGenEq} ensures that the component of 
$\tffNn_a$ in the first Wiener chaos is $0$). 
Writing down \eqref{e:GenEq} in each chaos component analogously to what  was done in~\eqref{e:System} 
and recalling the definition of $\Op_j$ in~\eqref{def:OpH}, $\tffNn_a$ can be equivalently written as
\begin{equation}\label{e:GenEqSys}
\begin{cases}
\tffNn_{a,a}=(-\gensy+\Op_{n+2-a})^{-1}\fg_a \,, & % \text{for $i=1$ or $2$}
\\
\tffNn_{a,j}=(-\gensy+\Op_{n+2-j})^{-1}\genap\tffNn_{a,j-1}\,, & \text{for $j=a+1,\dots, n$\,}.
\end{cases}
\end{equation}
In light of Proposition~\ref{p:Hj}, we now define the observables we aim at analysing,
$\ffNn_a$, $a=1,2$, (in the above discussion $\ffNn_a=\fbNn$ for $a=1$ and $\fhNn$ for $a=2$) 
as in~\eqref{e:GenEqSys} but with the operators $\Op_j$ replaced
by diagonal (and therefore explicitly invertible) operators expressed as perturbations of $-\gensy$. 
More precisely, for $a=1,2$, we set 
$\ffNn_a=(\ffNn_{a,j})_{j=a,\dots,n}$ as 
\begin{equation}\label{e:Fjn}
\begin{cases}
\ffNn_{a,a}\eqdef (-\gensy[1+G_{n+2-a}(\Ll(-\gensy))] )^{-1} \fg_a \,, &\\
\ffNn_{a,j}\eqdef (-\gensy[1+G_{n+2-j}(\Ll(-\gensy))] )^{-1} \genap \ffNn_{a,j-1}\,, & \text{for $j=a+1,\dots, n$\,}
\end{cases}
\end{equation}
where the functions $G_j$'s are those in~\eqref{e:Gj}. 

We have now all the elements we need in order to state the main result of this section 
which rigorously establishes~\eqref{e:point1} and~\eqref{e:point2} for the observables $\fbNn$. 
We will also obtain analogous estimates for $\fhNn$, which, 
even though has vanishing norm (see~\eqref{e:hL2}),
will be useful in the description of the $0$-mode of $h^N$ (see Corollary~\ref{c:zeromode}). 

\begin{theorem}\label{thm:ControlGen}
For $\psi\in H^1_0(\T^2)$, $N,\,n\in\N$ and $a=1,2$, let $\ffNn_a$ be given in~\eqref{e:Fjn} and 
$\fbNn=\ffNn_1$, $\fhNn=\ffNn_2$. 
Then, we have 
\begin{equs}
&\lim_{n\to\infty}\lim_{N\to\infty} \|\fbNn(\eta)-\eta(\psi)\|=0\,,\label{e:bL2}\\
&\lim_{n\to\infty}\lim_{N\to\infty} \| (-\gensy)^{-\tfrac12}\Big(\gensy^{\rm eff}\eta(\psi)-\gen \fbNn(\eta)\Big)\|=0\,,\label{e:bH1}
\end{equs}
and 
\begin{equs}
&\lim_{n\to\infty}\lim_{N\to\infty} \|\fhNn(\eta)\|=0\,,\label{e:hL2}\\
&\lim_{n\to\infty}\lim_{N\to\infty} \| (-\gensy)^{-\tfrac12}\Big(\nf_0(\eta)-\gen \fhNn(\eta)\Big)\|=0\,.\label{e:hH1}
\end{equs}
\end{theorem}

So far we have only briefly alluded to the process $\gamma^\psi$ in~\eqref{idea1}. 
As we already mentioned, proving that it is indeed a martingale boils down to show (see
Theorem~\ref{thm:Mart}) that the quadratic variation of $M^\psi$ in \eqref{idea1} is deterministic, and given by
$t\nueff\|\psi\|^2_{H^1(\T^2)}$.  The proof that the quadratic
variation is non-random is based on entirely different ideas than the
ones of the present section and will be given in
Section~\ref{sec:Feynman}. As far as the expectation of the quadratic
variation is concerned, it turns out (see Section \ref{sec:Mart}) to be proportional to $2t$ times the
$N,n\to\infty$ limit of $\|(-\gensy)^{1/2}\fbNn\|^2$.  
As the computation of this limit exploits the techniques introduced in this section, 
we give below the following proposition. 

\begin{proposition}\label{p:H1Norm}
For $\psi\in H^1_0(\T^2)$, $N,\,n\in\N$ and $a=1,2$, let $\ffNn_a$ be given in~\eqref{e:Fjn} and 
$\fbNn=\ffNn_1$, $\fhNn=\ffNn_2$. 
Then, we have 
\begin{equs}
&\lim_{n\to\infty}\lim_{N\to\infty} \|(-\gensy)^{1/2} \fbNn\|^2=\frac{\nueff}{2}\|\psi\|^2_{H^1(\T^2)}\,,\label{e:limnNb}\\
&\lim_{n\to\infty}\lim_{N\to\infty} \|(-\gensy)^{1/2} \fhNn\|^2=\frac{\nueff -1}{2} \,,\label{e:limnNh}\\
&\lim_{N\to\infty} \langle (-\gensy)^{1/2} \fbNn, (-\gensy)^{1/2} \fhNn\rangle=0\,.\label{e:limNbh}
\end{equs}
\end{proposition}

Let us briefly comment on the structure of the rest of the paper. 
In Section \ref{s:RecEst}, we derive suitable recursive estimates on the functions $\ffNn_{a,j}$ 
in~\eqref{e:Fjn}, which will be used to prove both that 
$\ffNn_a$ is indeed a good approximation
to $\tffNn_a$ (Proposition~\ref{p:ApproxGenEq}),
Theorems \ref{thm:ControlGen} and Proposition \ref{p:H1Norm} (Section~\ref{s:ProofThms}). 
In Section~\ref{s:theconvergence}, we will focus on the martingale $\gamma^\psi$ 
and show how to exploit the previous 
statements in order to conclude the proof of Theorem~\ref{thm:Conv}. 

\subsection{Recursive estimates}
\label{s:RecEst}

Given the recursive definition of $\ffNn_a$, $a=1,2$, in~\eqref{e:Fjn}, we first determine 
a general bound on functionals with a similar structure. 

\begin{lemma}\label{l:PrB}
Let $j,n\in\N$ with $n\ge j$ and $\fg\in\fock_j$. Define $\ft^N\in\fock_{j+1}$ according to 
\begin{equ}[e:RecGen]
\ft^N\eqdef (-\gensy[1+G_{n+2-j}(\Ll(-\gensy))] )^{-1}\genap\fg\,,
\end{equ}
where the functions $G_j$'s are given as in~\eqref{e:Gj}. 
Then, there exists a constant $C=C(j)>0$ such that the following estimates hold
\begin{equs}
\|\ft^N\|&\leq C\eps_N\|\fg\|\label{b:PrBL2}\\
\|(-\gensy)^{1/2}\ft^N\|&\leq C\|(-\gensy)^{1/2}\fg\|\,\label{b:PrBH1}
\end{equs}
and $\eps_N\to 0$ as $N\to\infty$ uniformly over $j$ and $n$. 
\end{lemma}
\begin{proof}
By Proposition~\ref{p:Hj}, the functions $G_j$'s are non-negative, therefore~\eqref{b:PrBL2} is a 
direct consequence of~\eqref{b:LA+}. For~\eqref{b:PrBH1}, we exploit the argument in~\eqref{e:Positivity}, 
which gives
\begin{equs}
\|(-\gensy)^{1/2}\ft^N\|&=\|(-\gensy)^{1/2}(-\gensy[1+G_{n+2-j}(\Ll(-\gensy)] )^{-1} \genap \fg\|\\
&\leq \|(-\gensy)^{-1/2}\genap \fg\|\lesssim\|(-\gensy)^{1/2}\fg\|
\end{equs}
where the last bound follows by Lemma~\ref{l:A+A-} with $\mu=0$. 
\end{proof}

As a consequence, we deduce the main estimates on the different components of 
$\ffNn_a=(\ffNn_a)_{j=a,\dots,n}$, which will be crucial in the proof of~\eqref{e:bL2} and~\eqref{e:hL2}. 

\begin{proposition}\label{l:PrelimBounds}
Let $n\in\N$. There exists a constant $C=C(n)>0$ such that  
\begin{equ}[e:PrelimBounds]
\begin{matrix}
&\|\fbNn_j\|\leq C \eps_N^{j-1} \|\psi\|_{L^2(\T^2)}\,,& \|(-\gensy)^{1/2}\fbNn_j\|\leq C \|\psi\|_{H^1(\T^2)}\,,& j=1,\dots,n \\
&\|\fhNn_j\|\leq C \eps_N^{j-1}\,, & \|(-\gensy)^{1/2}\fhNn_j\|\leq C\,, & j=2,\dots,n
\end{matrix}
\end{equ}
%\begin{equs}[e:PrelimBounds]
%\|\fbNn_j\|&\leq C \eps_N^{j-1} \|\psi\|_{L^2(\T^2)}\qquad& \|\fhNn_j\|&\leq C \eps_N^{j-1} \\
%\|(-\gensy)^{1/2}\fbNn_j\|&\leq C \|\psi\|_{H^1(\T^2)}\qquad& \|(-\gensy)^{1/2}\fhNn_j\|&\leq C 
%\end{equs}
where $\eps_N$ goes to $0$ as $N\to\infty$ uniformly over $n$. 
\end{proposition}
\begin{proof}
Notice first that, by definition of $\fbNn_j$ and $\fhNn_j$, 
we can recursively apply estimates~\eqref{b:PrBL2} and~\eqref{b:PrBH1} and obtain 
\begin{equ}
\begin{matrix}
&\|\fbNn_j\|\lesssim \eps_N^{j-1} \|\fbNn_1\|\,,& \|(-\gensy)^{1/2}\fbNn_j\|\lesssim  \|(-\gensy)^{1/2}\fbNn_1\|\,, \\
&\|\fhNn_j\|\lesssim \eps_N^{j-2}\|\fhNn_2\|\,, & \|(-\gensy)^{1/2}\fhNn_j\|\lesssim \|(-\gensy)^{1/2}\fhNn_2\|.
\end{matrix}
\end{equ}
%\begin{equs}
%\|\fbNn_j\|&\lesssim \eps_N^{j-1} \|\fbNn_1\| \qquad& \|\fhNn_j\|&\lesssim \eps_N^{j-2}\|\fhNn_2\| \\
%\|(-\gensy)^{1/2}\fbNn_j\|&\lesssim  \|(-\gensy)^{1/2}\fbNn_1\|\qquad& \|(-\gensy)^{1/2}\fhNn_j\|&\lesssim \|(-\gensy)^{1/2}\fhNn_2\|\,.
%\end{equs}
Hence, we are left to estimate the right hand sides of the above for which we will 
repeatedly use the fact that, for all $j$, $G_j\geq 0$. 

Let us begin with $\|\fbNn_1\|$ and $\|(-\gensy)^{1/2}\fbNn_1\|$, for which we have
\begin{equs}
\|\fbNn_1\|^2\lesssim \|(-\gensy)^{-1}\eta(-\Delta\psi)\|^2=2\|\psi\|_{L^2(\T^2)}^2
\end{equs}
where the last equality follows by Lemma~\ref{lem:generator}. 
Arguing as in~\eqref{e:Positivity}, it is not hard to see that
\begin{equs}
\|(-\gensy)^{1/2}\fbNn_1\|^2&\lesssim \|(-\gensy[1+G_{n+1}(\Ll(-\gensy)) ])^{-1/2} \eta(-\Delta)\|^2\\
&\leq \|(-\gensy)^{-1/2} \eta(-\Delta)\|^2=\langle (-\gensy)^{-1}\eta(-\Delta\psi),\eta(-\Delta\psi)\rangle\\
&=2\langle \eta(\psi),\eta(-\Delta\psi)\rangle=2\|\psi\|_{H^1(\T^2)}^2\,.
\end{equs}
For $\|\fhNn_2\|$ and $\|(-\gensy)^{1/2}\fhNn_2\|$, we proceed similarly 
and, for $\alpha\in\{0,1/2\}$, we get
\begin{equ}
\|(-\gensy)^{\alpha}\fhNn_2\|^2\leq\|(-\gensy)^{-1+\alpha}\nf_0\|^2\lesssim\frac{1}{\log N} \sum_{0<|\ell|\leq N}\frac{1}{|\ell|^{4-4\alpha}}\lesssim \frac{1}{(\log N)^{1-2\alpha}}
\end{equ}
from which the required estimates follow at once. 
\end{proof}

In order to prove~\eqref{e:bH1},~\eqref{e:hH1} and especially Proposition~\ref{p:H1Norm}, we need first 
to introduce suitable functions. 
Let $n\in\N$ and, for $\lambda>0$, $\fc$ be given as in~\eqref{e:nueff}. 
For $j=1,\dots,n$ and $i=0,\dots,j-1$, let $G^{+,n+2-j}_i$ be the function on $[0,\infty)$ defined as
\begin{equation}\label{e:G+}
G^{+,n+2-j}_i(x)\eqdef
\begin{cases}
1\,,& \text{if $i=0$}\\
\int_{\Delta^{i-1}_x} \prod_{\ell=0}^{i-1} \frac{1}{[1+G_{n+2-j+\ell}(x_\ell)]^2}\dd x_{0:i-1}\,,& \text{if $i=1,\dots,j-1$\,,}
\end{cases}
\end{equation}
where the functions $G_j$ are those in~\eqref{e:Gj} and $\Delta^{i-1}_x$ is the $i$-th dimensional simplex, i.e. 
$\Delta^{i-1}_x\eqdef \{x_{0:i-1}\eqdef (x_0,\dots,x_{i-1})\in [0,x]^{i}\,:\,0\leq x_0\leq\dots\leq x_{i-1}\leq x\}$. 
In the following lemma, we derive the main properties of the functions $G^{+,n+2-j}_i$.

\begin{lemma}\label{l:G+}
Let $n\in\N$ and, for $\hat\lambda>0$, $\fc$ be given as in~\eqref{e:nueff}. 
For $j=1,\dots,n$ and $i=0,\dots,j-1$, let $G^{+,n+2-j}_i$ be the functions defined in~\eqref{e:G+}. 

Then, for all $n,\,j,\,i$ as above and $x\ge0$, 
\begin{enumerate}[noitemsep, label=(\roman*)]
\item\label{i:Pos} $G^{+,n+2-j}_{i}(x)\geq0$ and $G^{+,n+2-j}_{i}(0)=\mathds{1}_{i=0}$,
\item\label{i:Rec} for $i\neq 0$, the following identity holds
\begin{equ}[e:DerG+]
G^{+,n+2-j}_{i}(x)'=\frac{G^{+,n+2-j}_{i-1}(x)}{[1+G_{n+1-j+i}(x)]^2}
\end{equ}
\item\label{i:Bound} we have the bounds
\begin{equs}
G^{+,n+2-j}_{i}(x)&\leq \frac{x^{i-1}}{(i-1)!}  \qquad&\text{for $i\geq 1$}\label{b:UnifG+}\\
|G^{+,n+2-j}_{i}(x)'|&\leq \frac{x^{i-2}}{(i-2)!} \qquad&\text{for $i\geq 2$}\label{b:UnifDerG+}\\\
|G^{+,n+2-j}_{i}(x)''|&\leq 2\frac{x^{i-2}}{(i-2)!} + \frac{x^{i-3}}{(i-3)!} \qquad&\text{for $i\geq 3$}\label{b:UnifDer2G+}\
\end{equs}
and the first and second derivative of $G^{+,n+2-j}_{1}$ and the second derivative of $G^{+,n+2-j}_{2}$ 
 are uniformly bounded by $1$. 
%\item the function $S_n(x)\eqdef \sum_{j=1}^{n} G^{+,n+2-j}_{j-1}$ converges as  
\end{enumerate}
\end{lemma}
\begin{proof}
Parts~\ref{i:Pos} and~\ref{i:Rec} are direct consequences of the definition of $G^{+,n+2-j}_{i}$ in~\eqref{e:G+} 
and the non-negativity of the functions $G_j$'s determined in Proposition~\ref{p:Hj}. 
For~\ref{i:Bound}, we exploit once more the latter facts to see that 
\begin{equ}
G^{+,n+2-j}_{i}(x)=\int_{\Delta^{i-1}_x} \prod_{\ell=0}^{i-1} \frac{1}{[1+G_{n+2-j+\ell}(x_\ell)]^2}\dd x_{1:i-1}\leq \int_{\Delta^{i-1}_x} \dd x_{1:i-1}\leq \frac{x^{i-1}}{(i-1)!}\,.
\end{equ}
The estimate~\eqref{b:UnifDerG+} can be easily derived by~\eqref{e:DerG+} and~\eqref{b:UnifG+} 
while~\eqref{b:UnifDer2G+} can be easily seen to hold thanks to Leibniz differentiation rule,~\eqref{b:UnifG+} and~\eqref{b:UnifDerG+}. 
\end{proof}

The next lemma establishes the main recursive bound needed in the identification of the $N\to\infty$ limit of the 
norms $\|(-\gensy)^{1/2}\ffNn_{a,j}\|$, $a=1,2$ and fixed $j=a,\dots,n$, 
whose proof is given in the subsequent Proposition~\ref{l:H1}. 

\begin{lemma}\label{p:RecH1}
Let  $n,\,j\in\N$, $1\le j\leq n$ and $G_j$'s be the functions given as in~\eqref{e:Gj}. 
For $\fg_1,\,\fg_2\in\fock_j$, let $\ft^N_1,\ft^N_2\in\fock_{j+1}$ be defined as in~\eqref{e:RecGen} with $G_{n+2-j}$ 
replaced by $G_{n+i+2-j}$ and $\fg$ replaced by $\fg_1$ and $\fg_2$ respectively.  
Then, there exists a constant $C=C(n,j)>0$ such that for all $i=0,\dots, j-1$ 
\begin{align}
\Big|\langle(-G_i^{+,n+2-j}(\Ll(-\gensy))\gensy)&\ft_1^N,\ft_2^N\rangle-\langle(-G_{i+1}^{+,n+2-j}(\Ll(-\gensy))\gensy)\fg_1,\fg_2\rangle\Big|\notag\\
&\leq C \eps_N  \|(-\gensy)^{1/2}\fg_1\|\|(-\gensy)^{1/2}\fg_2\|\label{e:Ind}
\end{align}
where $G_i^{+,n+2-j}$ is defined in~\eqref{e:G+} and 
$\eps_N$ goes to $0$ as $N\to\infty$ uniformly over $n,j,i$. 
\end{lemma}
\begin{proof}
To prove the statement, set $H^+\equiv G_i^{+,n+2-j}$ and 
$H\equiv 1+ G_{n+2-j+i}$ and define $\cJ^N$ as in~\eqref{def:JN} with $\mu=0$.  
By Proposition~\ref{p:Hj} and Lemma~\ref{l:G+}, both $H$ and $H^+$ satisfy the assumptions of Lemma~\ref{l:GenBound}, 
Now, by the definition of $\ft^N_a$, $a=1,2$, we have
\begin{equs}
\langle(-H^+(\Ll(-\gensy))\gensy)\ft_1^N,\ft_2^N\rangle%\langle (-H^+(\Ll(\mu-\gensy))\gensy)\ff^N, \ff^N\rangle\\
%=\langle (-H^+(\Ll(-\gensy))\gensy)\ft_1^N, \ft_2^N\rangle +\mu \langle\ft^N_1,\ft^N_2\rangle\\
=\langle \cJ^N\genap \fg_1, \genap \fg_2\rangle % +\mu \langle\ft^N_1,\ft^N_2\rangle
=\langle -\genam\cJ^N\genap \fg_1, \fg_2\rangle%  +\mu \langle\ft^N_1,\ft^N_2\rangle\\
% &\leq \langle -\genam\cJ^N\genap \fg_1, \fg_2\rangle +\mu \|\ft^N_1\|\|\ft^N_2\|
\,.
\end{equs}
By~\eqref{e:Htilde}, $\tilde H$ is given by 
\begin{equs}
  \label{e:intHtilde}
\tilde H(x)&=\int_0^x \frac{H^+(x_i)}{H(x_i)^2}\dd x_i=\int_0^x\frac{G^{+,n+2-j}_i(x_i)}{[1+G_{n+2-j+i}(x_i)]^2} \dd x_i\\
&=\int_0^x G^{+,n+2-j}_{i+1}(x_i)'\dd x_i= G^{+,n+2-j}_{i+1}(x)
\end{equs}
where the passage from the first to the second line follows by~\eqref{e:DerG+}. 
Hence, Lemma~\ref{l:GenBound} implies that~\eqref{e:Ind} can be bounded above by 
\begin{equs}
|\langle (-\genam\cJ^N\genap - G_{i+1}^{+,n+2-j}(\Ll(-\gensy))\gensy) \fg_1, \fg_2\rangle|
\lesssim \eps_N \|(-\gensy)^{1/2} \fg_1\|\|(-\gensy)^{1/2} \fg_2\|\,.% +\eps_N \|\fg_1\|\|\fg_2\|\lesssim \eps_N \|(-\gensy)^{1/2} \fg_1\|\|(-\gensy)^{1/2} \fg_2\|
\end{equs}
% for the second, while the last estimate is a consequence of $\|\cdot\|\leq \|(-\gensy)^{1/2}\cdot\|$ 
% since $-\gensy\geq 1$ as we are on a torus.
\end{proof}

\begin{proposition}\label{l:H1}
For any $n\in\N$, we have 
\begin{equs}
&\lim_{N\to\infty} \|(-\gensy)^{1/2} \fbNn_j\|^2= \frac{G^{+,n+2-j}_{j-1}(\fc)}{(1+G_{n+1}(\fc))^2}\frac{\nueff^2}{2}\|\psi\|^2_{H^1(\T^2)}\,,\label{e:H1bj}\\
&\lim_{N\to\infty} \|(-\gensy)^{1/2} \fhNn_j\|^2=\frac12 G^{+,n+2-j}_{j-1}(\fc)\label{e:H1hj}\\
&\lim_{N\to\infty} \langle (-\gensy)^{1/2} \fbNn_j, (-\gensy)^{1/2} \fhNn_j\rangle=0\label{e:H1hbj}
\end{equs}
where~\eqref{e:H1bj} and~\eqref{e:H1hbj} hold for all $j=1,\dots,n$ and~\eqref{e:H1hj} for $j=2,\dots,n$; the functions $G$'s and $G^{+}$'s are respectively defined 
in~\eqref{e:Gj} and~\eqref{e:G+}. 
\end{proposition}
\begin{proof}
To prove~\eqref{e:H1bj}~\eqref{e:H1hj} and~\eqref{e:H1hbj}, 
we apply Lemma~\ref{p:RecH1} inductively on $i$ starting from $i=0$. 
Recall that in this latter case, 
by definition~\eqref{e:G+}, $G^{+,n+2-j}_0\equiv 1$ so that $\|(-\gensy)^{1/2} \fbNn_j\|^2=\langle -G^{+,n+2-j}_0(\Ll(-\gensy))\gensy \fbNn_{j},  \fbNn_j\rangle$ and the same holds for $\|(-\gensy)^{1/2} \fhNn_j\|^2$. Therefore, for $j\geq 2$, one gets
\begin{equs}
\|(-\gensy)^{1/2} \fbNn_j\|^2&=\|(-G_{1}^{+,n+2-j}(\Ll(-\gensy))\gensy)^{1/2}\fbNn_{j-1}\|^2 +o(1)\\
\|(-\gensy)^{1/2} \fhNn_j\|^2&=\|(-G_{1}^{+,n+2-j}(\Ll(-\gensy))\gensy)^{1/2}\fhNn_{j-1}\|^2 +o(1)\\
\langle (-\gensy) \fbNn_j, \fhNn_j\rangle&=\langle(-G_{1}^{+,n+2-j}(\Ll(-\gensy))\gensy)\fbNn_{j-1},\fhNn_{j-1}\rangle +o(1)
\end{equs}
where $o(1)$ goes to $0$ as $N\to\infty$ (the error terms comes from the right hand side of \eqref{e:Ind}, 
with the norms being bounded via Lemma \ref{l:PrelimBounds}).
Applying again Lemma~\ref{p:RecH1} for $i=1,\dots,j-1$ and $i=1,\ldots, j-2$ respectively, one arrives at
\begin{equs}
\|(-\gensy)^{1/2} \fbNn_j\|^2&=\|(-G_{j-1}^{+,n+2-j}(\Ll(-\gensy))\gensy)^{1/2}\fbNn_{1}\|^2 +o(1)\label{e:H1bj1}\\
\|(-\gensy)^{1/2} \fhNn_j\|^2&=\|(-G_{j-2}^{+,n+2-j}(\Ll(-\gensy))\gensy)^{1/2}\fhNn_{2}\|^2 +o(1)\label{e:H1hj1}\\
\langle (-\gensy) \fbNn_j, \fhNn_j\rangle&=\langle(-G_{j-2}^{+,n+2-j}(\Ll(-\gensy))\gensy)\fbNn_2,\fhNn_{2}\rangle +o(1)\label{e:H1hbj1}
\end{equs}
where, as above, $o(1)$ goes to $0$ as $N\to\infty$ (the error terms accumulate along the recursion, but the number of steps does not grow with $N$). Here,~\eqref{e:H1bj1} is also valid for $j=1$ since in that case $G_{j-1}^{+,n+2-j}\equiv 1$.

Since $\fbNn_1$ is defined by the first line of \eqref{e:Fjn} with $a=1$, the first summand at the right hand side of~\eqref{e:H1bj1} equals
\begin{equs}
\|(-G_{j-1}^{+,n+2-j}&(\Ll(-\gensy))\gensy)^{1/2}\fbNn_{1}\|^2\\
&=\frac{\nueff^2}{2}\sum_{k\in\Z^2\setminus\{0\}}|k|^2 |\hat\psi(k)|^2\frac{  G_{j-1}^{+,n+2-j}(\Ll(\tfrac12|k|^2))}{[1+G_{n+1}(\Ll(\tfrac12|k|^2))]^2}\,
\end{equs} 
and, since the quotient is uniformly bounded in $N$ and $\psi\in H_0^1(\T^2)$, we can apply the dominated convergence 
theorem to pass the limit inside the sum. Both $G_{j-1}^{+,n+2-j}$ and $G_{n+1}$ are continuous and, for fixed $k$, 
$\Ll(\tfrac12|k|^2)$ converges to $\fc$ as $N\to\infty$, so that~\eqref{e:H1bj} follows at once. 

For the right hand side of~\eqref{e:H1hj1} instead we observe that $\fhNn_{2}$ is defined by the first line of \eqref{e:Fjn} with $a=2$ so that we have 
\begin{equs}
\|(-&G_{j-2}^{+,n+2-j}(\Ll(-\gensy))\gensy)^{1/2}\fhNn_{2}\|^2 \\
&= \frac{2\hat\lambda^2}{\log N}\sum_{\ell\neq 0} (\nonlin_{\ell,-\ell})^2\frac{|\ell|^2G_{j-2}^{+,n+2-j}(\Ll(|\ell|^2))}{|\ell|^4[1+G_n(\Ll(|\ell|^2))]^2}\\
&= \frac{2\hat\lambda^2}{\log N}\int_{\mathbb R^2}(\nonlin_{xN,-xN})^2\frac{G_{j-2}^{+,n+2-j}(\Ll(N^2|x|^2))}{|x|^2[1+G_n(\Ll(N^2|x|^2))]^2}\dd x+o(1)\\
&=\frac12 \int_{\Ll(N^2)}^{\Ll(1)}\frac{G_{j-2}^{+,n+2-j}(y)}{[1+G_n(y)]^2}\dd y +o(1)=\frac12 G_{j-1}^{+,n+2-j}(\Ll(1))+o(1)
\end{equs}
where in the steps from the second to the fourth line we adopted the very same strategy as in the proof of 
Theorem~\ref{thm:Dbulk} to determine~\eqref{e:ApproxDbulk}, and, as therein, $o(1)$ goes to $0$ as $N\to\infty$. In the last equality we used also the identity \eqref{e:intHtilde}.
Therefore~\eqref{e:H1hj} follows by recalling the definition of $\Ll$ and taking the limit as $N\to\infty$. 

At last, we have 
\begin{equ}
\langle(-G_{j-2}^{+,n+2-j}(\Ll(-\gensy))\gensy)\fbNn_2,\fhNn_{2}\rangle=0
\end{equ}
as the Fourier support of $\fbNn_2$ has empty intersection with that of $\fhNn_{2}$. 
Indeed, while by Lemma~\ref{lem:generator} $\CF(\fbNn_2)(\ell,m)=0$ for all $\ell\,,m$ such that $\ell+m=0$, 
$\CF(\fhNn_2)(\ell,m)\neq 0$ only for $\ell\,,m$ such that $\ell+m=0$. 
Given that the operator $G_{j-2}^{+,n+2-j}(\Ll(-\gensy))\gensy$ is non-negative 
and diagonal in Fourier space, and therefore does not alter the support of the functions to which it is applied, 
the conclusion follows at once. 
\end{proof}

We now turn our attention to the difference between $\ffNn_a$ and $\tffNn_a$. 
First, in the next lemma we show that the cost at which we replace 
$\Op_j$ with the operator $(-\gensy)G_{j}(\Ll(-\gensy))$, vanishes in the large $N$ limit. 

\begin{lemma}\label{p:RecRepl}
Let  $n,\,j\in\N$ such that $1\le j\leq n$.  
Let $\Op_j$ be the operators defined in~\eqref{def:OpH} and $G_j$'s be the functions given as in~\eqref{e:Gj}. 
For $\fs,\,\tilde\fs\in\fock_j$, let $\fp^N$ and $\tilde\fp^N$ be respectively defined as 
\begin{equs}[e:RecGen2]
\fp^N&\eqdef (-\gensy[1+G_{n+2-j}(\Ll(-\gensy))] )^{-1}\fs\\
\tilde\fp^N&\eqdef (-\gensy+\Op_{n+2-j})^{-1}\tilde\fs\,.
\end{equs}
Then, there exists a constant $C=C(n,j)>0$ such that 
\begin{equ}[e:ApproxGenG]
\|(-\gensy)^{1/2}(\tilde\fp^N-\fp^N)\|\leq C\Big(\|(-\gensy)^{-1/2}(\tilde\fs-\fs)\|+\eps_N\|(-\gensy)^{-1/2}\fs\|\Big)
\end{equ}
where $\eps_N\to 0$ as $N\to\infty$ uniformly in $n$ and $j$. Further, if $\fs=\genap\fg$ and $\tilde\fs=\genap\tilde\fg$ 
for some $\fg,\,\tilde\fg\in\fock_{j-1}$, then, under the same conditions as above,
\begin{equ}[e:Approx+GenG]
\|(-\gensy)^{1/2}(\tilde\fp^N-\fp^N)\|\leq C\Big(\|(-\gensy)^{1/2}(\tilde\fg-\fg)\|+\eps_N\|(-\gensy)^{1/2}\fg\|\Big)\,.
\end{equ} 
\end{lemma}
\begin{proof}
Notice first that~\eqref{e:Approx+GenG} follows immediately by~\eqref{e:ApproxGenG} 
and~\eqref{b:A+} in Lemma~\ref{l:A+A-}, so that we only need to focus on~\eqref{e:ApproxGenG}. 

Similarly to the proof of Proposition~\ref{p:Hj}, 
let $A_j=-\gensy+\Op_{n+2-j}$ and $B_j=-\gensy-G_{n+2-j}(\Ll( -\gensy))\gensy$. 
We add and subtract $(-\gensy)^{-1/2}A_j^{-1}\fs$ to the left hand side of~\eqref{e:ApproxGenG}, 
use the definition of $\fp^N$ and $\tilde\fp^N$ in~\eqref{e:RecGen2} and apply triangular inequality, 
thus obtaining
\begin{equs}
\|(-\gensy)^{1/2} (\tilde \fp^N- \fp^N)\|\leq&\|(-\gensy)^{1/2} [A_j^{-1}-B_j^{-1}] \fs\|+\|(-\gensy)^{1/2} A_j^{-1} (\tilde\fs- \fs)\|\\
=&\one+\two\,.
\end{equs}
For $\one$, we exploit~\eqref{e:OpId} and get
\begin{equs}
\one&=\|(-\gensy)^{1/2} A_j^{-1}[B_j-A_j]B_j^{-1}\fs\|
%&=\|(-\gensy)^{1/2} A_j^{-1}[\Op_{n+2-j}+G_{n+2-j}(\Ll(\mu -\gensy))\gensy]B_j^{-1}\genap \ffNn_{j-1}\|\\
\leq \|(-\gensy)^{-1/2}[B_j-A_j]B_j^{-1}\fs\|
\end{equs}
where we used also the non-negativity of $\Op_{n+2-j}$. 
Now, by Proposition~\ref{p:Hj} and 
in particular~\eqref{e:Hj}, we have 
\begin{equs}
\|&(-\gensy)^{-1/2}[B_j-A_j]B_j^{-1}\fs\|^2\\
&=\langle [-\Op_{n+2-j}-G_{n+2-j}(\Ll( -\gensy))\gensy]B_j^{-1}\fs, (-\gensy)^{-1}[B_j-A_j]B_j^{-1}\fs\rangle\\
&\lesssim \eps_N \|(-\gensy)^{1/2}B_j^{-1}\fs\|\|(-\gensy)^{-1/2}[B_j-A_j]B_j^{-1}\fs\|\\
&\leq \frac{\eps_N}2 \|(-\gensy)^{-1/2}[B_j-A_j]B_j^{-1}\fs\|^2 +\frac{\eps_N}{2}\|(-\gensy)^{1/2}B_j^{-1}\fs\|^2.
\end{equs}
%which holds for any $\delta>0$. Upon choosing $\delta$ sufficiently small, w
For $N$ large enough, we can bring the first summand 
to the left hand side and conclude that
\begin{equs}
\|(-\gensy)^{-1/2}[B_j-A_j]B_j^{-1}\fs\|&\lesssim \sqrt{\eps_N}\|(-\gensy)^{1/2}B_j^{-1}\fs\|% \\
% &\leq \eps_N\|B_j^{-1/2}\fs\|
\leq \sqrt{\eps_N}\|(-\gensy)^{-1/2}\fs\|
\end{equs}
where the last  inequality follows from the fact that  $-G_{n+2-j}(\Ll(-\gensy))\gensy$ is a 
non-negative operator. 

For $\two$, similar arguments provide the following bound:
\begin{equs}
\two=\|(-\gensy)^{1/2} (-\gensy+\Op_{n+2-j})^{-1} (\tilde\fs- \fs)\|\leq \|(-\gensy)^{-1/2}(\tilde\fs- \fs)\|
\end{equs}
so that~\eqref{e:ApproxGenG} follows by putting these two bounds together. 
\end{proof}

We conclude this subsection by showing that, for $n$ fixed and large $N$, 
$\fbNn\equiv \ffNn_1$ and $\fhNn\equiv \ffNn_2$ are indeed good 
approximations of $\tfbNn\equiv \tffNn_1$ and $\tfhNn\equiv \tffNn_2$, respectively. 

\begin{proposition}\label{p:ApproxGenEq}
Let $n\in\N$ be fixed. 
Then, for $a=1,2$, the following limit holds 
\begin{equ}[e:ApproxGenEq]
\lim_{N\to\infty}\|(-\gensy)^{1/2} (\tffNn_a- \ffNn_a)\|=0\,.
\end{equ}
\end{proposition}
\begin{proof}
Notice that it suffices to show that for all $j=a,\dots,n$
\begin{equ}[e:ApproxGenEqj]
\lim_{N\to\infty}\|(-\gensy)^{1/2} (\tffNn_{a,j}- \ffNn_{a,j})\|=0\,.
\end{equ}
We prove the statement for $a=1$ (i.e. $\ffNn_a\equiv\fbNn$), the other case being analogous.
By iteratively applying~\eqref{e:Approx+GenG} in Lemma~\ref{p:RecRepl} we obtain
\begin{equs}
\|(-\gensy)^{1/2} (\tfbNn_j- \fbNn_j)\|\lesssim \|(-\gensy)^{1/2}(\tfbNn_1- \fbNn_1)\|+\eps_N\sum_{r=1}^{j-1}\|(-\gensy)^{1/2}\fbNn_r\|\,.
%\|(-\gensy)^{1/2} (\tfhNn_j- \fhNn_j)\|&\lesssim \|(-\gensy)^{1/2}(\tfhNn_2- \fhNn_2)\|+\eps_N\|(-\gensy)^{1/2}\fhNn_2\|
\end{equs}
At this point, recall that $\tfbNn_1$ and $ \fbNn_1$ are 
defined by the first lines of \eqref{e:GenEqSys} and \eqref{e:Fjn} respectively, with $a=1$. 
We can therefore apply \eqref{e:ApproxGenG} with $\fp= \fbNn_1, \tilde\fp= \tfbNn_1$ and $\fs=\tilde\fs=\fg_1$ 
given in~\eqref{e:InputGenEq} and altogether we get
\begin{eqnarray}
  \label{eq:gnaafacciopiu}
  \|(-\gensy)^{1/2} (\tfbNn_j- \fbNn_j)\|\lesssim\eps_N\left(\|(-\gensy)^{-1/2}\fg_1 \|+\sum_{r=1}^{j-1}\|(-\gensy)^{1/2}\fbNn_r\|\right)\,.
\end{eqnarray}
% By~\eqref{e:ApproxGenG}, we get 
% \begin{equs}
% \|(-\gensy)^{1/2} (\tfbNn_j- \fbNn_j)\|&\lesssim\eps_N\|(-\gensy)^{1/2}\fbNn_i\|\\
% \end{equs}
% and since
By Proposition~\ref{l:PrelimBounds}, $\|(-\gensy)^{1/2}\fbNn_r\|\lesssim 1$ (the constant hidden 
into ``$\lesssim$'' depends on $\|\psi\|_{H^1(\T^2)}$ but this is irrelevant for our purposes) while
$\|(-\gensy)^{-1/2}\fg_1 \|^2=2\nueff\|\psi\|^2_{H^1(\mathbb T^2)}$. Altogether~\eqref{e:ApproxGenEqj} and 
consequently the statement follow. 
\end{proof}

\subsection{Proof of Theorem~\ref{thm:ControlGen} and Proposition~\ref{p:H1Norm}}\label{s:ProofThms}

The goal of this section is to prove Theorem~\ref{thm:ControlGen} and Proposition~\ref{p:H1Norm}. 
We begin with the former. 

\begin{proof}[of Theorem~\ref{thm:ControlGen}]
We focus first on~\eqref{e:bL2} and~\ref{e:hL2}. 
Notice that, since each of the $\fbNn_j$'s lives in the $j$-th homogeneous chaos, 
$\eta(\psi)$ lives in the first and the different chaoses are orthogonal, we have
\begin{equ}
\|\fbNn(\eta)-\eta(\psi)\|^2=\|\fbNn_1(\eta)-\eta(\psi)\|^2+\sum_{j=2}^n\|\fbNn_j\|^2\,.
\end{equ}
Thanks to Proposition~\ref{l:PrelimBounds}, the sum goes to $0$ as $N\to\infty$ (and therefore the same 
holds if we afterwards take the limit in $n$) so that we 
only need to focus on the first term. By definition, 
\begin{equs}
\|\fbNn_1(\eta)-\eta(\psi)&\|^2= % \frac{2\fc+1}{4}
\sum_{k\in\Z^2\setminus\{0\}} |\hat\psi(k)|^2\left|\frac{\nueff}{1+G_{n+1}(\Ll(\tfrac12|k|^2))}-1\right|^2 
\end{equs}
and we want to take first the limit as $N\to\infty$.  Since the term in the
second line is bounded uniformly in $N$ and moreover we have that
$\psi\in H^1(\T^2)$ hence $\psi\in L^2(\T^2)$, we can apply the dominated convergence theorem to
pass the limit inside the sum. For fixed $k$,
$\Ll(\tfrac12|k|^2)$ converges to $\fc$ as $N\to\infty$ and by the continuity of
$G_{n+1}$, we get 
\begin{equ}
\lim_{N\to\infty}\|\fbNn_1(\eta)-\eta(\psi)\|^2=\|\psi\|^2_{L^2(\T^2)}\left|\frac{\nueff}{1+G_{n+1}(\fc)}-1\right|^2\,.
\end{equ}
Hence,~\eqref{e:bL2} follows by taking the limit $n\to\infty$ and applying~\eqref{e:GjConv}. 
On the other hand, by Proposition~\ref{l:PrelimBounds}, for every fixed $n$, $\lim_{N\to\infty}\|\fhNn\|=0$ 
so that~\eqref{e:hL2} clearly holds.
\medskip

We now turn to~\eqref{e:bH1} and~\ref{e:hH1}. Let $\ffNn_a$ be as in~\eqref{e:Fjn}, i.e. $\ffNn_1=\fbNn$ and 
$\ffNn_2=\fhNn$. By construction, $\ffNn_a$ lives in the $n$-th inhomogeneous Wiener-chaos so that 
\begin{equs}
\gen \ffNn_a&= \gen_n\ffNn_a +\genap \ffNn_{a,n} = \gen_n\tffNn_a +\gen_n[\ffNn_a-\tffNn_a]+\genap \ffNn_{a,n}\\
&= -\fg_a+\gen_n[\ffNn_a-\tffNn_a]+\genap \ffNn_{a,n}\,.
\end{equs}
where $\tffNn_a$ is the solution of~\eqref{e:GenEq} for $a=1$ and \eqref{e:GenEq2} for $a=2$. 
Hence, by the above and the triangular inequality we have 
\begin{equ}
\|(-\gensy)^{-1/2}\Big(\gen \ffNn_a+\fg_a\Big)\|\leq \|(-\gensy)^{-1/2}\gen_n[\ffNn_a-\tffNn_a]\|+\|(-\gensy)^{-1/2}\genap \ffNn_{a,n}\|\,.
\end{equ}
We will separately deal with the two summands above, starting from the first. 
Now, by the very definition of $\gen_n$ and the properties of $\gensy$, $\genap$ and $\genam$ 
determined in Lemma~ \ref{lem:generator}, the following equality holds
\begin{equ}
\gen_n[\ffNn_a-\tffNn_a]=\gensy[\ffNn_a-\tffNn_a]+\sum_{j=1}^{n-1}\genap [\ffNn_{a,j}-\tffNn_{a,j}] +\sum_{j=2}^{n}\genam [\ffNn_{a,j}-\tffNn_{a,j}]
\end{equ}
where we used that $\genam\phi=0$ if $\phi$ is in the first chaos. 
The previous, together with~\eqref{b:A+} and~\eqref{b:A-} in Lemma~\ref{l:A+A-}, implies that 
\begin{equs}
\|(-\gensy)^{-1/2}&\gen_n[\ffNn_a-\tffNn_a]\|^2\lesssim \|(-\gensy)^{1/2}[\ffNn_a-\tffNn_a]\|^2\\
&+\sum_{j=1}^{n-1}\|(-\gensy)^{-1/2}\genap [\ffNn_{a,j}-\tffNn_{a,j}]\|^2 +\sum_{j=2}^{n}\|(-\gensy)^{-1/2}\genam [\ffNn_{a,j}-\tffNn_{a,j}]\|^2\\
\lesssim_n&  \|(-\gensy)^{1/2}[\ffNn_a-\tffNn_a]\|^2
\end{equs}
and the right hand side goes to $0$ as $N\to\infty$ by Proposition~\ref{p:ApproxGenEq}. 

We are left with the second summand. Applying once more~\eqref{b:A+}, we have 
\begin{equs}
\|(-\gensy)^{-1/2}\genap\ffNn_{a,n}\|\lesssim_{T,n} \|(-\gensy)^{1/2}\ffNn_{a,n}\|\,.
\end{equs}
By Proposition~\ref{l:H1} and the bound~\eqref{b:UnifG+} in Lemma~\ref{l:G+} we have 
\begin{equs}
&\lim_{N\to\infty} \|(-\gensy)^{1/2}\fbNn_n\|^2=\frac{G^{+,2}_{n-1}(\fc)}{(1+G_{n+1}(\fc))^2}\frac{\nueff^2}{2}\|\psi\|^2_{H^1(\T^2)}\lesssim \frac{\fc^{n-2}}{(n-2)!} \|\psi\|^2_{H^1(\T^2)}\,,\\
&\lim_{N\to\infty} \|(-\gensy)^{1/2}\fhNn_n\|^2=\frac12 G^{+,2}_{n-1}(\fc)\lesssim \frac{\fc^{n-2}}{(n-2)!} \,.
\end{equs}
Since the right hand sides go to $0$ as $n\to\infty$, both~\eqref{e:bH1} and~\eqref{e:hH1}
follow at once. 
\end{proof}

We now turn to the proof of Proposition~\ref{p:H1Norm}. 

\begin{proof}[of Proposition~\ref{p:H1Norm}]
As an immediate consequence of Proposition~\ref{l:H1} and orthogonality of different chaoses, we have the limits 
\begin{equs}
&\lim_{N\to\infty} \|(-\gensy)^{1/2} \fbNn\|^2=\frac{\nueff^2}{2}\|\psi\|^2_{H^1(\T^2)}\frac{\sum_{j=1}^n G^{+,n+2-j}_{j-1}(\fc)}{{(1+G_{n+1}(\fc))^2}}\,,\\%\label{e:limNb}\\
&\lim_{N\to\infty} \|(-\gensy)^{1/2} \fhNn\|^2=\frac12 \sum_{j=2}^n G^{+,n+2-j}_{j-1}(\fc)\,,\\%\label{e:limNh}\\
&\lim_{N\to\infty} \langle (-\gensy)^{1/2} \fbNn, (-\gensy)^{1/2} \fhNn\rangle=0\,,%\label{e:limNbh}
\end{equs}
the last of which directly implies~\eqref{e:limNbh}. 
Hence, it remains to argue the validity of~\eqref{e:limnNb} and~\eqref{e:limnNh}, for which we proceed as follows. 
For $n\in\N$, consider the function $S_n(x)\eqdef \sum_{j=1}^n G^{+,n+2-j}_{j-1}(x)$. 
Thanks to the definition of  $G^{+,n+2-j}_i$ in~\eqref{e:G+} and 
Lemma~\ref{l:G+}~\ref{i:Pos}~\ref{i:Rec} and~\ref{i:Bound}, for every $n$ and $x\ge0$ 
\begin{enumerate}[noitemsep, label=(\roman*)]
\item $S_n(0)= 1$ and $S_n(x)\geq 0$
\item we have the bounds 
\begin{equs}
S_n(x)\leq 1+\sum_{j=2}^n \frac{x^{j-2}}{(j-2)!} \lesssim e^x\quad\text{and}\quad |S_n(x)'|\vee |S_n(x)''|\lesssim e^x\,,
\end{equs}
\item the following relation holds
\begin{equs}
S_{n+1}(x)'&= \sum_{j=2}^{n+1} \frac{G^{+,n+3-j}_{j-2}(x)}{[1+G_{n+1}(x)]^2}= \frac{\sum_{j=1}^{n}G^{+,n+2-j}_{j-1}(x)}{[1+G_{n+1}(x)]^2}=\frac{S_n(x)}{[1+G_{n+1}(x)]^2}
\end{equs}
where we used that $G^{+,n+3-j}_{0}\equiv1$.
\end{enumerate}
Arguing as in the proof of Theorem~\ref{thm:Dbulk} and using the fact, showed therein, 
that $G_j$ converges uniformly to $G$ given in~\eqref{e:LimG}, we see that $S_n$ and $S_n'$
converge uniformly on every compact set to $S$ and $S'$, for $S$ the unique solution to the Cauchy-value problem
\begin{equ}
{S}'=\frac{S}{[1+G]^2}\,,\qquad S(0)=1\,.
\end{equ}
$S$ is explicit and given by
\begin{equs}
S(x)&=\exp\Big[\int_0^x \frac{\dd y}{[1+G(y)]^2}\Big]=\exp\Big[\int_0^x \frac{\dd y}{2y+1}\Big]%% \\
%% &=\exp\Big[\tfrac12\log(2x+1)\Big]
=\sqrt{2x+1}\,.
\end{equs}
Therefore, 
\begin{equ}
\lim_{n\to\infty} S_n(\fc)=\lim_{n\to\infty} \sum_{j=1}^n G^{+,n+2-j}_{j-1}(\fc)= \sqrt{2\fc+1}=\nueff\,.
\end{equ}
Now, since $G(\fc)=\sqrt{2\fc+1}-1$, we immediately deduce that 
\begin{equ}
\lim_{n\to\infty}\lim_{N\to\infty} \|(-\gensy)^{1/2} \fbNn\|^2=\frac{\nueff^2}{2}\|\psi\|^2_{H^1(\T^2)}\lim_{n\to\infty}\frac{S_n(\fc)}{{(1+G_{n+1}(\fc))^2}}=\frac{\nueff}{2}
\end{equ}
so that~\eqref{e:limnNb} holds, and 
\begin{equ}
\lim_{n\to\infty}\lim_{N\to\infty} \|(-\gensy)^{1/2} \fhNn\|^2=\lim_{n\to\infty}\frac{ S_n(\fc)-1}{2}=\frac{\nueff-1}{2}
\end{equ}
which concludes the proof. 
\end{proof}

\section{The convergence}
\label{s:theconvergence}

The goal of this section is to prove the main result of the present paper, namely Theorem~\ref{thm:Conv}, 
but first we need to recall some preliminary facts. 

For $N\in\N$, let $h^N$ be defined as in~\eqref{e:AKPZ}. Recall that $u^N\eqdef (-\Delta)^{1/2}h^N$ 
and $\hat h^N(0)$ respectively solve~\eqref{e:AKPZ:u} and~\eqref{eq:modozero}.  
%Let $n\in\N$. %  $\fbNn$ and $\fhNn$ be given by~\eqref{e:Fjn} with $a=1$ and $a=2$ respectively, 
% and let $\ffNn$ be either of them. 
% Since $\ffNn$ is a cylinder function (both $\fbNn$ and $\fhNn$ are polynomials), 
 It\^o's formula ensures that for $T>0$ and all $t\in [0,T]$ the following equality holds for any smooth function $\ff$, say with a finite chaos expansion,
\begin{equs}[e:Ito]
\ff(u^N_t)-\ff(\eta)-\int_0^t \gen \ff(u^N_s)\dd s=\CM_t(\ff)
\end{equs}
where  $\CM_t(\ff)$ is the martingale given by 
\begin{equ}[e:MartN]
\CM_t(\ff)= \int_0^t \sum_{k\in\Z^2\setminus\{0\}} |k| \, D_k\ff(u^N_s)\,\dd B_s(k)
\end{equ}
whose quadratic variation is 
\begin{equ}[e:QMartN]
\langle\CM_\cdot(\ff)\rangle_t=  \int_0^t \sum_{k\in\Z^2\setminus\{0\}} |k|^2 \, \big| D_k\ff(u^N_s)\big|^2\dd s\,.
\end{equ}
We will apply these formulas especially in the case $\ff=\ffNn_a$, $a=1,2$, 
where the latter is defined according to~\eqref{e:Fjn}. 
We start by studying the behaviour of the martingale in~\eqref{e:MartN} and its quadratic variation in~\eqref{e:QMartN} 
in these cases. 

\subsection{The martingales}\label{sec:Mart}

In this section, we focus on the martingale~\eqref{e:MartN}. 
Our goal is to prove the following theorem which on the one hand 
shows that, as $N\to\infty$, its quadratic variation in~\eqref{e:QMartN}
converges in mean square to a deterministic function of time, while on the other determines 
the limit as first $N$ and then $n\to\infty$ of such a function. 

\begin{theorem}\label{thm:Mart}
Let $\Psi\in\CC^\infty(\T^2)$ and define $\psi_0$ and $\psi$ according to~\eqref{e:Psi}.
Let $\ffNn\eqdef\fbNn+\psi_0\fhNn=\ffNn_1+\psi_0\ffNn_2 $ and
$\CM_t(\ffNn)$ be the martingale defined according to~\eqref{e:MartN} 
with quadratic variation $\langle \CM_\cdot(\ffNn)\rangle_t$ given in~\eqref{e:QMartN}. 
Then, for all $t\geq 0$ the following holds
\begin{equ}[e:Var]
\lim_{N\to\infty} {\rm Var}(\langle \CM_\cdot(\ffNn)\rangle_t)=0\,
\end{equ}
and
\begin{equs}[e:ExpQMart]
  \lim_{n\to\infty}\lim_{N\to\infty} \Exp[\langle \CM_\cdot(\ffNn)\rangle_t]=t\times
  \Big(\nueff\|\Psi\|_{L^2(\T^2)}^2-\psi_0^2\Big).
% \begin{cases}
% \nueff\|\psi\|_{H^1(\T^2)}^2\,, & \text{if $\ffNn=\fbNn$\,,}\\
% \nueff-1\,, & \text{if $\ffNn=\fhNn$\,,}\\
% \nueff\|\Psi\|_{L^2(\T^2)}^2-\psi_0^2\,, & \text{if $\ffNn=\fbNn+\psi_0\fhNn$\,.}
% \end{cases}
\end{equs}
\end{theorem}
\begin{remark}
  Note that, because of $\nueff\ge1$, the r.h.s. of \eqref{e:ExpQMart} is indeed positive.
  The reason why the quadratic variation has the extra factor $-t\psi_0^2$ with respect to
  that of   $\cM^\Psi_t$  (see \eqref{e:Mart2}) is, roughly speaking, that the latter
  also contains a term $\psi_0B_t(0)$, which is the independent standard Brownian motion appearing in the equation \eqref{eq:modozero} of the zero mode of $h^N$. This issue is further expanded in Section \ref{sec:ProofMain} below.
\end{remark}

We begin by proving~\eqref{e:ExpQMart}, which is direct consequence of Proposition~\ref{p:H1Norm}.
The proof of~\eqref{e:Var} is much more involved and is deferred to Section \ref{sec:Feynman}.
\begin{proof}[of~\eqref{e:ExpQMart}]
%We first prove~\eqref{e:ExpQMart}. 
Notice that by~\eqref{e:QMartN} and the definition~\eqref{e:Energy} of the energy $\energy$, we have 
\begin{equs}
 \Exp[\langle \CM_\cdot(\ffNn)\rangle_t]&%=\Exp\Big[\int_0^t \sum_{k\in\Z^2} |k|^2 \, \big| D_k\ffNn(u^N_s)\big|^2\dd s\Big]
 =\int_0^t \Exp\Big[\sum_{k\in\Z^2} |k|^2 \, \big| D_k\ffNn(u^N_s)\big|^2\Big]\dd s=t\, \E[\energy(\ffNn)]\\
 &=2 t \|(-\gensy)^{1/2}\ffNn\|^2
\end{equs}
where in the last step we applied~\eqref{e:EnergyMart}. 
% Now, if $\ffNn=\fbNn$ or $\fhNn$ the result follows respectively by~\eqref{e:limnNb} and~\eqref{e:limnNh} 
% in Proposition~\ref{p:H1}. Hence, we are left to consider the case of $\ffNn=\fbNn+\psi_0\fhNn$. 
Notice that
\begin{equs}
\|(-\gensy)^{1/2}(\fbNn+\psi_0\fhNn)\|^2=&\|(-\gensy)^{1/2}\fbNn\|^2+\psi_0^2\|(-\gensy)^{1/2}\fhNn\|^2\\
&+2\psi_0\langle (-\gensy)^{1/2}\fbNn, (-\gensy)^{1/2}\fhNn\rangle\,.
\end{equs}
The double limit of each of the summands at the right hand side was determined in~\eqref{e:limnNb},~\eqref{e:limnNh} 
and~\eqref{e:limNbh} in Proposition~\ref{p:H1Norm} and gives 
exactly $\half(\nueff\|\Psi\|^2_{L^2(\T^2)}-\psi_0^2)$ (see ~\eqref{e:L2H1}), so that~\eqref{e:ExpQMart} follows at once.
\end{proof}
%
%\fabioText{I don't find this remark very useful. I shortened, but I'd even remove}
%\begin{remark}
%The proof of the previous proposition %% reveals more about the nature of the martingale
%%% $\CM_t(\ffNn)$, for $\ffNn=\fbNn+\psi_0\fhNn$. 
%%% Indeed, by~\eqref{e:Var},
%shows that
%the quadratic variation of $\CM_t(\ffNn)$ is converging (in probability) 
%in the limit $N\to\infty$ to $a_n t$, where the value of $a_n$ can be deduced
%%% as in the proof of Proposition~\ref{e:ExpQMart}, 
%from the right hand sides of~\eqref{e:limNb},~\eqref{e:limNh} and~\eqref{e:limNbh}. 
%Hence, by~\cite[Theorem 7.1.4]{EK}, {\it for each fixed $n\in\N$}
%$\CM_t(\ffNn)$ converges in law to a Brownian motion with an explicit quadratic variation 
%so that we recover the full space-time white noise term of~\eqref{e:SHEakpz}  as the $n\to\infty$ limit 
%of a sequence of Gaussian noises. 
%\end{remark}
%
%
%

\subsection{Proof of Theorem~\ref{thm:Conv}}\label{sec:ProofMain}

Before turning to the proof of Theorem~\ref{thm:Conv}, 
let us give two statements which are immediate consequences of Theorem~\ref{thm:ControlGen} 
and the It\^o trick, Lemma~\ref{l:ItoTrick}. 

\begin{proposition}\label{p:Rewriting}
For $N\in\N$, let $u^N$ be the unique stationary solution to~\eqref{e:AKPZ:u}. Let $n\in\N$ and 
$\ffNn_a,a=1,2$  be defined according to~\eqref{e:Fjn}. 
Let $\fg\in\{\fg_1,\,\fg_2\}$, where the $\fg_a$'s, $a=1,2$, be those in~\eqref{e:InputGenEq}. 
Then, the process $\CR_t^{N,n}(u^N)$ given by 
\begin{equs}[e:Rewriting]
\CR_t^{\ffNn_a}(u^N)\eqdef \int_0^t \gen \ffNn_a(u^N_s)\dd s+\int_0^t \fg_a(u^N_s)\dd s 
\end{equs}
is such that for all $p\geq 2$ and all $T>0$
%given by
%\begin{equ}[e:Rem]
%\CR_t^{N,n}(u^N)\eqdef \int_0^t \gen_n[\ffNn-\tffNn](u^N_s)\dd s+\int_0^t \genap \ffNn_n(u^N_s)\dd s
%\end{equ}
%and is such that for all $p\geq 2$
\begin{equ}[e:RemLim]
\lim_{n\to\infty} \limsup_{N\to\infty} \Exp\Big[\sup_{t\in[0,T]} \Big|\CR_t^{\ffNn_a}(u^N)\Big|^p\Big]^{\frac{1}{p}}=0\,.
\end{equ}
%In particular, the zero-th Fourier mode of $h^N$ is given by
%\begin{equ}[e:zeromode]
%h^N_t(e_0)-\chi(e_0)=\,.
%\end{equ}
\end{proposition}
\begin{proof}
Notice that by the It\^o trick, in particular~\eqref{e:ItoTrick}, we have 
\begin{equ}
\Exp\Big[\sup_{t\in[0,T]} \Big|\CR_t^{\ffNn_a}(u^N)\Big|^p\Big]^{\frac{1}{p}}\lesssim \|(-\gensy)^{-1/2}\Big(\gen \ffNn_a+\fg_a\Big)\|
\end{equ}
which, by~\eqref{e:bH1} and~\eqref{e:hH1} in Theorem~\ref{thm:ControlGen}, converges to $0$ in the limit 
$N\to\infty$ first and $n\to\infty$ after. 
\end{proof}

The following immediate corollary shows 
how to rewrite the $0$-th Fourier mode of $h^N$ in terms of $u^N$ and the observable $\fhNn$. 

\begin{corollary}\label{c:zeromode}
For $N,\,n\in\N$, let $\fhNn$ be defined according to~\eqref{e:Fjn} with $\fg_2=\nf_0$. 
Then, with  $\CR^{\fhNn}$ given by~\eqref{e:Rewriting} and $\CM_t(\fhNn)$ the martingale defined in \eqref{e:Ito},
\begin{equ}[e:Newzeromode]
\hat h^N_t(0)-\hat\chi(0)=\fhNn(\eta)-\fhNn(u^N_t)+\CR^{\fhNn}_t(u^N)+\CM_t(\fhNn)+B_t(0)\,.
\end{equ}
\end{corollary}
\begin{proof}
By~\eqref{eq:modozero}, we have 
\begin{equ}
\hat h^N_t(0)-\hat\chi(0)=\int_0^t\CN^N_0[u_s^N]\dd s +B_t(0)=\int_0^t\nf_0(u_s^N)\dd s +B_t(0)\,.
\end{equ}
Applying~\eqref{e:Ito} to $\ffNn_2=\fhNn$ and exploiting~\eqref{e:Rewriting} we obtain
\begin{equ}
\int_0^t\nf_0(u_s^N)\dd s = \fhNn(\eta)-\fhNn(u^N_t)+\CR^{\fhNn}_t(u^N)+\CM_t(\fhNn)
\end{equ}
from which the statement follows at once. 
\end{proof}

We are now ready for the proof of Theorem~\ref{thm:Conv}. 

\begin{proof}[of Theorem~\ref{thm:Conv}]
Let $T>0$ and $h$ a limit point for $h^N$ in $C([0,T],\cD'(\T^2))$, which exists in view of Theorem~\ref{thm:Tightness}. 
To establish the theorem, we will prove that $h$ satisfies the Martingale Problem in Definition~\ref{def:MPSHE} 
for which we need to verify that the processes $\cM^\Psi$ and $\Gamma^\Psi$ 
defined in~\eqref{e:Mart1} and~\eqref{e:Mart2} respectively, are (local) martingales. 
The latter in turn follows, once we show that for all $0\leq s\leq t\leq T$, all $\Psi\in\CC^\infty(\T^2)$
and all bounded continuous function $G\colon C([0,s],\cD'(\T^2))\to\R$ we have 
\begin{equs}[e:Goal]
\Exp\Big[\Big(\delta_{s,t}\cM^\Psi_\cdot\Big) G(h\restr_{[0,s]})\Big]=0=\Exp\Big[\Big(\delta_{s,t}\Gamma^\Psi_\cdot\Big) G(h\restr_{[0,s]})\Big]\,,
\end{equs}
where from here onwards, to compactify the notations, 
we will denote the time increment of a generic function $f$ by $\delta_{s,t}f_{\cdot}\eqdef f_t-f_s$ 
and its restriction to an interval $[0,s]$ by $f\restr_{[0,s]}$. 
Let us begin with the first equality, which holds if 
\begin{equs}[e:Claim1]
\Exp&\Big[\Big(\delta_{s,t}\cM^\Psi_\cdot\Big) G(h\restr_{[0,s]})\Big]\\%\label{e:Claim1}\\
&=\lim_{n\to\infty}\lim_{N\to\infty} \Exp\Big[\Big(\delta_{s,t}\CM_\cdot(\fbNn)+\psi_0\delta_{s,t}\CM_\cdot(\fhNn)+\psi_0\delta_{s,t}B_\cdot(0)\Big) G(h^N\restr_{[0,s]})\Big]\,,
\end{equs}
where $\fbNn$, $\fhNn$ are the observables defined according to~\eqref{e:Fjn} with $a=1$ and $a=2$ respectively.
Indeed, since $B(0)$ is a Brownian motion and both $\CM(\fbNn)$ and $\CM(\fhNn)$ are martingales, 
the expectation at the right hand side equals $0$ for all $n,\,N\in\N$. 
\medskip

Set as usual $u\eqdef (-\Delta)^{1/2}h$ and, with the notations introduced in~\eqref{e:Psi},~\eqref{e:hu} gives
\begin{equs}
\delta_{s,t}\cM^\Psi_\cdot&=\delta_{s,t}h_\cdot(\Psi)-\frac{\nueff}{2}\int_s^t h_r(\Delta\Psi)\dd r=\delta_{s,t}u_\cdot(\psi)-\int_s^t \gensy^{\rm eff}u_r(\psi)\dd r +\psi_0\, \delta_{s,t}\hat h_\cdot(0)\,.
\end{equs}
%Trivially, the left hand side of~\eqref{e:Goal} equals
%\begin{equs}[e:Subshhn]
%%\Exp&\Big[\Big(\delta_{s,t}\cM^\Psi_\cdot\Big) G(h\restr_{[0,s]})\Big]=\\
%&\Exp\Big[\Big(\delta_{s,t}u_\cdot(\psi)-\int_s^t \gensy^{\rm eff}u_r(\psi)\dd r+\psi_0\delta_{s,t}\hat h_\cdot(0)\Big) G(h\restr_{[0,s]})\Big]\\
%&=\lim_{n\to\infty}\Exp\Big[\Big(\delta_{s,t}u_\cdot(\psi)-\int_s^t \gensy^{\rm eff}u_r(\psi)\dd r+\psi_0\delta_{s,t}\hat h_\cdot(0)\Big) G(h\restr_{[0,s]})\Big]\,.
%\end{equs}
We now want to replace $u$ and $h$ in the expectation with $u^N$ and $h^N$. 
To do so we proceed as in the proof of~\cite[Theorem 4.6]{GPGen}. 
Note that if $F_1,\,F_2\in L^2(\eta)$ and $F_2$ is a bounded cylinder function, then
\begin{equs}
\limsup_{N\to\infty}&\Big|\Exp\big[ F_1(u_r) G(h\restr_{[0,s]})\big]-\Exp\big[ F_1(u^N_r) G(h^N\restr_{[0,s]})\big]\Big|\\
\lesssim&\, \Exp\big[ |F_1(u_r)-F_2(u_r)|\big]+ \limsup_{N\to\infty}\Exp\big[ |F_1(u^N_r)-F_2(u^N_r)|\big]\\
&\qquad+\limsup_{N\to\infty}\Big|\Exp\big[ F_2(u_r) G(h\restr_{[0,s]})\big]-\Exp\big[ F_2(u^N_r) G(h^N\restr_{[0,s]})\big]\Big|\label{e:SubshhN}\\
\lesssim&\, \Exp\big[ |F_1(u_r)-F_2(u_r)|^2\big]^{\frac12}+ \limsup_{N\to\infty}\Exp\big[ |F_1(u^N_r)-F_2(u^N_r)|^2\big]^{\frac12}=2\|F_1-F_2\|
\end{equs}
where $r$ is either $s$ or $t$, in the first step we used that $G$ is bounded and, 
in the passage from the second to the third inequality, that $h^N$ converges in law to $h$ to conclude that 
the third summand converges to $0$. We also employed the fact that the law of $u^N_r$ for fixed $r$ is independent of $r,N$.
Since we can approximate arbitrarily well general elements in $L^2(\eta)$ with bounded cylinder functions, 
it follows that the right hand side of~\eqref{e:SubshhN} is arbitrarily small and therefore the left hand side is zero. 
Hence, the left hand side of~\eqref{e:Claim1} equals
\begin{equ}[e:Claim1new]
\lim_{N\to\infty}\Exp\Big[\Big(\delta_{s,t}u^N_\cdot(\psi)-%&
\int_s^t \gensy^{\rm eff}u^N_r(\psi)\dd r\Big) G(h^N\restr_{[0,s]})\Big]+\psi_0\lim_{N\to\infty}\Exp\Big[\delta_{s,t}\hat h^N_\cdot(0) G(h^N\restr_{[0,s]})\Big]%\\
%&+\psi_0 \lim_{n\to\infty}\lim_{N\to\infty}\Exp\big[\big(\delta_{s,t}\hat h^N_\cdot(0)\big) G(h^N\restr_{[0,s]})\big]
\end{equ}
and we will separately consider the two summands. 
For the first, we need to substitute $u^N_\cdot(\psi)$ with $\fbNn$ and $\gensy^{\rm eff}u^N_r(\psi)$ with $\gen \fbNn(u^N)$. 
To do so, Cauchy-Schwarz inequality and boundedness of $G$ give 
\begin{equ}[e:CS]
\left|\Exp\Big[\big(u^N_r(\psi)-\fbNn(u^N_r)\big)G(h^N\restr_{[0,s]})\Big]\right|\lesssim \|\eta(\psi)-\fbNn(\eta)\|\,, \qquad r\in\{s,t\}\,.
\end{equ}
and by~\eqref{e:bL2} in Theorem~\ref{thm:ControlGen}, the right hand side converges to $0$ as $N\to\infty$ first 
and $n\to\infty$ after. 
On the other hand,~\eqref{e:Rewriting} applied to $\fbNn$ implies
\begin{equ}
\int_s^t \gensy^{\rm eff}u^N_r(\psi)\dd r=-\int_s^t \fg_1(u^N_r)\dd r =\int_s^t \gen\fbNn(u^N_r)\dd r-\delta_{s,t}\CR^{\fbNn}_\cdot(u^N)
\end{equ}
where $\fg_1$ is defined in~\eqref{e:InputGenEq} and, arguing as in~\eqref{e:CS}, 
we have 
\begin{equ}[e:R]
\left|\Exp\Big[\CR^{\fbNn}_r(u^N)\,G(h^N\restr_{[0,s]})\Big]\right|\lesssim  \Exp\Big[\sup_{t\in[0,T]} \Big|\CR_t^{\fbNn}(u^N)\Big|^2\Big]^{\frac{1}{2}}\,,\qquad \text{for all $r\in[0,t]$}
\end{equ}
and, by~\eqref{e:RemLim}, the right hand side converges to $0$ if we let $N\to\infty$ and then $n\to\infty$.
Hence, the first summand of~\eqref{e:Claim1new} equals 
\begin{equ}
\lim_{N\to\infty}\Exp\Big[\Big(\delta_{s,t}\fbNn(u^N_\cdot)-\int_s^t \gen\fbNn(u^N_r)\dd r\Big) G(h^N\restr_{[0,s]})\Big]
\end{equ}
up to an error term which is controlled by the right hand sides of~\eqref{e:CS} and~\eqref{e:R}. 
Since further the first summand of~\eqref{e:Claim1new} is independent of $n$, 
we can take the limit for $n\to\infty$ and, by It\^o's formula~\eqref{e:Ito} applied to $\fbNn$, we obtain
\begin{equs}
\lim_{N\to\infty}\Exp\Big[\Big(\delta_{s,t}u^N_\cdot(\psi)-&\int_s^t \gensy^{\rm eff}u^N_r(\psi)\dd r\Big) G(h^N\restr_{[0,s]})\Big]\\
&=\lim_{n\to\infty}\lim_{N\to\infty}\Exp\Big[\Big(\delta_{s,t}\cM_\cdot(\fbNn)\Big) G(h^N\restr_{[0,s]})\Big]=0\,.
\end{equs}

It remains to treat the second summand in~\eqref{e:Claim1new}. 
Setting $\psi_0$ aside, thanks to Corollary~\ref{c:zeromode}, it equals
\begin{equs}
\lim_{n\to\infty}\lim_{N\to\infty}\Exp\Big[\Big(-\delta_{s,t}\fhNn(u^N_\cdot)+\delta_{s,t}\CR^{\fhNn}_\cdot(u^N)+\delta_{s,t}\CM_t(\fhNn)+\delta_{s,t}B_\cdot(0)\Big) G(h^N\restr_{[0,s]})\Big]\,.
\end{equs}
Now, the term containing the increment of $\CR^{\fhNn}_\cdot(u^N)$ can be treated as in~\eqref{e:R} and therefore 
does not contribute. The same holds for the first summand, since we can argue as in~\eqref{e:CS} 
but using~\eqref{e:hL2} in Theorem~\ref{thm:ControlGen}. 
Hence,~\eqref{e:Claim1} follows and so does the first equality in~\eqref{e:Goal}. 
\medskip

We now turn to $\Gamma^\Psi$, for which, using arguments similar to those exploited above, we obtain 
\begin{equs}[e:LimQua]
\Exp&\Big[\Big(\delta_{s,t}\Gamma^\Psi_\cdot\Big) G(h\restr_{[0,s]})\Big]\\
&=\lim_{n\to\infty}\lim_{N\to\infty}\Exp\Big[\Big(\delta_{s,t}(\CM^{N,n}_\cdot)^2 -\nueff(t-s)\|\Psi\|_{L^2(\T^2)}^2\Big) G(h^N\restr_{[0,s]})\Big]\\
&=\lim_{n\to\infty}\lim_{N\to\infty}\Exp\Big[\Big(\delta_{s,t}\langle\CM^{N,n}_\cdot\rangle_\cdot -\nueff(t-s)\|\Psi\|_{L^2(\T^2)}^2\Big) G(h^N\restr_{[0,s]})\Big]
\end{equs}
where we have set 
\begin{equ}
\CM^{N,n}_t\eqdef \CM_t(\fbNn)+\psi_0(\CM_t(\fhNn)+B_t(0))= \CM_t(\fbNn+\psi_0 \fhNn)+\psi_0 B_t(0)
\end{equ}
and, in the last step, used that, by definition of quadratic variation, 
$(\CM^{N,n}_\cdot)^2- \langle\CM^{N,n}_\cdot\rangle_\cdot$ is a martingale. 
Notice that by construction both $\{u_\cdot^N(k)\}_{k\neq 0}$ and $\{B(k)\}_{k\neq 0}$ are independent of 
$B_\cdot(0)$ so that the quadratic variation of $\CM^{N,n}$ is given by 
\begin{equ}[e:FinalQV]
\langle\CM^{N,n}_\cdot\rangle_t=\langle \CM_\cdot(\fbNn+\psi_0 \fhNn)\rangle_t + \psi_0^2\, t\,,\qquad \text{for all $t\geq 0$.}
\end{equ}
By~\eqref{e:Var} in Theorem~\ref{thm:Mart}, we can replace the quadratic variation 
by its expectation, so that the limit in~\eqref{e:LimQua} equals
\begin{equ}
\Exp\Big[G(h\restr_{[0,s]})\Big]\lim_{n\to\infty}\lim_{N\to\infty}\Big(\Exp\Big[\langle\CM^{N,n}_\cdot\rangle_t-\langle\CM^{N,n}_\cdot\rangle_s\Big] -\nueff(t-s)\|\Psi\|_{L^2(\T^2)}^2\Big)=0
\end{equ}
and the last equality follows by~\eqref{e:ExpQMart} and~\eqref{e:FinalQV}. 
Hence, both equalities in~\eqref{e:Goal} hold and the proof in concluded. 
\end{proof}

\subsection{The variance of the quadratic variation}
\label{sec:Feynman}

 Notice at first that, using \eqref{e:QMartN}, we 
can write the variance of the quadratic variation as 
\begin{equs}[e:Var1bound]
{\rm Var}(\langle \CM_\cdot(\ffNn)\rangle_t)=\Exp\left[\Big(\int_0^t \fFNn(u^N_s)\dd s\Big)^2\right]^{1/2}\lesssim t^{1/2} \|(-\gensy)^{-1/2}\fFNn\|
\end{equs}
where the last inequality follows by~\eqref{e:ItoTrick}, $\fFNn$ is the cylinder function given by
\begin{equ}[e:GenericF]
\fFNn(\eta)\eqdef \sum_{k\in\Z^2}|k|^2\left(|D_k\ffNn(\eta)|^2-\E\left[|D_k\ffNn(\eta)|^2\right]\right) %= \sum_{j_1,j_2=1}^n\fFNn_{j_1,j_2}(\eta)\,.
\end{equ}
and the Malliavin derivative $D_k$ was defined in \eqref{e:Malliavin}.
Since $\ffNn=\sum_{j=1}^{n} \ffNn_j$, $\fFNn(\eta)$ can be written as 
\begin{equs}
  \label{e:cac}
\fFNn(\eta)&=\sum_{j_1,j_2=1}^n\fFNn_{j_1,j_2}(\eta)\eqdef\sum_{j_1,j_2=1}^n\sum_{k\in\mathbb Z^2}|k|^2\fFNn_{j_1,j_2;k}(\eta)\\&\eqdef \sum_{j_1,j_2=1}^n\sum_{k\in\Z^2}|k|^2 \left(D_k\ffNn_{j_1}(\eta)D_{-k}\ffNn_{j_2}(\eta)-\E\left[D_k\ffNn_{j_1}(\eta)D_{-k}\ffNn_{j_2}(\eta)\right]\right)
%&=\sum_{j_1,j_2=1}^n\fFNn_{j_1,j_2}(\eta)\,,
\end{equs}
where we used the fact that $(D_k f(\eta))^*=D_{-k}f(\eta)$, if $f$ is real-valued.
In view of~\eqref{e:Var1bound},~\eqref{e:Var} follows once we prove that 
$ \lim_{N\to\infty}\|(-\gensy)^{-1/2}\fFNn\|=0$ at fixed $n$. Hence, it is sufficient to show
\begin{equ}
  \lim_{N\to\infty}\|(-\gensy)^{-1/2}\fFNn_{j_1,j_2}\|=0
\label{eq:dd}
\end{equ}
separately for every $j_1,j_2\le n$.

Up to this point, we have not used the definition of $\ffNn_{j}$ and it is time to do so.
Recall that $\ffNn=\ffNn_1+\psi_0\ffNn_2$ with $\ffNn_a$  defined as in \eqref{e:Fjn}, with $a=1,2$. As a consequence, to prove \eqref{eq:dd} it is enough to show
\begin{equ}
  \lim_{N\to\infty}\|(-\gensy)^{-1/2}\fFNn_{a_1,j_1;a_2,j_2}\|=0, \qquad \fFNn_{a_1,j_1;a_2,j_2}=\sum_k |k|^2\fFNn_{a_1,j_1;a_2,j_2;k}
\label{eq:ddaa}
\end{equ}
for every $a_1,a_2\in\{1,2\}$, where $\fFNn_{a_1,j_1;a_2,j_2}$ and $\fFNn_{a_1,j_1;a_2,j_2;k}$ 
are defined similarly to $\fFNn_{j_1,j_2},\fFNn_{j_1,j_2;k}$ in~\eqref{e:cac}, 
with the only difference that $\ffNn_{j_1},\ffNn_{j_2}$ are replaced by $\ffNn_{a_1,j_1},\ffNn_{a_2,j_2}$. 
Recall from \eqref{e:Fjn} that $\ffNn_{a,j}$ is zero unless $a\le j\le n$, so that
\eqref{eq:ddaa} has to be checked only for $a_1\le j_1\le n, a_2\le j_2\le n$.
\medskip

Before proceeding, we need to recall a couple of useful properties 
about Wick polynomials $:\hat\eta(p_1)\dots\hat \eta(p_n):$ 
in the specific case where $\eta$ is a white noise, 
so that \eqref{e:NoiseFourier} holds (see \cite[Chap.  3]{janson1997gaussian}). 
Let us begin with the following definition. 

\begin{definition}
  \label{def:combinatoria}
  For $n,\,m\in\N$, we will denote by $\mathcal G[n,m]$ the set of collections $G$ of
      disjoint edges connecting a point in $V^n_1\eqdef\{(1,i)\colon i=1,\dots,n\}$ to one in
      $V^m_2\eqdef\{(2,j)\colon j=1,\dots,m\}$ and by $V[G]$ the subset of vertices in $V^{n,m}_{1,2}\eqdef V_1\cup V_2$ 
      which are not endpoints of any edge in $G$.
\end{definition}

If a random variable $f\in \SH_n,n\ge 1$ is written as
\begin{equ}
  f(\eta)=\sum_{p_{1:n}}\hat f_n(p_{1:n}):\hat\eta(p_1)\dots\hat \eta(p_n):
\end{equ}
with a symmetric kernel $\hat f$ and $D_k$ is the Malliavin derivative \eqref{e:Malliavin}, then 
\begin{equ}[e:Mallia]
  D_k f(\eta)=n\sum_{p_{1:n-1}}\hat f_n(p_{1:n-1},k) :\hat\eta(p_1)\dots\hat \eta(p_{n-1}):.
\end{equ}
Further, with the conventions of Definition \ref{def:combinatoria}, the product of two Wick's polynomials 
is given by \cite[Th. 3.15]{janson1997gaussian}
\begin{equ}[e:Wick]
 \Wick{\prod_{v\in V^n_1}\hat\eta(p_v)\,} \Wick{\prod_{v\in V^m_2}\hat\eta(p_{v})\,}
	=\sum_{G\in \mathcal G[n,m]}\prod_{(a,b)\in G}\mathds{1}_{p_{a}=-p_{b}}\Wick{\prod_{v\in V[G]}\hat\eta(p_{v})}\,.
\end{equ}
Since Wick products are centered,  the expectation of
the left hand side of \eqref{e:Wick} consists of the sum of the terms
such that $V[G]=\emptyset$, which is possible only if $n=m$.
\medskip

We now get back to 
\begin{equ}
\fFN_{k}(\eta)\eqdef \fFN_{a_1,j_1;a_2,j_2;k}
\end{equ}
in~\eqref{eq:ddaa} where we dropped the indices $n,j_1,j_2,a_1,a_2$ of $\fF_{a_1,j_1;a_2,j_2;k}^{N,n}$
since these arguments are fixed, and we will keep only the dependence on $k$, which is summed over, 
and $N$ over which we take the limit.
\medskip

Using the properties of Wick products introduced above, one can write 
\begin{equs}[e:FkG]
  \fFN_{k}(\eta)=j_1 j_2\sum_{\substack{G\in \mathcal G[j_1-1,j_2-1]:\\ V[G]\ne\emptyset}}\fFN_{k;G}(\eta)\eqdef& j_1 j_2\sum_{\substack{G\in \mathcal G[j_1-1,j_2-1]:\\ V[G]\ne\emptyset}}\sum_{\substack{\ell_v\in\Z^2\setminus\{0\}\\ v\in V_{1,2}}}\prod_{(a,b)\in G}\mathds{1}_{\ell_a=-\ell_b} \\
  &\quad\times {\hffNn}_{a_1,j_1}(\ell|_{V_1},k)\hffNn_{a_2,j_2}(\ell|_{ V_2},-k)
   \Wick{\prod_{v\in V[G]}\hat\eta(\ell_v)}\,,
 \end{equs}
where we adopted the notation $\ell|_{ V_1}\eqdef (\ell_v)_{v\in V_1}$ and 
we set $V_1\equiv V_1^{j_1-1}, V_2\equiv V_2^{j_2-1}, V_{1,2}\equiv V_{1,2}^{j_1-1,j_2-1}$. 
Since the cardinality of $\mathcal G[j_1-1,j_2-1]$ does not depend on  $N$, to prove
 \eqref{eq:dd} it is enough to prove
\begin{equ}
  \lim_{N\to\infty}\|(-\gensy)^{-1/2}\fFN_{G}\|=0\,,\qquad   \fFN_{G}\eqdef \sum_{k\in \mathbb Z^2}|k|^2\fFN_{k;G}
  \label{eq:dd2}
\end{equ}
for every $G\in \mathcal G[j_1-1,j_2-1],V[G]\ne\emptyset$. 

In order to see that~\eqref{eq:dd2} holds true, we need to exploit the specific form of the $\hffNn$ and 
are therefore forced to distinguish two cases: 
either both $j_x,x=1,2$ are strictly larger than $a_x$, or at least one of them equals $a_x$. 
We will  start with the former, which requires the most work.

\subsubsection{The case $j_1>a_1,j_2> a_2$}
\label{sec:case>}
Fix $a=1,2$ and $a+1\le j\le n$. Then, with $S_j$ the set of permutations of $j$ elements,
\begin{equs}[e:dnm]
\ffNn_{a,j}&= (-\gensy[1+G_{n+2-j}(\Ll(-\gensy))])^{-1}\genap \ffNn_{a,j-1}\\
&=\frac{\hat\lambda}{\sqrt{\log N}}\sum_{\ell_{1:j}\in\mathbb Z^{2j}\setminus\{0\}}\frac{2}{|\ell_{1:j}|^2[1+G_{n+2-j}(\Ll(|\ell_{1:j}|^2/2))]}\\
&\quad\times\frac{1}{j!}\sum_{\sigma\in S_j}\nonlin_{\ell_{\sigma(1)},\ell_{\sigma(2)}}|\ell_{\sigma(1)}+\ell_{\sigma(2)}|
\hat{\ff}^{N,n}_{a,j-1}(\ell_{\sigma(1)}+\ell_{\sigma(2)},\ell_{\sigma(3):\sigma(j)})\Wick{\prod_{i=1}^{j}\hat\eta({\ell_i})}	\end{equs}
and the coefficient of $\Wick{\prod_{i=1}^{j}\hat\eta({\ell_i})}$ defines ${\hffNn}_{a,j}(\ell_{1:j})$, i.e. 
\begin{equs}\label{eq:milleindici}
  {\hffNn}_{a,j}&(\ell_{1:j})=\sum_{\sigma\in S_j}{\hffNn}_{a,j,\sigma}(\ell_{1:j})\\
  &\eqdef \sum_{\sigma\in S_j} \frac{1}{j!}
  \frac{2\hat\lambda}{\sqrt{\log N}}\frac{\nonlin_{\ell_{\sigma(1)},\ell_{\sigma(2)}}|\ell_{\sigma(1)}+\ell_{\sigma(2)}|}{|\ell_{1:j}|^2[1+G_{n+2-j}(\Ll(|\ell_{1:j}|^2/2))]}
\hat{\ff}^{N,n}_{a,j-1}(\ell_{\sigma(1)}+\ell_{\sigma(2)},\ell_{\sigma(3):\sigma(j)}).
\end{equs}
This whole expression should be plugged into the definition of $\fFN_{G}$, both for
 ${\hffNn}_{a_1,j_1}$ and for $\hffNn_{a_2,j_2}$. 
 Since the cardinality of $S_j$ does not depend on $N$, to obtain \eqref{eq:dd2} it is sufficient to show that, for every $\sigma^{(1)}\in S_{j_1},\sigma^{(2)}\in S_{j_2}$ one has
\begin{equ}
  \lim_{N\to\infty}\|(-\gensy)^{-1/2}\fFN_{G,\sigma^{(1)},\sigma^{(2)}}\|=0
  \label{eq:dd3}
\end{equ}
where
\begin{equs}
\fFN_{G,\sigma^{(1)},\sigma^{(2)}}\eqdef\sum_k |k|^2&\sum_{\substack{\ell_v\in\Z^2\setminus\{0\}\\ v\in V_{1,2}}} \prod_{(a,b)\in G}\mathds{1}_{\ell_a=-\ell_b} \\
  &\times {\hffNn}_{a_1,j_1,\sigma^{(1)}}(\ell|_{V_1},k)\hffNn_{a_2,j_2,\sigma^{(2)}}(\ell|_{ V_2},-k)
   \Wick{\prod_{v\in V[G]}\hat\eta(\ell_v)}\,.
\end{equs}
 To continue, we observe that
\begin{equs}
(-\gensy)^{-1} \Wick{\hat \eta(q_1)\dots\hat\eta(q_n)}
= \frac{2}{{\sum_{i\le n}|q_i|^2}} \Wick{\hat \eta(q_1)\dots\hat\eta(q_n)}\,.
\end{equs}
Altogether, this implies that
\begin{equs}\label{e:drd}
  \|&(-\gensy)^{-1/2}\fFN_{G,\sigma^{(1)},\sigma^{(2)}}\|^2\\
  &= 2\sum_{k,k'}|k|^2\, |k'|^2 \sum_{\substack{\ell_v,\,\ell_v'\\ v\in V_{1,2}}} \Big(\prod_{(a,b)\in G}\mathds{1}_{\substack{\ell_a=-\ell_b\\ \ell'_a=-\ell'_b}}\Big) \mathbb E \Bigl[ \Wick{\prod_{v\in V[G]}\hat\eta(\ell_v)}\Wick{\prod_{v\in V[G]}\hat\eta(\ell_v')}\Bigr]\\
  &\quad\quad\times \frac{{\hffNn}_{a_1,j_1,\sigma^{(1)}}(\ell|_{ V_1},k){\hffNn}_{a_2,j_2,\sigma^{(2)}}(\ell|_{V_2},-k){\hffNn}_{a_1,j_1,\sigma^{(1)}}(\ell|_{ V_1}',k'){\hffNn}_{a_2,j_2,\sigma^{(2)}}(\ell|_{ V_2}',-k')}{\sum_{v\in V[G]}|\ell_v|^2}\,.
\end{equs}
In order to express the expectation above, let $\mathcal P[G]$ be the set of bijections of $V[G]$ to itself 
which we think of as a set of edges whose first endpoint is in $V[G]$ and the second in an identical copy of $V[G]$. 
Then, by the observation just after \eqref{e:Wick}, we have
\begin{equs}
\mathbb E \Bigl[ \Wick{\prod_{v\in V[G]}\hat\eta(\ell_v)}\Wick{\prod_{v\in V[G]}\hat\eta(\ell_v')}\Bigr]=  \sum_{P\in \mathcal P[G]}\prod_{(v,w)\in P}\mathds{1}_{\ell_v=-\ell'_w}.
\end{equs}
Once more, note that the cardinality of $\mathcal P[G]$ is bounded independently of $N$  
so that it is sufficient to show that, for every $P\in \mathcal P[G]$,
\begin{equs}
  \label{e:tremendissima}
&\sum_{k,k'}|k|^2\, |k'|^2 \sum_{\substack{\ell_v,\,\ell_v'\\ v\in V_{1,2}}} \prod_{(a,b)\in G}\mathds{1}_{\substack{\ell_a=-\ell_b\\ \ell'_a=-\ell'_b}}  \prod_{(v,w)\in P}\mathds{1}_{\ell_v=-\ell'_w}\\
  &\quad\times \frac{{\hffNn}_{a_1,j_1,\sigma^{(1)}}(\ell|_{V_1},k){\hffNn}_{a_2,j_2,\sigma^{(2)}}(\ell|_{ V_2},-k){\hffNn}_{a_1,j_1,\sigma^{(1)}}(\ell|_{V_1}',k'){\hffNn}_{a_2,j_2,\sigma^{(2)}}(\ell|_{V_2}',-k')}{\sum_{v\in V[G]}|\ell_v|^2}\end{equs}
tends to zero as $N\to\infty$.

Since in \eqref{e:tremendissima} we will take absolute values and we
want to take the limit $N\to\infty$ for the $j_i$'s fixed, in the
definition \eqref{eq:milleindici} of ${\hffNn}_{a,j,\sigma}$ we can
drop the factor $2\hat\lambda/j!$, we can bound $\nonlin$ with the
indicator function of its arguments being smaller than $N$ and neglect
the functions $G_{n+2-j}$'s, as they are positive.  Therefore, we need
to show that for every $j_1>a_1,j_2>a_2$, $G\in \mathcal
G[j_1-1,j_2-2]$ with $V[G]\ne\emptyset$, $P\in\mathcal P[G]$ and permutations
$\sigma^{(1)},\,\sigma^{(2)}$, the quantity
\begin{equs}\label{e:trem2}
&\frac1{(\log N)^2}\sum_{k,k'}\sum_{\substack{\ell_v,\,\ell_v'\\ v\in V_{1,2}}}\frac{|k|^2\, |k'|^2}{\sum_{v\in V[G]}|\ell_v|^2}  \prod_{(a,b)\in G}\mathds{1}_{\substack{\ell_a=-\ell_b\\ \ell'_a=-\ell'_b}}  \prod_{(v,w)\in P}\mathds{1}_{\ell_v=-\ell'_w}\\
  &\quad\times \frac{{\ft}_{a_1,j_1,\sigma^{(1)}}(\ell|_{V_1}, k)\,{\ft}_{a_2,j_2,\sigma^{(2)}}(\ell|_{V_2}, -k)
      \,{\ft}_{a_1,j_1,\sigma^{(1)}}(\ell|_{V_1}',k')\,{\ft}_{a_2,j_2,\sigma^{(2)}}(\ell|_{V_2}',-k')}{(|\ell|_{V_1}|^2+ |k|^2)(|\ell|_{V_2}|^2+ |k|^2)(|\ell|_{V_1}'|^2+ |k'|^2)(|\ell|_{V_2}'|^2+ |k'|^2)}
\end{equs}
tends to zero as $N\to\infty$, where
\begin{equ}[eq:deft]
  {\ft}_{a,j,\sigma}(q_{1:j})\eqdef |q_{\sigma(1)}+q_{\sigma(2)}|\, |\hffNn_{a,j-1}(q_{\sigma(1)}+q_{\sigma(2)},q_{\sigma(3):\sigma(j)})|\mathds{1}_{|q_{\sigma(1)}|,|q_{\sigma(2)}|\le N}.
\end{equ}

To make formula \eqref{e:trem2} more readable, let us introduce the following conventions. 
Let $\tilde G$ be the graph given by an identical copy of $G$ whose vertex set is $V_{3,4}\eqdef V_3\cup V_4$ where 
$V_3\eqdef \{(3,i)\,: i=1,\dots,j_1-1\}$ and $V_4\eqdef \{(4,i)\,: i=1,\dots,j_2-1\}$, 
and whose edges are such that $((3,a),(4,b))\in \tilde G$ iff $((1,a),(2,b))\in G$. Accordingly, we view $P$ as 
a bijection from $V[G]$ to $V[\tilde G]$, i.e. as the set of edges connecting 
a variable $p_{(u,i)}$ for $u\in \{1,2\}$ with another variable $p_{(v,j)}$ with $v\in \{3,4\}$.
Then, we set 
\begin{equ}[e:Feynman]
\gamma\eqdef G\cup \tilde G\cup P\cup\{((1,0),(2,0)),\, ((3,0),(4,0))\},
\end{equ}
which is the edge set of a graph whose vertices are those in $V\eqdef
V_{1,2}\cup V_{3,4}\cup\{(u,0)\colon u=1,\dots,4\}$, $|V|=2(j_1+j_2)$
and which has $j_1+j_2$ edges.  We redefine the momentum variables
$\ell,\ell'$ by introducing $p_{(u,i)},\,u=1,\dots,4$ where
$p_{(1,i)}\eqdef\ell_{v}$, $p_{(3,i)}\eqdef\ell_{v}'$ for $v=(1,i)$
and $i\le j_1-1$, and $p_{(2,i)}\eqdef \ell_v$, $p_{(4,i)}\eqdef
\ell_v'$ for $v=(2,i)$ and $i\le j_2-1$, and
$p_{(1,0)}=-p_{(2,0)}\eqdef k$ and $p_{(3,0)}=-p_{(4,0)}\eqdef
k'$. Then~\eqref{e:trem2} reads 
\begin{equ}[e:trem3]
\frac1{(\log N)^2}\sum_{\substack{p_{(u,i)}\\ (u,i)\in V}} \frac{1}{\sum_{(u,a)\in V[G]}|p_{(u,a)}|^2}\prod_{((u,i),(v,j))\in \gamma}\mathds{1}_{p_{(u,i)}=-p_{(v,j)}} \prod_{u=1}^4\frac{ |p_{(u,0)}| |{\ft}^{u}(p_{(u)})|}{|p_{(u)}|^2},
\end{equ}
and in Figure \ref{fig:Feyn} we provide 
graphical representation to which we will refer throughout. 
Here, for each $(u,i)\in V$, $p_{(u,i)}$ takes values in
$\Z^2\setminus\{0\}$ with the additional constraint that
$|p_{(u,i)}|\leq N$ if $p_{(u,i)}$ is one of the eight variables
inside one of the blue rectangles (this comes from the indicator in \eqref{eq:deft}). Also, we introduced the
shorthand notations $p_{(u)}\eqdef (p_{(u,0)},\dots,p_{(u,j^u-1)})$
and
\begin{equ}
  \label{e:ti}
{\ft}^{u}\eqdef {\ft}_{a^u,j^u,\sigma^{(u)}}
\end{equ}
where
\begin{equ}
  \label{e:ju}
j^u=j_1 \text{ for } u\in \{1,3\} \text{ and } j^u=j_2 \text{ for } u\in\{2,4\}
\end{equ}
and an analogous definition holds for $\sigma^{(u)},a^u$.

We call $\Gamma[j_1,j_2]$ the set of all allowed Feynman diagrams 
(this includes the constraint that $P$ is non-empty), whose graph $\gamma$ is necessarily of the form~\eqref{e:Feynman}. 
%If we establish that the edges of the diagram are the edges in $G,P$ plus the two edges joining $p_{(1,0)}$ to $p_{(2,0)}$ and $p_{(3,0)}$ to $p_{(4,0)}$ respectively, then the
%product of indicator functions in
%\eqref{e:trem3} is the product over edges of $\gamma$ of the indicator that the momenta at the endpoints of the edges are opposite.
\begin{figure}[t]
\begin{center}
\includegraphics[width=9cm]{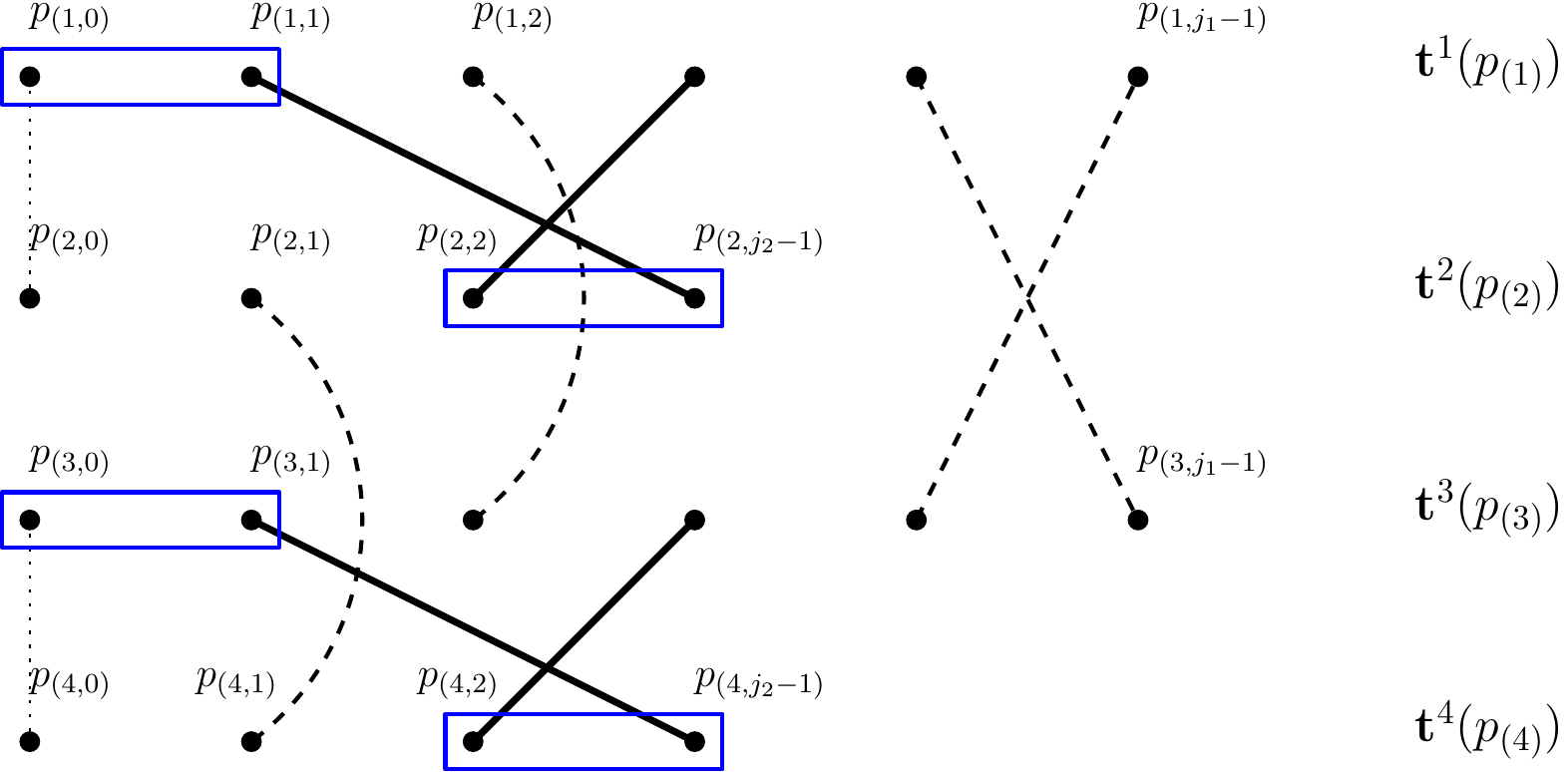}
\caption{A graphical representation of formula \eqref{e:trem3} for
  $j_1=6,j_2=4$. Each of the four horizontal strings of dots represent the vertices in $V_1,\dots,V_4$ and to each line we associate the corresponding factor
  ${\ft}^{u}$. %% he
  %% first dot in each line stands for the variable $\pm k$ or $\pm k'$
  %% as indicated, while the other dots stand for $p_{(a,i)},
  %% a=1,\dots,4$.
  An edge connecting two dots means that the variables labelled by the 
   two endpoints coincide modulo a sign change. Thick edges correspond to edges
  in $G,\tilde G$ (note that they are repeated identically in the first and second pair of rows) while dashed edges correspond to
  edges in $P$. There is necessarily at least one dashed line, as $V[G]$ is non-empty. 
  In fact, the collection of endpoints of dashed lines in the first two rows is exactly $V[G]$. 
  The two dotted edges are present in every Feynman diagram $\gamma$ and encode the conditions $p_{(1,0)}= -p_{(2,0)}$ and $p_{(3,0)}=-p_{(4,0)}$. Finally, the positions of the rectangles depend on the permutation $\sigma$ and correspond to the variables which are added up in $\ft^{(i)}$ (see~\eqref{eq:deft}) - 
  for instance, in the first line they represent $p_{(1,0)}+p_{(1,1)}$ in 
  $\ft^{(1)}(p_{(1)})={\ft}_{a_1,j_1,\sigma^{(1)}}(p_{(1)})$. 
  Rectangles can either contain one of the variables $p_{(u,0)},\, u=1,\dots,4$ (as in row 1) or not (as in row 2). Note that the two vertices contained in a rectangle are in the same row but need not be adjacent, although for clarity of the pictures we will always draw cases where they are.}
\label{fig:Feyn}
\end{center}
\end{figure}
Altogether, up to now we have proved the following proposition. 

\begin{proposition} 
  If \eqref{e:trem3} converges to zero as $N\to\infty$ for every $\gamma\in\Gamma[j_1,j_2]$, then \eqref{eq:dd2} (and hence
  \eqref{eq:ddaa}) holds with $j_1>a_1,j_2>a_2$.
\label{prop:bonn}
\end{proposition}

It is important for what follows to note that, when applying the
definition \eqref{eq:deft} of $ {\ft}_{a,j,\sigma}(q_{1:j})$ for
instance to ${\ft}_{a_1,j_1,\sigma^{(1)}}(p_{(1)})$, it may happen that
the two summed variables $q_{\sigma(1)}+q_{\sigma(2)}$ (i.e. the
variables in the rectangle of row $1$) are two of the variables
$p_{(1,i)},i\ge1$, or that it is the sum of $p_{(1,0)}$ plus one of
the variables $p_{(1,i)},i\ge1$. This depends on the chosen
permutation $\sigma^{(1)}$.  In fact, we are going to estimate the
Feynman diagrams differently according to the number $\kappa\in\{0,\dots,4\}$ of indices $u=1,\dots,4$ 
such that the variable $p_{(u,0)}$ is contained in the rectangle of
row $u$. We denote by $\Gamma_\kappa[j_1,j_2]\subset \Gamma[j_1,j_2]$ the set of Feynman diagrams such that the quantity of $u\in\{1,2,3,4\}$  with the property that $p_{(u,0)}$ is contained in a rectangle equals $\kappa$. Actually, since the positions of the rectangles in rows $1,2$ is the same as in rows $3,4$, 
only the values $\kappa=0,2,4$ can occur.
The crucial estimate is the content of the following proposition. 

  \begin{proposition} \label{prop:K0} Let $j_1>a_1,j_2>a_2$.
    For every diagram $\gamma\in \Gamma_\kappa[j_1,j_2]$, $\kappa=0,2$, \eqref{e:trem3} is upper bounded by an $N$-independent constant times
    \begin{equ}
      \label{kappa02}
      \frac1{(\log N)^{(4-\kappa)/2}}\|(-\gensy)^{1/2}\ffNn_{a_1,j_1-1}\|^2\|(-\gensy)^{1/2}\ffNn_{a_2,j_2-1}\|^2.
    \end{equ}
    If instead $\kappa=4$, then  \eqref{e:trem3} is upper bounded by the supremum of 
    % (according to the geometry of the graph) either by 
    $1/(\log N)^2$ times 
    \begin{eqnarray}
      \label{risultatok4}
      \|(-\gensy)^{1/2}\ffNn_{a_1,j_1-1}\|^2\|(-\gensy)^{1/2}\ffNn_{a_2,j_2-1}\|^2
    \end{eqnarray}
    and the product \eqref{risultatok4} where, in at least  one of the four norms $\|(-\gensy)^{1/2}\ffNn_{a_i,j_i-1}\|$ 
    the operator $(-\gensy)^\half$ is replaced by $\cS$, which in turn is the diagonal operator whose Fourier
  multipliers $\sigma=(\sigma_n)_n$ act as
\begin{equ}[e:S]
\sigma_n(k_{1:n})\eqdef \sum_{\pi\in S_n}\frac{|k_{\pi(1)}|^2}{|k_{\pi(2)}|},\,\qquad k_{1:n}\in\Z^{2n}\setminus\{0\}, \qquad \text{for all} \quad n\ge 2
\end{equ}
 while $\sigma_n\equiv 0, n\le 1$.
\end{proposition}

%\giuseppeText{alternativa a~\eqref{e:K=4}. 
%\begin{equs}
%\|&(-\gensy)^{1/2}\ffNn_{j_1-1}\|\|(-\gensy)^{1/2}\ffNn_{j_2-1}\|\\
%&\times \Big(\frac{1}{\log N} \|(-\gensy)^{1/2}\ffNn_{j_1-1}\|\|(-\gensy)^{1/2}\ffNn_{j_2-1}\| + \|\cS)^{1/2}\ffNn_{j_1-1}\|\|(-\gensy)^{1/2}\ffNn_{j_2-1}\| + \|(-\gensy)^{1/2}\ffNn_{j_1-1}\|\|(\cS)^{1/2}\ffNn_{j_2-1}\|
%\end{equs}}

The proof of these estimates is a based on a very elaborate and iterative procedure, and it is deferred to Section \ref{sec:tourdeforce}.

\begin{remark}
  The most challenging case is $\kappa=4$. 
  We will see that in this case, the denominator $\sum_{(u,a)\in V[G]}|p_{(u,a)}|^2$ - due to the existence of at least one dashed line - 
  is crucial in the estimate.
  Recall also that this denominator  originates from the operator $(-\gensy)^{-1}$ in \eqref{e:Var1bound}.
\end{remark}

We need the following estimate concerning the action of the operator $\cS$, whose proof 
is deferred to Appendix \ref{app:S}.  
\begin{proposition}\label{prop:Spiccolo}
  For any  $a=1,2$, $n\in\mathbb N,a\le j\le n$ one has
  \begin{equ}
\lim_{N\to\infty}    \|\cS^\half\ffNn_{a,j}\|=0.
  \end{equ}
\end{proposition}
In Section \ref{sec:subsub2} we will show how the previous estimates imply \eqref{eq:dd2}. Before doing that, we need to consider the case where $j_1=a_1$ and/or $j_2=a_2$, that we have left aside for the moment.

\subsubsection{The case where $j_1=a_1$ and/or $j_2=a_2$}
\label{sec:case=}  In this section we consider the case where either $j_1=a_1$ or $j_2=a_2$, 
or both. Recall that we need to prove \eqref{eq:ddaa}. Note that 
the case $j_1=a_1=1=a_2=j_2$ is trivial: since $\ffNn_{1,1}$ belongs to the first Wiener chaos, 
$D_k \ffNn_{1,1}$ is a constant and one sees immediately from its definition that $\fFNn_{1,1; 1,1}$ 
vanishes identically.
Note also that if $a_1=j_1$ then the function $\ffNn_{a_1,j_1}$ that appears
in the definition of $\fFNn_{a_1,j_1; j_2,a_2}$ is given by the first line in \eqref{e:Fjn}  
so that (recall the definition of $\fg_1,\fg_2$) its kernels are
\begin{equ}[e:kernh1]
  \hffNn_{1,1}(k)=\nueff\frac{\hat\psi(k)}{1+G_{n+1}(\Ll(|k|^2/2))}\,
\end{equ}
and 
\begin{equation}
  \label{e:kernh2}
  \hffNn_{2,2}(k_1,k_2)=\lambda_N\mathds{1}_{k_1=-k_2}\frac{\nonlin_{k_1,-k_1}}{|k_1|^2(1+G_n(\Ll(|k_1|^2)))}\,,
\end{equation}
respectively. One can first of all deal with the case where both $a_1=j_1$ and $a_2=j_2$.

\begin{proposition} \label{prop:gialla}
One has $\fFNn_{2,2;1,1}\equiv 0$, 
\begin{equ}
   \|(-\gensy)^{-1/2}\fFNn_{2,2;1,1}\|^2\lesssim \frac1{\log N} \|\psi\|_{H^1(\mathbb T^2)}^2\,
\end{equ}
and
\begin{equ}
   \|(-\gensy)^{-1/2}\fFNn_{2,2;2,2}\|^2\lesssim \frac1{(\log N)^2} .
\end{equ}
\end{proposition}
These bounds follow  from a direct computation, using $G_n\ge 0$ and the explicit expressions of the kernels \eqref{e:kernh1} and  \eqref{e:kernh2}; we omit the details.

The last two cases that can occur are
\begin{itemize}[noitemsep]
\item $j_1=a_1=1$ and $j_2>a_2$ (or $j_2=a_2=1$, $j_1>a_1$);
  
\item $j_1=a_1=2$ and $j_2>a_2$ (or $j_2=a_2=2$ and $j_1>a_1$).
\end{itemize}
We start with the former. In this case, the function $\ffNn_{a_1,j_1}$ that appears
in the definition of $\fFNn_{1,1; j_2,a_2}$ is given by \eqref{e:kernh1},
while $\ffNn_{a_2,j_2}$ is given as in \eqref{e:dnm}.  Proceeding as
in Section \ref{sec:case>}, one can derive a bound for
$\|(-\gensy)^{-\half}\fFNn_{1,1;a_2,j_2}\|^2$ as a sum over diagrams
$\gamma\in\Gamma[1,j_2]$, whose terms are similar to~\eqref{e:trem3}. The 
 diagrams $\gamma\in\Gamma[1,j_2]$ (see Figure \ref{fig:Feyn6} left for an example) have a simpler geometry than those in $\Gamma[j_1,j_2]$ with $j_1>a_1,j_2>a_2$, introduced in the previous section.
\begin{figure}[t]
\begin{center}
  \includegraphics[width=9cm]{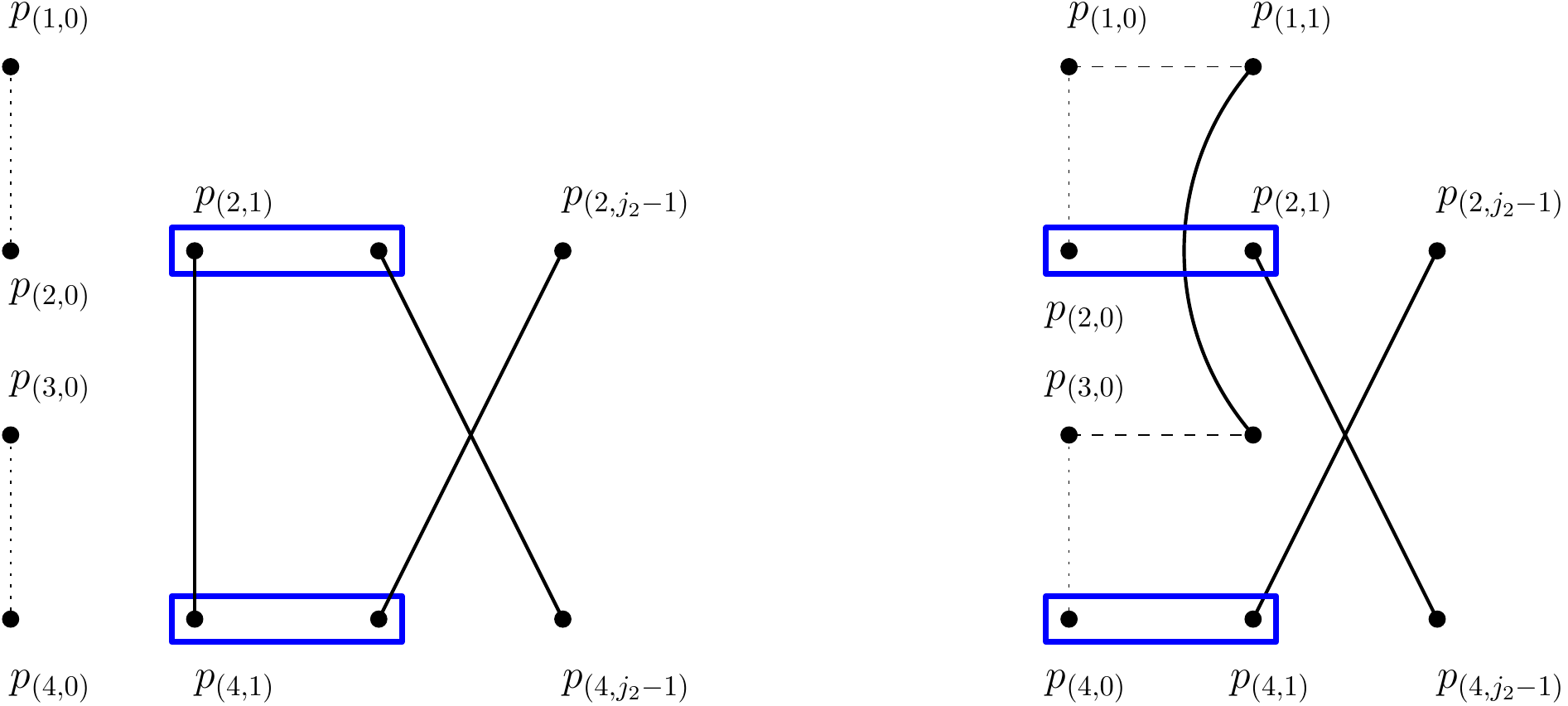}
\label{eq:h2kernel}
\caption{Left: A Feynman diagram in the case $j_1=a_1=1$ and $j_2=4$. In this example, the rectangles do not contain the vertices $(2,0),(4,0)$, that is, $\kappa=0$. 
%When performing the bounds, edges connected to rectangles need to be colored and directed as in 
%Section \ref{sec:case>} but this is not displayed here.
Right: a diagram in the case $j_1=a_1=2$ and $j_2=3$. The horizontal dashed lines in the first/third rows indicate that, because of the form of the kernel of $\ffNn_{2,2}$ (see \eqref{e:kernh2}), the two variables at the endpoints take opposite values.
}
\label{fig:Feyn6}
\end{center}
\end{figure}
Note first that the graph $G\in \mathcal G[j_1-1,j_2-1]$ (see \eqref{e:FkG}) in each $\gamma\in \Gamma[1,j_2]$
is empty, because $j_1-1=0$. As a
consequence, in the left drawing of Figure \ref{fig:Feyn6}, there are no thick lines between the
first and second (or between third and fourth) rows. This implies that
$V[G]=V_2=\{(2,i), i=1,\dots,j_2-1\}$. For the same
reason, the first and third row contain a single vertex, which are
$p_{(1,0)}$ and $p_{(3,0)}$, and there are rectangles only in
the second and fourth rows.  The analog of \eqref{e:trem3} is then the following upper bound on the contribution from $\gamma\in \Gamma[1,j_2]$ to $\|(-\gensy)^{-\half}\fFNn_{1,1;a_2,j_2}\|^2$
\begin{equs}
  \label{e:tremj1}
  &\frac1{\log N}\sum_{\substack{p_{(u,i)}\\ (u,i)\in V}} \frac{1}{\sum_{(u,a)\in V_2}|p_{(u,a)}|^2}\\
  &\qquad\qquad\times\prod_{((u,i),(v,j))\in \gamma}\mathds{1}_{p_{(u,i)}=-p_{(v,j)}} \prod_{u=2,4}\frac{ |p_{(u,0)}| |{\ft}^{u}(p_{(u)})|}{|p_{(u)}|^2}\prod_{u=1,3}|\ffNn_{1,1}(p_{(u,0)})|\,.
\end{equs}
Note that the  prefactor $(\log N)^{-2}$ of \eqref{e:trem3} is replaced here by $(\log N)^{-1}$,
because the functions $\ffNn_{1,1}$ have no logarithmic factor in
front.  
The direct analog of Proposition \ref{prop:bonn} is the following. 

\begin{proposition} \label{prop:bonn2} If \eqref{e:tremj1} converges to zero as $N\to\infty$ for every $\gamma\in \Gamma[1,j_2]$ with $j_2>a_2$, then \eqref{eq:dd2} (and hence \eqref{eq:ddaa}) holds with $j_1=1,j_2>a_2$.
\end{proposition}

In order to bound the sum in~\eqref{e:tremj1}, again one has to
distinguish diagrams $\gamma$ according to the number $\kappa\in\{0,2\} $ of
rows with index $u\in \{2,4\}$ such that the rectangle includes the
vertex $(u,0)$. As a result, we have the next proposition. %In turn, the sum in \eqref{e:tremj1} is bounded as follows:

\begin{proposition}
  \label{prop:jpiccolo1}
  If $j_1=a_1=1 $ and $j_2>a_2$, for diagrams $\gamma$ with  $\kappa=0$, \eqref{e:tremj1} is upper bounded by 
  \begin{equ}
    \label{eq:ev}
    \frac1{\log N}\|(-\gensy)^{1/2}\ffNn_{a_2,j_2-1}\|^2 \|\psi\|_{H^1(\mathbb T^2)}^2.
  \end{equ}
  If instead $\kappa=2$, then~\eqref{e:tremj1} is bounded by 
  % the value of the diagram is bounded  (according to the geometry of the graph) either by \eqref{eq:ev} or by 
  \begin{equ}
   \|\psi\|_{H^1(\mathbb T^2)}^2\times \max\left[\frac1{\log N}\|(-\gensy)^{1/2}\ffNn_{a_2,j_2-1}\|^2 ,\|\cS^{1/2}\ffNn_{a_2,j_2-1}\|^2\right].
  \end{equ}
\end{proposition}
% \fabioText{the first case corresponds to 1 conn component, and the second to 2, but at this stage we have not yet defined connected components and also in the previous proposition we have not specified that}

Finally, we consider the case where $j_1=a_1=2$ and  $j_2> a_2$.
In this case, the function $\ffNn_{2,2}$ that appears in the definition of 
$\fFNn_{2,2; j_2,a_2}$ is given \eqref{e:kernh2} and  $\ffNn_{a_2,j_2}$ is as in \eqref{e:dnm}. 
Once more, we bound $\|(-\gensy)^{-\half}\fFNn_{2,2;a_2,j_2}\|^2$ as a sum over diagrams
$\gamma\in\Gamma[2,j_2]$, an example of which can be seen in the right drawing of  Fig. \ref{fig:Feyn6}. 
The first and third lines contain just two vertices, $p_{(1,0)},p_{(1,1)}$ and $p_{(3,0)},p_{(3,1)}$ respectively. 
Because of the indicator function in \eqref{e:kernh2}, one has  $p_{(1,0)}=-p_{(1,1)}$ and $p_{(3,0)}=-p_{(3,1)}$.
The analogue of \eqref{e:trem3} and \eqref{e:tremj1} is then the following 
%upper bound on the contribution from $\gamma\in \Gamma[2,j_2]$ to $\|(-\gensy)^{-\half}\fFNn_{2,2;a_2,j_2}\|^2$:
\begin{equs}
  \label{e:tremj2}
  &\frac1{\log N}\sum_{\substack{p_{(u,i)}\\ (u,i)\in V}} \frac{1}{\sum_{(u,a)\in V_2}|p_{(u,a)}|^2}\\
  &\qquad\qquad\times\prod_{((u,i),(v,j))\in \gamma}\mathds{1}_{p_{(u,i)}=-p_{(v,j)}} \prod_{u=2,4}\frac{ |p_{(u,0)}| |{\ft}^{u}(p_{(u)})|}{|p_{(u)}|^2}\prod_{u=1,3}|\ffNn_{2,2}(p_{(u,0)},p_{(u,1)})|\,
\end{equs}
and the direct analog of Proposition \ref{prop:bonn}  and \ref{prop:bonn2} holds. 

\begin{proposition} \label{prop:bonn3} If \eqref{e:tremj2} converges to zero as $N\to\infty$ for every $\gamma\in \Gamma[2,j_2]$ with $j_2>a_2$, then \eqref{eq:dd2} (and hence \eqref{eq:ddaa}) holds with $j_1=2,j_2>a_2$.
\end{proposition}

Finally, the sum in \eqref{e:tremj2} admits the following bounds.

\begin{proposition}
  \label{prop:jpiccolo2}
  For eacy $\gamma
  \in \Gamma[2,j_2], j_2>a_2$ the sum in \eqref{e:tremj2} is bounded by 
  \begin{equ}
    \frac{1}{\log N}\|(-\gensy)^{1/2}\ffNn_{a_2,j_2-1}\|^2\,.
  \end{equ}
\end{proposition}

The proof of Propositions \ref{prop:jpiccolo1} and \ref{prop:jpiccolo2} is very similar 
(but simpler, because the Feynman graphs involved are simpler) 
to that of Proposition \ref{prop:K0}, 
so that in Section \ref{sec:tourdeforce} we will only focus on the latter. 

\subsubsection{Proof of \eqref{e:Var}}
\label{sec:subsub2}

As argued at the beginning of Section \ref{sec:Feynman}, it is enough to prove
\eqref{eq:ddaa} for every $a_1,a_2\in\{1,2\}$ and $a_1\le j_1\le n,\,a_2\le j _2\le n$. 
The cases $j_1=a_1$ and $j_2=a_2$ are covered by Proposition \ref{prop:gialla} 
since $\psi\in H^1(\T^2)$. 
When instead $j_1>a_1, j_2>a_2$, by Proposition \ref{prop:bonn}, it suffices to show that
\eqref{e:trem3} tends to zero as $N\to\infty$ for every $\gamma\in\Gamma[j_1,j_2]$, 
which in turn is a consequence of the estimates in Proposition \ref{prop:K0}. 
In fact, if $\gamma\in\Gamma_\kappa[j_1,j_2],\kappa=0,2$ we note that the two
norms at the right hand side of~\eqref{kappa02} are bounded uniformly in $N$ in
view of Proposition \ref{l:PrelimBounds} (recall that $\ffNn_a$ is
either $\fbNn$ or $\fhNn$, according to the value of $a$) while the
prefactor tends to zero.  As for the case $\kappa=4$, we use the
additional information provided by Proposition \ref{prop:Spiccolo}.

At last, for $j_1=a_1=1$ (resp. $j_1=a_1=2$) and $j_2>a_2$, by Proposition \ref{prop:bonn2} 
(resp. Proposition \ref{prop:bonn3}), one has to prove that
\eqref{e:tremj1} (resp. \eqref{e:tremj2}) tends to zero as $N\to\infty$
for every $\gamma\in \Gamma[1,j_2] $ (resp. $\gamma\in \Gamma[2,j_2]$). 
This time we apply the bounds in Proposition \ref{prop:jpiccolo1} 
(resp. Proposition \ref{prop:jpiccolo2}) and argue as above. 
%again using the fact that  $\|(-\gensy)^{1/2}\ffNn_{a_2,j_2-1}\|=O(1)$,
%$ \|\cS^{1/2}\ffNn_{a_2,j_2-1}\|=o(1)$ and the function $\psi$ is
%$N$-independent and has finite $H^1(\mathbb T^2)$ norm.

This
concludes the proof of \eqref{e:Var}, assuming Proposition 
\ref{prop:K0}, Proposition~\ref{prop:jpiccolo1} and Proposition~\ref{prop:jpiccolo2}, 
the first of which is proven in the next section.

\subsection{Proof of Proposition \ref{prop:K0}}
% ********************
\label{sec:tourdeforce}

The main source of difficulty in bounding~\eqref{e:trem3} is represented by the factors $|p_{(u,0)}|$, $u=1,\dots, 4$, 
and the summed variables in the expression of $\ft^u$ in~\eqref{eq:deft}, 
which need to be counterbalanced by the denominators in~\eqref{e:trem3}. 
For  $u=1,\dots,4$ let $(u,i_u^1),\,(u,i_u^2)$ be the vertices in the rectangle of row $u$ and 
define the triplet of variables $x_{(u)}$ as 
\begin{equ}[e:xu]
x_{(u)}\eqdef (p_{(u,0)},p_{(u,i_u^1)},p_{(u,i_u^2)})\,.
\end{equ}
At first, we bound $|p_{(u)}|^2\geq |x_{(u)}|^2$ so that~\eqref{e:trem3} is bounded from above by
\begin{equs}[e:trem4]
      \frac1{(\log N)^2}&\sum_{\substack{p_{\bv}\\ \bv\in V}}\frac{1}{\sum_{\bv\in V[G]}|p_{\bv}|^2}
      \prod_{((u,i),(v,j))\in
        \gamma}\mathds{1}_{p_{(u,i)}=-p_{(v,j)}}\\
        &\times
      \prod_{u=1}^4\frac{ |p_{(u,0)}|
        |p_{(u,i^1_u)}+p_{(u,i^2_u)}||{\fs}^{u}(p_{(u,i^1_u)}+p_{(u,i^2_u)},p_{(u)\setminus\mathcal
            V_u})|}{|x_{(u)}|^2}\mathds{1}_{|p_{(u,i^1_u)}|,|p_{(u,i^2_u)}|\le
          N}
\end{equs}
where (see \eqref{eq:deft} and \eqref{e:ti})
$\fs^{u}\eqdef\ffNn_{a^u,j^u-1}$ and $\cV_u\eqdef \{(u,i^1_u),\,(u,i^2_u)\}$.

At this point we look at the edges of $\gamma$ (see \eqref{e:Feynman}) and split them into three disjoint families:
\begin{itemize}
\item the ``undirected edges'' $U$ with vertices $V_U$,  whose endpoints \emph{do not contain} 
variables from any of the triplets $x_{(u)},\,u=1,\dots,4$,
\item the two ``special edges'' $S$ with vertices $V_S$ that connect variables $p_{(u,0)}$ and $p_{(v,0)}$, $u\ne v$,
\item the ``directed edges'' $D$ with vertices $V_D$, that are all the others. 
\end{itemize}
In order to control~\eqref{e:trem4}, we will bound the sum over each of the variables $p_{(u,i)}$, 
first those that are endpoints in $U$, then $D$ and finally $S$. 
We will separately treat the cases $\kappa=0,2,4$ and we proceed in increasing order of difficulty, starting with $\kappa=0$. We refer to Appendix~\ref{sec:example} for an illustrative example of the estimates to come.
\medskip

\noindent{\bf The case $\kappa=0$}
\medskip 

\noindent In this case, the ratio $({\sum_{\bv\in V[G]}|p_{\bv}|^2})^{-1}$ and the indicator functions play no role: 
since all the variables involved are different from $0$, we simply bound them by $1$, thus reducing~\eqref{e:trem4} to
\begin{equs}[e:trem5]
      \frac1{(\log N)^2}&\sum_{\substack{p_\bv\\ \bv\in V}}
      \prod_{((u,i),(v,j))\in
        \gamma}\mathds{1}_{p_{(u,i)}=-p_{(v,j)}}\\
        &\times
      \prod_{u=1}^4\frac{ |p_{(u,0)}|
        |p_{(u,a^1_u)}+p_{(u,a^2_u)}||{\fs}^{u}(p_{(u,i^1_u)}+p_{(u,i^2_u)},p_{(u)\setminus\mathcal
            V_u})|}{|x_{(u)}|^2}.
\end{equs}
Now, the variables $p_{(u,i)}$ and
$p_{(v,j)}$, $u\neq v$, whose indices are vertices of an edge in $U$, appear only in $\fs^u,\fs^v$ so that 
we can bound the corresponding sum simply using Cauchy Schwarz, i.e. 
\begin{equs}
  \label{e:semplice}
\sum_{\substack{p_{\bv}\\\bv\in V_U}}\prod_{((u,i),(v,j))\in U}\mathds{1}_{p_{(u,i)}=-p_{(v,j)}}&\prod_{u=1}^{4} |{\fs}^{u}(p_{(u,i^1_u)}+p_{(u,i^2_u)},p_{(u)\setminus\mathcal
            V_u})|\\
&\leq \prod_{u=1}^4 \Big( \sum_{\substack{p_{(u,i)}\\(u,i)\in V_U}} |{\fs}^{u}(p_{(u,i^1_u)}+p_{(u,i^2_u)},p_{(u)\setminus\mathcal
            V_u})|^2\Big)^{\half}\,.
\end{equs}
Notice that the functions at the right hand side still depend on all the variables 
whose index is not in $V_U$ and not anymore on those whose index is in $V_U$, and the number 
of remaining variables is bounded uniformly in $j_1,\,j_2$.  
It is important to keep in mind that, in order to control what is left, 
we will perform the sums over the remaining indices {\it one by one} 
and at each step we will lose the dependence on a single variable. The dependence on the remaining ones
will be clear from the context and we will leave it implicit hereafter. 

We now continue with the variables whose index is in $V_D$. 
Let us first endow the edges in $D$ with both a {\it direction} and 
a {\it colour} - both of which will determine the {\it order} of summation, i.e. which sum is treated first,
and the type of estimate we apply. 
We will regard two edges in $D$ as ``connected'' if they intersect the same rectangle. 
In this way, the edges in $D$ can be split into (at most four) connected components, each of which is 
either a closed loop or an open path. 
For an open path, choose arbitrarily one of the two endpoints and orient the path from there to the other endpoint.
We assign to the first edge along the path the colour blue, to the last the colour green and the colour 
red to all the others (see Figure~\ref{fig:Feyn2}, right graph). 
For closed loops instead, choose (arbitrarily) one of the two possible orientations; as for the colours, 
choose arbitrarily two connected edges, 
painting the first purple, the second orange and the others red (see Figure~\ref{fig:Feyn2}, left graph). 
At last, relabel the vertices of the graph (and relabel the variables in~\eqref{e:trem5}) 
in such a way that, for $u=1,\dots 4$, $(u, i_u^1)$ is always the tip of an edge, while $(u, i_u^2)$ the base point.
\begin{figure}\begin{center}
\includegraphics[width=5cm]{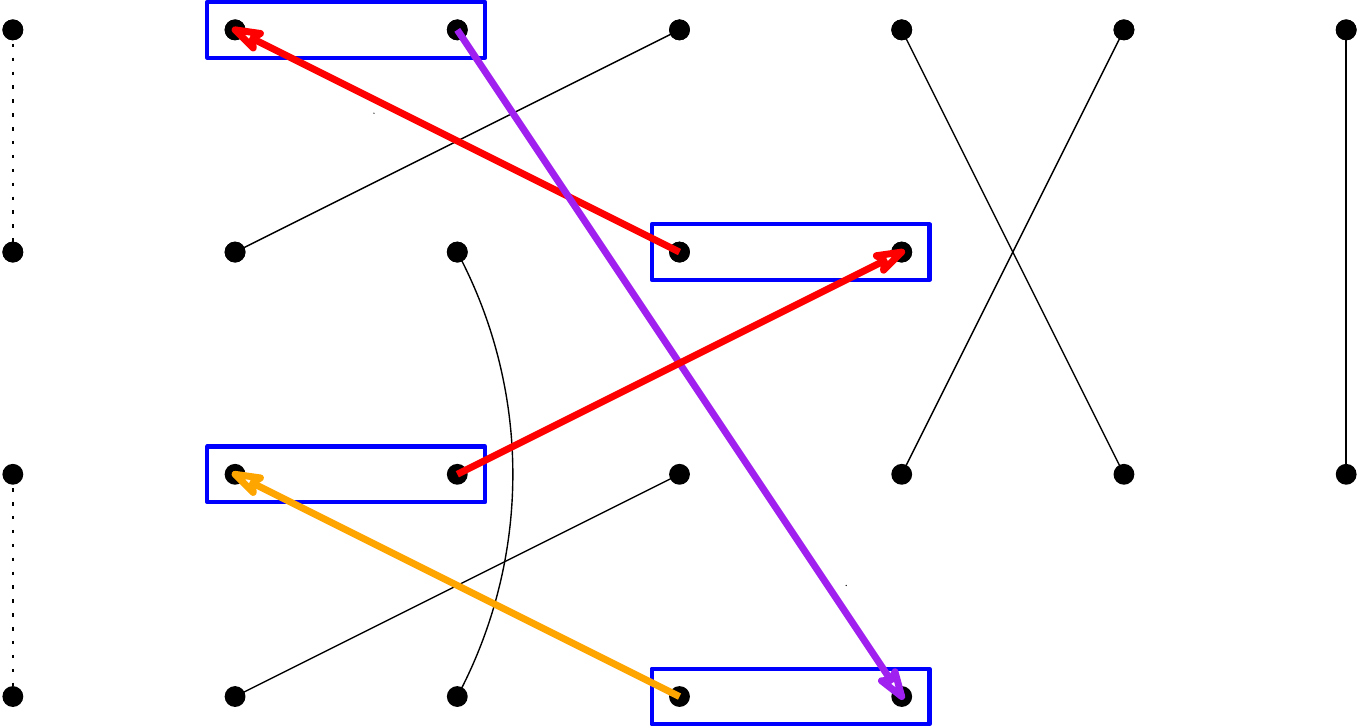}\qquad\qquad\qquad\includegraphics[width=5cm]{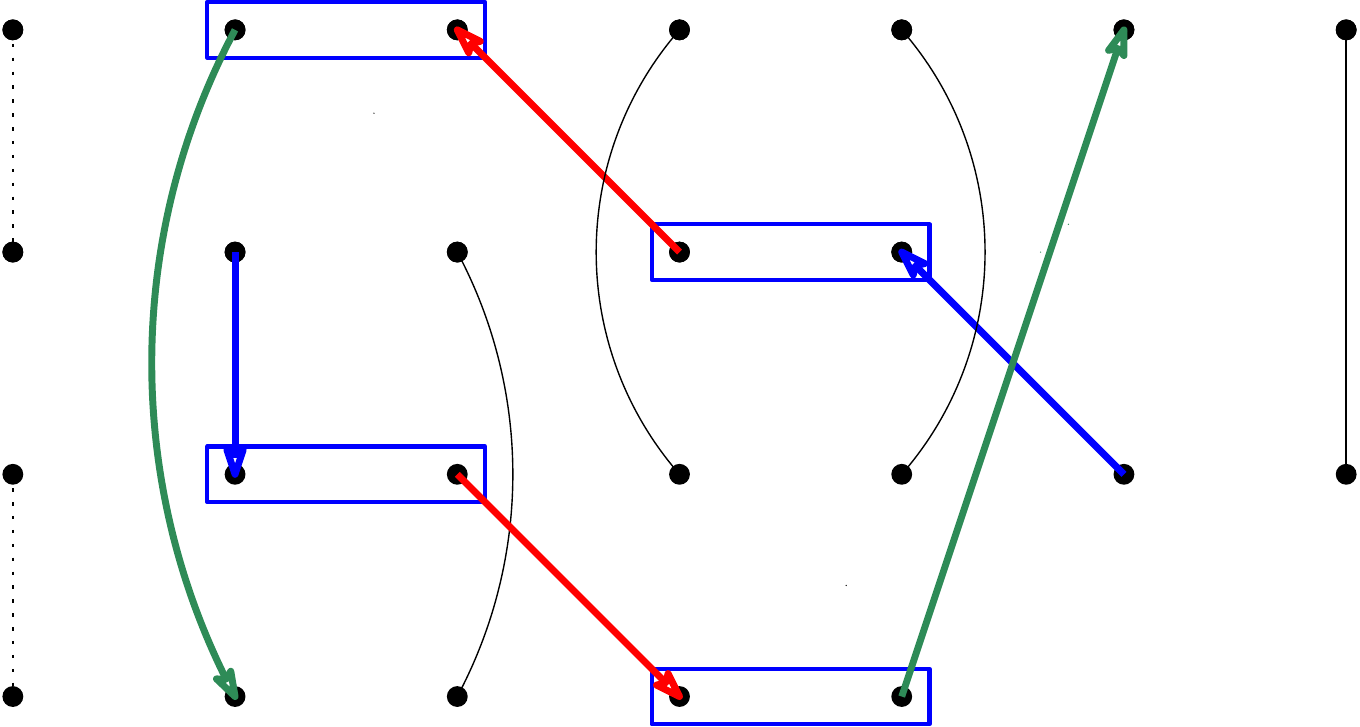}
\caption{Two examples of graphs in $\Gamma_0[j_1,j_2], j_1=7,j_2=5$: the two special edges (dotted) have no endpoints in the rectangles. Non-special edges not interesecting the rectangles are  undirected edges $U$, while those intersecting the rectangles are the directed edges $D$. Connected components of directed edges can be either closed loops (left drawing) or open paths (right drawing). For a graph $\gamma$, it can happen that both loops and open paths exist.  % For an open path, choose arbitrarily one of the two endpoints and orient the path from there to the other endpoint. The first edge along the path is painted blue, the last one green and oll other red. For closed loops instead, orient the edges anticlockwise and choose arbitrarily two edges along the path, painted orange and purple respectively; the other are painted red.
}
\label{fig:Feyn2}
\end{center}
\end{figure}

Now, let us see how to control the sums over the variables whose indices are in an open path. 
First, we upper bound~\eqref{e:trem5} by removing from the denominator all 
the variables $p_{(u,i_u^1)}$ and those whose index is the base point of the blue edge or the tip of the green one. 
Then, we first bound the sum over the variables whose indices are those of the blue edge and then the sum 
over the variables whose indices are those of the edge connected to that (which could be either red or green) and so
on. The last sum we perform is that on the variables whose indices are those of the green edge. 
We write $f_u$ and $f_v$ to denote two functions, which at each step represent 
what is left of the kernel ${\fs}^u$ and ${\fs}^v$ after performing 
the bounds over all the variables considered before, so that they do not depend on any variable which 
was already summed but might depend on all the variables whose sum was not yet considered. In the sequel when we write something like $f_u(p)$, then it does not necessarily mean that $f_u$ only depends on $p$ but might depend on other variables (those that were not summed over yet). For ease of readability however we suppress these additional variables from the notation.
Then, the estimates we exploit for the generic edge $((u,i),(v,j))$, oriented in such a way that 
$(u,i)$ is the base and $(v,j)$ is the tip, depend on the colour as follows
\begin{itemize}[noitemsep, leftmargin=*, label=-]
\item {\bf blue directed edge}: we apply Cauchy-Schwarz on  
	\begin{equs}[e:yde]
	\sum_{p_{(u,i)},\,p_{(v,i_v^1)}}& \mathds{1}_{p_{(u,i)}=-p_{(v,i_v^1)}}\,f_{u}(p_{(u,i)})|p_{(v,i_v^1)}+p_{(v,i_v^2)}| f_{v}(p_{(v,i_v^1)}+p_{(v,i_v^2)}) \\
	&\leq \Big(\sum_{p_{(u,i)}}f_{u}(p_{(u,i)})^2\Big)^{\half}\Big(\sum_{p_{(v,i_v^1)}} |p_{(v,i_v^1)}|^2f_{v}(p_{(v,i_v^1)})^2\Big)^{\half}\,.
	\end{equs} 
Note that the right hand side {\it does not depend} on the other variable in the original sum, i.e. $p_{(v,i_v^2)}$, anymore. 

\item {\bf red directed edge}: we apply Cauchy-Schwarz on 
	\begin{equs}[e:rde]
	\sum_{p_{(u,i_u^2)},\,p_{(v,i_v^1)}} &\mathds{1}_{p_{(u,i_u^2)}=-p_{(v,i_v^1)}}\,\frac{|p_{(v,i_v^1)}+p_{(v,i_v^2)}| f_{v}(p_{(v,i_v^1)}+p_{(v,i_v^2)})}{|p_{(u,i_v^2)}|^2 +|p_{(u,0)}|^2}\\
	&\leq \Big(\sum_{p_{(u,i_u^2)}}\frac{1}{(|p_{(u,i_u^2)}|^2 +|p_{(u,0)}|^2)^2}\Big)^{\half}\Big(\sum_{p_{(v,i_v^1)}} |p_{(v,i_v^1)}|^2 f_{v}(p_{(v,i_v^1)})^2\Big)^{\half}\\
	&\lesssim \frac{1}{|p_{(u,0)}|}\Big(\sum_{p_{(v,j)}} |p_{(v,j)}|^2f_{v}(p_{(v,j)})^2\Big)^{\half}
	\end{equs} 
and once again note that the right hand side does {\it does not depend} on $p_{(v,i_v^2)}$ anymore. 

\item {\bf green directed edge}: we apply Cauchy-Schwarz but on
\begin{equs}\label{e:gde}
	&\sum_{p_{(u,i_v^2)},\,p_{(v,j)}} \mathds{1}_{p_{(u,i_v^2)}=-p_{(v,j)}}\,\,\frac{f_{v}(p_{(v,j)})}{|p_{(u,i_v^2)}|^2 +|p_{(u,0)}|^2}\\
	&\quad\leq \Big(\sum_{p_{(u,i_v^2)}}\frac{1}{(|p_{(u,i_v^2)}|^2 +|p_{(u,0)}|^2)^2}\Big)^{\half}\Big(\sum_{p_{(v,j)}} f_{v}(p_{(v,j)})^2\Big)^{\half}\lesssim \frac{1}{|p_{(u,0)}|}\Big(\sum_{p_{(v,j)}} f_{v}(p_{(v,j)})^2\Big)^{\half},
	\end{equs} 
\end{itemize}
where, in both~\eqref{e:rde} and~\eqref{e:gde}, we exploited the following estimate: 
\begin{equ}[e:gaink]
\sum_{\ell}\frac{1}{(|\ell|^2 +|p_{(u,0)}|^2)^2}\lesssim \int_0^\infty\frac{r \dd r}{(r^2+|p_{(u,0)}|^2)^2}\lesssim\int_{|p_{(u,0)}|^2}^\infty\frac{\dd r}{r^2}=\frac{1}{|p_{(u,0)}|}\,.
\end{equ}

The reason why the first term at the left hand sides of~\eqref{e:rde} and~\eqref{e:gde} do not display $f_u$ 
is that both red and green edges always comes after either a blue or a red edge, whose tip 
$(u,i_u^1)$ is the unique other vertex in the same rectangle as the base $(u,i_u^2)$.
This means that the dependence on $p_{(u,i_u^2)}$ of the 
only function which would have it as a variable (in a sum) 
has already been removed by either the bound~\eqref{e:yde} or~\eqref{e:rde}. Also note that the denominator in~\eqref{e:rde} and~\eqref{e:gde} come from the fact that a red and green edge are always preceded by a red or blue edge.

Next, we sum over variables corresponding to oriented closed loops (if there are any). 
First, we upper bound~\eqref{e:trem5} by removing from
the denominator all the variables whose index is either a vertex on
the orange edge or at  the tip of a red edge,  {\it not} those whose index is a vertex of the purple edge. 
We proceed by  summing over the variables whose indices are those of the orange edge, and then 
with the next edge anti-clockwise, until the last of the loop, which is purple. 
As before, the generic edge $((u,i),(v,j))$ is oriented in such a way that 
$(u,i)$ is the base and $(v,j)$ is the tip. 
The type of bound we use will depend on the colour of the edge as follows 
(we do not specify what to do with red edges, as this has already been explained above)
\begin{itemize}[noitemsep,  leftmargin=*, label=-]
\item {\bf orange directed edge}: we apply Cauchy-Schwarz on
	\begin{equs}
	\sum_{p_{(u,i^2_u)},\,p_{(v,i^1_v)}} &\mathds{1}_{p_{(u,i_u^2)}=-p_{(v,i_v^1)}}\,|p_{(u,i_u^1)}+p_{(u,i_u^2)}|f_u(p_{(u,i_u^1)}+p_{(u,i_u^2)})\,\\
	&\qquad\qquad\qquad\times|p_{(v,i_v^1)}+p_{(v, i_v^2)}|f_{v}(p_{(v,i_v^1)}+p_{(v, i_v^2)})\\
	&\leq \Big(\sum_{p_{(u,i_u^2)}} |p_{(u,i_u^2)}|^2f_{u}(p_{(u,i_u^2)})^2\Big)^{\half}\Big(\sum_{p_{(v,i_v^1)}} |p_{(v,i_v^1)}|^2f_{v}(p_{(v,i_v^1)})^2\Big)^{\half}\label{e:ode}
	\end{equs} 
and the right hand side  depends neither on $p_{(u,i_u^1)}$ nor on $p_{(v,i_v^2)}$. 
\item {\bf purple directed edge}: we apply Cauchy-Schwarz as follows:
	\begin{equs}\label{e:pde}
	&\sum_{p_{(u,i_u^2)},\,p_{(v,i_v^1)}} \mathds{1}_{p_{(u,i_u^2)}=-p_{(v,i_v^1)}}\,\frac1{|p_{(u,i_u^2)}|^2+ |p_{(u,0)}|^2}\,\frac1{|p_{(v,i_v^1)}|+| p_{(v,0)}|^2}\\
	&\quad\leq \Big(\sum_{p_{(u,i_u^2)}} \frac1{(|p_{(u,i_u^2)}|^2+ |p_{(u,0)}|^2)^2}\Big)^{\half}\Big(\sum_{p_{(v,i_v^1)}} \frac1{(|p_{(v,i_v^1)}|+| p_{(v,0)}|^2)^2}\Big)^{\half}\lesssim \frac{1}{|p_{(u,0)}||p_{(v,0)}|}
	\end{equs} 
	where in the last step we used~\eqref{e:gaink}. 
\end{itemize}

At this point we have summed over all the variables whose labels belong to directed edges in $D$ 
and it is not hard to check that we obtained the following
\begin{itemize}[leftmargin=*]
      \item the factors $|p_{(u,0)}|^{-1}$ provided by the bounds on red, green and purple edges 
      altogether give a contribution
        \begin{equ}
          \label{e:prodp}
          \frac1{\prod_{u=1}^4 |p_{(u,0)}|}.  
        \end{equ} 
        This is because~\eqref{e:rde} and~\eqref{e:gde} give such factor for the sum over variables 
        whose index is the base of a red and green edge respectively, and~\eqref{e:pde} give it for
        the sum over variables whose index is both the base and the tip of a purple edge. On the
        other hand, in each Feynman diagram, for each row there is exactly either
        the base of a green/red/purple edge, or the tip of a purple
        edge.  
    \item the factors depending on $\fs^u$ altogether take the form
      \begin{equs}
        \prod_{u=1}^4 \Big( \sum_{\substack{p_{(u,i)}\\ (u,i)\in V_U\cup V_D\setminus\{(u,i_u^2)\}}} |p_{(u,i_u^1)}|^2|{\fs}^{u}(p_{(u)\setminus\{(u,i_u^2)\}}))|^2\Big)^{\half}\,
      \end{equs}
and each of the factors {\it only depends} on $p_{(u,0)}$.
To see this, recall that
$\fs^u(p_{(u,i^1_u)}+p_{(u,i^2_u)},p_{(u)\setminus\mathcal V_u})$ in
\eqref{e:semplice} depends symmetrically on its $j^u-1$ arguments (where $j^u$ is defined in \eqref{e:ju}) and one
of the components of $p_{(u)\setminus\mathcal V_u}$ is $p_{(u,0)}$, whose sum has {\it not} been performed yet. 
The appearance of $|p_{(u,i_u^1)}|^2$ is due to the outcome of the bounds on blue, red and
orange edges. Indeed, one such factor appears on rows corresponding to tips of 
red/blue/orange edges or to the base of orange edges. 
On the other hand, in every Feynman diagram there is exactly one such vertex per row.
\end{itemize}

It remains to sum \eqref{e:trem5} over $p_{(u,0)},u=1,\dots,4$, that are the vertices of special edges. The factor \eqref{e:prodp} cancels exactly the same
        factor in the numerator of \eqref{e:trem5}.
        Then, \eqref{e:trem5} is upper bounded by a constant times
        \begin{equs}\label{e:kappa0almostfinal}
          \frac1{(\log N)^2}\sum_{p_{(u,0)},u=1,\dots,4} & \mathds{1}_{p_{(1,0)}=-p_{(2,0)},p_{(3,0)}=-p_{(4,0)}}\\
          &\times \prod_{u=1}^4 \Big( \sum_{\substack{p_{(u,i)}\\ (u,i)\in V_U\cup V_D\setminus\{(u,i_u^2)\}}} |p_{(u,i_u^1)}|^2|{\fs}^{u}(p_{(u)\setminus\{(u,i_u^2)\}}))|^2\Big)^{\half}\,.
        \end{equs}
Recall that $\fs^u=\ffNn_{a^u,j^u-1}$ and relabel the variables $p_{(u)\setminus\{(u,i_u^2)\}}=(p_{(u,0)}, \ell_{1:j^u-1})$. Now, we apply Cauchy-Schwarz once more 
to the sums with respect to  $p_{(u,0)},u=1,\dots,4$, so that the previous expression is upper bounded by 
\begin{equs}\label{e:kappa0final}          
\frac1{(\log N)^2}\sum_{p_{(1,0)},\ell_{1:j_1-2}} |\ell_1|^2|\ffNn_{a_1,j_1-1}(p_{(1,0)},\ell_{1;j_1-2})|^2
          \sum_{p_{(2,0)},\ell_{1:j_2-2}} |\ell_1|^2|\ffNn_{a_2,j_2-1}(p_{(2,0)},\ell_{1;j_2-2})|^2\\
          \le 
      \frac1{(\log N)^{2}}\|(-\gensy)^{1/2}\ffNn_{a_1,j_1-1}\|^2\|(-\gensy)^{1/2}\ffNn_{a_2,j_2-1}\|^2,
        \end{equs}
where we used the symmetry of $\ffNn$ in the last step. Since the right hand side corresponds to~\eqref{kappa02} 
for $\kappa=0$, this case is concluded. 
\medskip

\noindent{\bf The case $\kappa=2$}
\medskip 

\noindent Like for $\kappa=0$, also for $\kappa=2$ the ratio $({\sum_{\bv\in V[G]}|p_{\bv}|^2})^{-1}$ plays no 
role so that we will bound it by $1$. In this case though, we will need to keep track of the indicator function.  
More precisely, we need to control
\begin{equs}[e:trem6]
      \frac1{(\log N)^2}\sum_{\substack{p_\bv\\ \bv\in V}}&\1_{|p_{(1,0)}|,\,|p_{(3,0)}|\leq N}
      \prod_{((u,i),(v,j))\in
        \gamma}\mathds{1}_{p_{(u,i)}=-p_{(v,j)}}\\
        &\times
      \prod_{u=1}^4\frac{ |p_{(u,0)}|
        |p_{(u,i^1_u)}+p_{(u,i^2_u)}||{\fs}^{u}(p_{(u,i^1_u)}+p_{(u,i^2_u)},p_{(u)\setminus\mathcal
            V_u})|}{|x_{(u)}|^2} 
\end{equs}
where $(u,i^2_u)$, $u=1,\dots, 4$, is such that 
$i^2_u=0$, either for $u=1,\,3$ or $u=2,\,4$ and, without loss of generality, 
we assume that the first holds, i.e.
\begin{equ}[e:ass]
i^2_1=i^2_3=0,\,\qquad i^2_2=i^2_4\neq 0\,.
\end{equ} 

Getting back to~\eqref{e:trem6}, we first bound all the sums over variables whose indices 
are vertices in $V_U$ by removing  such variables from the denominator and applying 
Cauchy-Schwarz as in~\eqref{e:semplice}. Next, we turn to the sums 
whose indices are vertices of edges in $D$ and $S$.  

According to the conventions introduced for the previous case, we note that 
the edges in $S$ are all connected to those in $D$. Therefore, $D\cup S$ can be split into 
(at most four) connected components and each of the two edges in $S$ 
must be contained in exactly one of them (which  could be the same). 
First we bound the sums over the variables whose index is a vertex of an edge in
the connected components fully contained in $D$ (i.e. 
which do not contain any special edge), which are at most $2$ (and there could be none). 
We endow the edges with the same colours and directions 
defined above and apply the bounds 
in~\eqref{e:yde},~\eqref{e:rde},\eqref{e:gde},~\eqref{e:ode} and~\eqref{e:pde} accordingly.

Therefore, we are left with one or two connected components, each containing at least one 
special edge. In both cases, we will first provide direction and colour to those edges in the connected components 
which also lie in $D$, then we will bound the sums over the variables whose indices are vertices of these edges 
and at last we will estimate the sums over the variables indexed by vertices of edges in $S$. 

Assume the edges $s_1,\,s_2\in S$ lie in different connected components and denote them by 
$C_1,\,C_2$ respectively. 
The first edge $e\in D\cap C_i$ is the farthest from $s_1$ and it is assigned the colour blue. 
Those in between $e$ and $s_i$ (if any) are directed in successive order starting at $e$ and are 
all assigned the colour red. All edges in $C_1$ are directed towards $s_1$ (see Figure~\ref{fig:Feyn3}, left graph). 
If instead $s_1,\,s_2$ belong to the same connected component $C$, then 
we start from the edge $e\in D\cap C$ connected to, say, $s_1$. We impose its base point lies in the same 
rectangle to which a vertex of $s_1$ belongs (and its tip to a different rectangle) 
and assign to $e$ the colour orange. Once again, all the other edges in $D\cap C$ are ordered 
starting from $e$ directed towards $s_2$ and painted in red (see Figure~\ref{fig:Feyn3}, right graph).

\begin{figure}\begin{center}
\includegraphics[width=4cm]{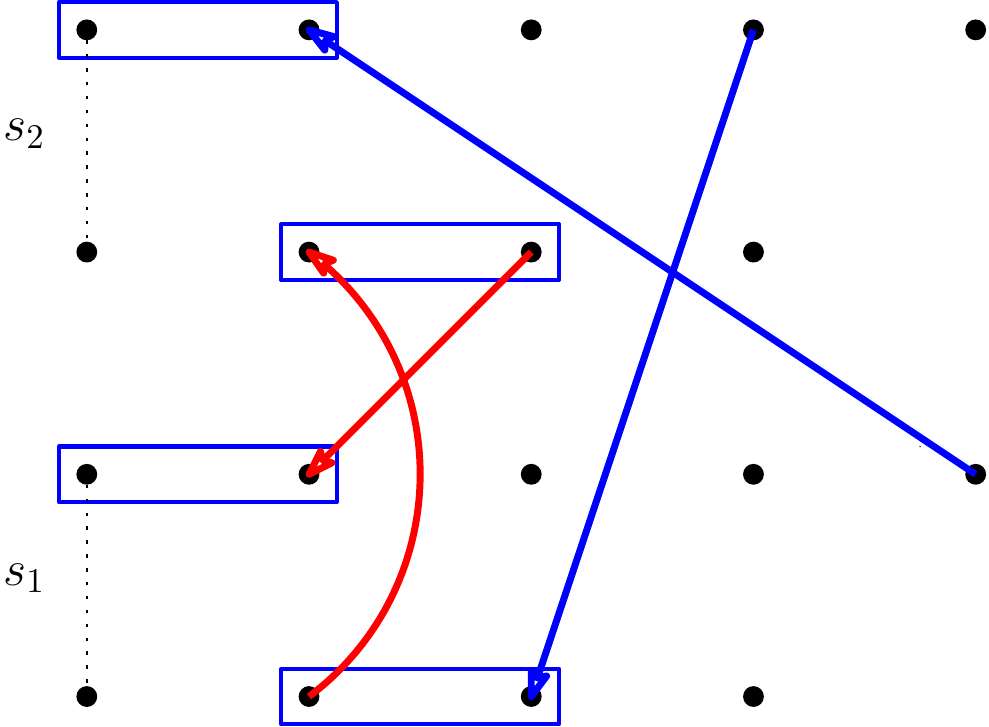}\qquad\qquad\qquad\includegraphics[width=4cm]{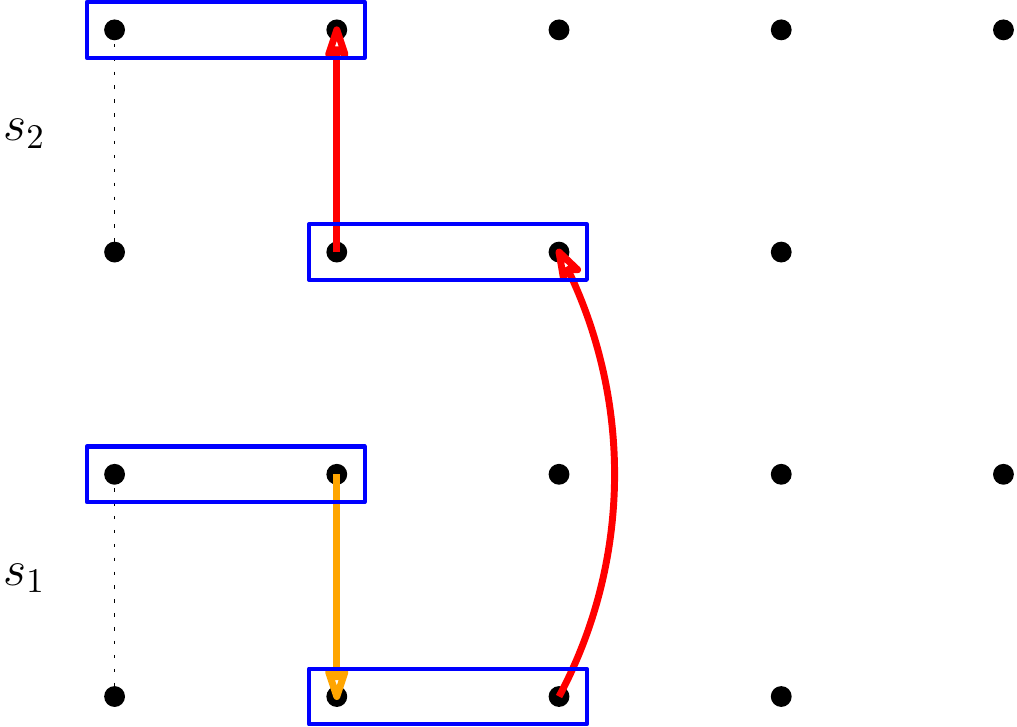}
\caption{The two possible types of connectivities of special edges in Feynman diagrams occurring when $\kappa=2$, as described in the text. For clarity, only directed and special edges $(s_1,s_2)$ are  drawn. In the left figure there are two connected components $C_1,C_2$ that we orient towards $s_1,s_2$. In the right drawing there is a single component, that we orient from $s_1$ to $s_2$. Special edges are not oriented. In some Feynman diagrams, also closed loops of oriented edges (and therefore purple/orange edges) can be present.
}
\label{fig:Feyn3}
\end{center}
\end{figure}

We now upper bound~\eqref{e:trem6} by removing from
the denominator all the variables whose index is either a vertex on
an orange or a blue edge, or at the tip of a red edge. 
Let us stress that in this way we get a prefactor of the form 
\begin{equ}[e:prefactor]
 \frac{1}{|p_{(1,0)}|^2|p_{(3,0)}|^2}
\end{equ}
since the rectangles in which $(1,0)$ and $(3,0)$ lie can only contain the tip of a blue or red edge, 
or the base point of an orange one. This prefactor cancels exactly the factor $\prod_{u=1}^4|p_{(u,0)}|$ in \eqref{e:trem6}, because $|p_{(1,0)}|=|p_{(2,0)}|,|p_{(3,0)}|=|p_{(4,0)}|$.
Then, we perform the sums over the variables indexed by the vertices of directed edges in $C_1,C_2$ (or just $C$ in the second case), 
according to the order  just introduced. We bound 
the sum of variables indexed by vertices of a blue edge via~\eqref{e:yde}, those 
indexed by vertices of a red edge via~\eqref{e:rde} and those 
indexed by vertices of an orange edge via~\eqref{e:ode}. 
\medskip

At this point we have estimated the sums over all the variables whose labels belong to directed edges in $D$, 
it can be directly checked that we obtained the following

\begin{itemize}[leftmargin=*]
\item the factors $|p_{(u,0)}|^{-1}$ provided by the bounds on red edges in $C$ (or $C_1,C_2$) and on 
the red, green and purple edges on the open or closed connected components of $D$ disconnected from the special edges, altogether 
give a contribution of the form
\begin{equ}[e:prodp2]
\frac{1}{|p_{(2,0)}||p_{(4,0)}|}=\frac{1}{|p_{(1,0)}||p_{(3,0)}|}\,.
\end{equ}
This is a consequence of~\eqref{e:ass}, which imposes $i^2_1=i^2_3=0$,
and the bounds~\eqref{e:rde},~\eqref{e:gde} and~\eqref{e:pde}. Indeed,
as noted above,~\eqref{e:pde} gives  a factor $|p_{(u,0)}|^{-1}$  both for $u$  the row containing base and the row containing the tip of a purple edge
but, if in our graph there is a purple edge then this necessarily
belongs to a closed  connected component of $D$, disconnected from the special edges, 
and by~\eqref{e:ass} it will involve vertices on the rows $2$ and $ 4$.
Instead,~\eqref{e:rde} and~\eqref{e:gde} give a factor
$|p_{(u,0)}|^{-1}$ for $u$ whose row contains  the {\it
  base} of a red/green edge and, by construction, in the rectangles containing
the vertices of special edges there is either a tip of a (red or
blue) edge or the base of an orange edge.  On the other hand, in
each such Feynman diagram, each of the rows $2$ and $4$ will have
precisely one base point of a green/red edge, or one the base point
and one the tip of a purple edge.
\item the factors depending on $\fs^u$ altogether take the form
      \begin{equs}
       &\frac{\1_{|p_{(1,0)}|,\,|p_{(3,0)}|\leq N}}{|p_{(1,0)}||p_{(3,0)}|} \prod_{u=2,4} \Big( \sum_{\substack{p_{(u,i)}\\ (u,i)\in V_U\cup V_D\setminus\{(u,a_u^2)\}}} |p_{(u,i_u^1)}|^2|{\fs}^{u}(p_{(u)\setminus\{(u,i_u^2)\}}))|^2\Big)^{\half}\\
       &\qquad\qquad\qquad\times  \sum_{p_{(1,0)},\ell_{1:j_1-2}} |\ell_1|^2|\ffNn_{a_1,j_1-1}(p_{(1,0)},\ell_{1;j_1-2})|^2\\
       &\leq\frac{\1_{|p_{(1,0)}|,\,|p_{(3,0)}|\leq N}}{|p_{(1,0)}||p_{(3,0)}|} \prod_{u=2,4} \Big( \sum_{\substack{p_{(u,i)}\\ (u,i)\in V_U\cup V_D\setminus\{(u,i_u^2)\}}} |p_{(u,i_u^1)}|^2|{\fs}^{u}(p_{(u)\setminus\{(u,a_u^2)\}}))|^2\Big)^{\half}\\
       &\qquad\qquad\qquad\qquad\qquad\qquad\times  \|(-\gensy)^{1/2}\ffNn_{a_1,j_1-1}\|^2
     \end{equs}
Note that the product in the first line (corresponding to rows $2$ and $4$ of the graph) 
has still to be summed over $p_{(2,0)}$ and $p_{(4,0)}$. 
On the other hand, in rows $1,3$ all variables have already been summed over 
and no dependency on $p_{(1,0)}$ and $p_{(3,0)}$ is left. 
This is a consequence of the fact that the bounds \eqref{e:yde}, \eqref{e:rde}, \eqref{e:ode} 
obtained when summing over red/blue/orange edges do not depend on the second variable contained in the same rectangle whose index is the vertex at the tip of the blue/red edge, 
or at the base of the orange edge. Such variables are exactly $p_{(1,0)},\,p_{(3,0)}$.
% \fabioText{togliere la frase seguenteThe first factor is due to the fact that the term in~\eqref{e:prefactor} cancels out $|p_{(1,0)}||p_{(3,0)}|$ 
% at the numerator of~\eqref{e:trem6}. 
\end{itemize}

We now perform the sums over $p_{(u,0)}$, $u=1,\dots, 4$ and conclude the proof of the case $\kappa=2$. 
We have that~\eqref{e:trem6} is bounded by 
\begin{equs}
  & \frac{1}{(\log N)^2}\sum_{p_{(u,0)}}  \frac{\1_{|p_{(1,0)}|,\,|p_{(3,0)}|\leq N}}{|p_{(1,0)}||p_{(3,0)}|}\mathds{1}_{p_{(1,0)}=-p_{(2,0)},p_{(3,0)}=-p_{(4,0)}}\\
  &\qquad\times \prod_{u=2,4} \Big( \sum_{\substack{p_{(u,i)}\\ (u,i)\in V_U\cup V_D\setminus\{(u,a_u^2)\}}} |p_{(u,i_u^1)}|^2|{\fs}^{u}(p_{(u)\setminus\{(u,i_u^2)\}}))|^2\Big)^{\half}\|(-\gensy)^{1/2}\ffNn_{a_1,j_1-1}\|^2\\
  &\leq \frac1{(\log N)^2} \Big(\sum_{|p_{(1,0)}|\leq N}\frac{1}{|p_{(1,0)}|^2}\Big)^{\half}\Big(\sum_{|p_{(3,0)}|\leq N}\frac{1}{|p_{(3,0)}|^2}\Big)^{\half}\prod_{i=1}^2\|(-\gensy)^{1/2}\ffNn_{a_1,j_i-1}\|^2\\
  &\lesssim\frac{1}{\log
    N}\prod_{i=1}^2\|(-\gensy)^{1/2}\ffNn_{a_i,j_i-1}\|^2
\end{equs}
where we used that 
\begin{equ}[e:losslog]
\sum_{|\ell|\leq N}\frac{1}{|\ell|^2}\lesssim\int_1^N\frac{\dd r}{r}=\log N\,.
\end{equ}
Consequently,~\eqref{kappa02} for $\kappa=2$ follows at once. 
\medskip

\noindent{\bf The case $\kappa=4$}
\medskip 

\noindent When $\kappa=4$, all the vertices $(u,0)$, $u=1,\dots,4$ belong to a rectangle 
so that 
\begin{equ}[e:ass4]
i_u^2=0\,,\qquad \text{for all $u=1,\dots 4$.}
\end{equ} 
Before performing any bound let us look at the edges in $D\cup S$. Adopting the same conventions as above, 
there are two scenarios we need to 
consider - either $D\cup S$ has only one connected component 
or two. Let us begin with the first.  

Note that the unique connected component of $D\cup S$ can only be an open path or a cycle. 
If it is an open path then $D$ has exactly two edges with one vertex in a rectangle 
and one outside (and one edge connecting two rectangles). 
We endow the former two  with a direction by imposing their tip lies in the rectangle 
and paint them blue. 
If instead the connected component of $D\cup S$ is a cycle then $D$ has cardinality $2$. We pick one of 
its two edges, assign it a direction (no matter which one) and paint it orange. 
In both cases, we have left aside an edge $e=((\bar u,i_{\bar u}^1),(\bar v,i_{\bar v}^1))\in D$ 
(the dashed line in Figure \ref{fig:Feyn4})
which is necessarily such that either $\bar u\in\{1,2\}$ 
and $\bar v\in\{3,4\}$, or $\bar v\in\{1,2,\}$ and $\bar u\in\{3,4\}$. Without loss of generality we assume that
\begin{equ}[e:ass41]
\bar u\in\{1,2\}\qquad\text{ and }\qquad\bar v\in\{3,4\}\,.
\end{equ} 

\begin{figure}\begin{center}
\includegraphics[width=4cm]{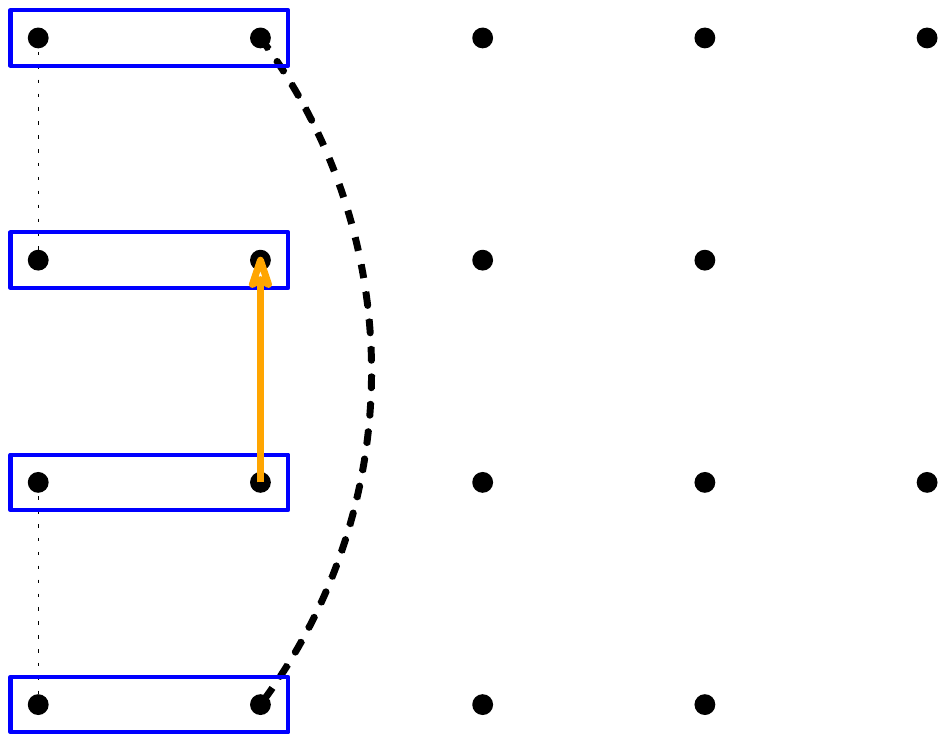}\qquad\qquad\qquad\includegraphics[width=4cm]{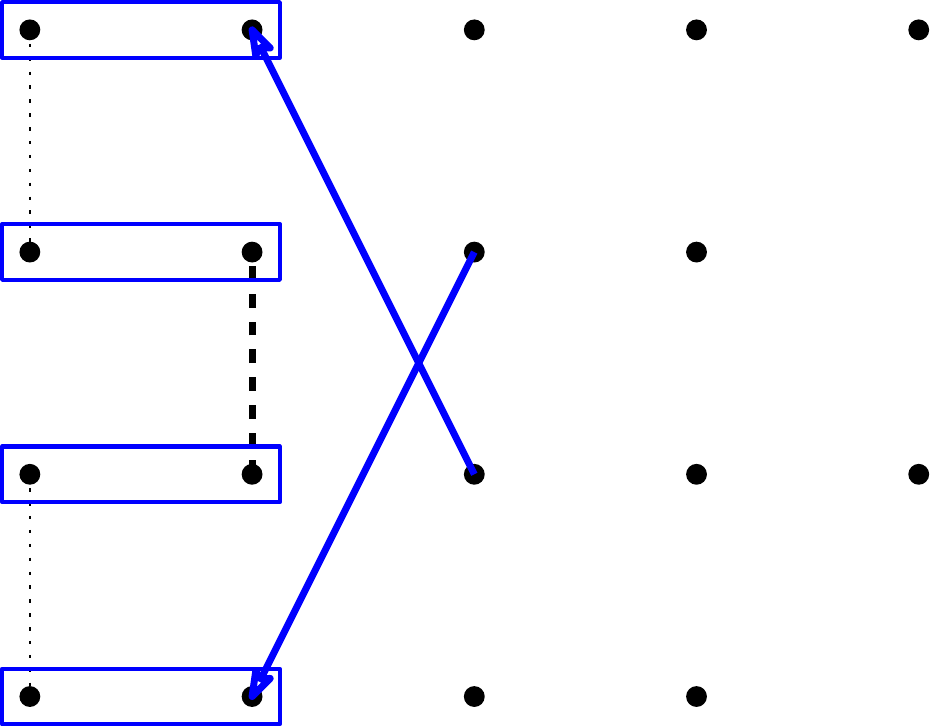}
\caption{The two different possible situations in the case $\kappa=4$ and $D\cup S$ has a single connected components. The thick dashed edge is the edge $e=((\bar u,i^1_{\bar u}),(\bar v,i^1_{\bar v}))$ that we treat separately. We don't need to assign it a direction and a color.
}
\label{fig:Feyn4}
\end{center}
\end{figure}

We now get back to the sums in~\eqref{e:trem4}. 
Thanks to~\eqref{e:ass41}, $(\bar u,i_{\bar u}^1)\in V[G]$, hence we can estimate  
\begin{equ}[e:ration]
\frac{1}{\sum_{\bv\in V[G]}|p_{\bv}|^2}\leq \frac{1}{|p_{(\bar u,i_{\bar u}^1)}|^2}\,.
\end{equ}
We only keep the indicator function imposing
$|p_{(\bar u,i_{\bar u}^1)}|,\,|p_{(\bar v,i_{\bar v}^1)}|\leq N$ and
bound the sums over all variables whose indices are vertices of edges
in $U$ as in~\eqref{e:semplice}.  For the sums over the variables
whose indices are vertices of the two blue or the orange edge, we
remove these variables from the denominator and exploit~\eqref{e:yde}
or~\eqref{e:ode}, respectively.

Before summing over $p_{(\bar u,i_{\bar u}^1)}$ and $p_{(\bar v,i_{\bar v}^1)}$, we focus on 
the sum over the variables whose indices are vertices of the special edges. 
Let $\bar x$ be either $\bar u$ or $\bar v$ and $\{x\}$ be $\{1,2\}\setminus\{\bar u\}$ if $\bar x=\bar u$ and
$\{3,4\}\setminus\{\bar v\}$ if $\bar x=\bar v$. Then, we need to control 
\begin{equ}[e:k=41]
\sum_{p_{(x, 0)},\,p_{(\bar x,0)}} \1_{p_{(x, 0)}=-p_{(\bar x,0)}}\frac{|p_{(x, 0)}||p_{(\bar x, 0)}||p_{(\bar x, 0)}+p_{(\bar x, i_{\bar x}^1)}| \check{\fs}^{\bar x}(p_{(\bar x, 0)}+p_{(\bar x, i_{\bar x}^1)})}{|p_{(x,0)}|^2(|p_{(\bar x,0)}|^2+|p_{(\bar x,i_{\bar x}^1)}|^2)}
\end{equ}
where 
\begin{equ}[e:scheck]
  \check{\fs}^{\bar x}(p_{(\bar x, 0)}+p_{(\bar x, i_{\bar
      x}^1)})\eqdef \Big( \sum_{\substack{p_{(\bar x,i)}\\(\bar
      x,i)\in V_U}} |{\fs}^{\bar x}(p_{(\bar x, 0)}+p_{(\bar
    x,i^1_{\bar x})},p_{(\bar x)\setminus\mathcal V_{\bar
      x}})|^2\Big)^{\half}\,.
\end{equ}
To see how~\eqref{e:k=41} arises, it suffices to note that the sums we have bounded up to now involving 
variables whose index is a vertex in the $\bar x$-th row never considered $V_D$ or $V_S$-indexed variables, 
while the dependence of the kernel $\fs^x$ on $p_{(x,0)}$ was lost because of~\eqref{e:yde} or~\eqref{e:ode}. 

Now,~\eqref{e:k=41} equals
\begin{equs}[e:imp]
\sum_{\ell} &\frac{|\ell+p_{(\bar x, i_{\bar x}^1)}| \check{\fs}^{\bar x}(\ell+p_{(\bar x, i_{\bar x}^1)})}{|\ell|^2+|p_{(\bar x,i_{\bar x}^1)}|^2}\\
&\leq \Big(\sum_{\ell} \frac{1}{(|\ell|^2+|p_{(\bar x,i_{\bar x}^1)}|^2)^2}\Big)^{\half}\Big(\sum_{\ell}|\ell|^2\check{\fs}^{\bar x}(\ell+p_{(\bar x, i_{\bar x}^1)})^2\Big)^{\half}\lesssim \frac{1}{|p_{(\bar x,i_{\bar x}^1)}|} \|(-\gensy)^{\half}\fs^{\bar x}\|
\end{equs}
where in the last step we used~\eqref{e:gaink}, the definition of $\check{\fs}$ in~\eqref{e:scheck} and the 
symmetry of $\fs$. 

By collecting what we obtained so far, we see that~\eqref{e:trem4} is upper bounded by 
\begin{equs}
\|(-\gensy)^{1/2}&\ffNn_{a_1,j_1-1}\|^2\|(-\gensy)^{1/2}\ffNn_{a_2,j_2-1}\|^2\frac{1}{(\log N)^2}\sum_{0<|p_{(\bar u,i_{\bar u}^1)}|,\,|p_{(\bar v,i_{\bar v}^1)}|\leq N}\frac{\1_{p_{(\bar u,i_{\bar u}^1)}=-p_{(\bar v,i_{\bar v}^1)}}}{|p_{(\bar u,i_{\bar u}^1)}|^3|p_{(\bar v,i_{\bar v}^1)}|}%% \\
%% &=\|(-\gensy)^{1/2}\ffNn_{a_1,j_1-1}\|^2\|(-\gensy)^{1/2}\ffNn_{a_2,j_2-1}\|^2\frac{1}{(\log N)^2}\sum_{0<|\ell|\leq N}\frac{1}{|\ell|^4}
\\
&\lesssim \frac{1}{(\log N)^2}\|(-\gensy)^{1/2}\ffNn_{a_1,j_1-1}\|^2\|(-\gensy)^{1/2}\ffNn_{a_2,j_2-1}\|^2
\label{e:esaurito}
\end{equs}
the last inequality being a consequence of the fact that $\sum_{\ell\ne0} |\ell|^{-4}<\infty$. 
\medskip

We now turn to the scenario in which $D\cup S$ has two connected components, for which 
the factor $({\sum_{\bv\in V[G]}|p_{\bv}|^2})^{-1}$ will be crucial in controlling~\eqref{e:trem4}. 
Recall that this factor contains the sum over vertices of edges in $P\neq \emptyset$ which connect  
$V[G]$ to $V[\tilde G]$. 

We first consider the case $P\cap D=\emptyset$, which implies that $P\subset U$ as, 
clearly, $P\cap S=\emptyset$. 
Pick arbitrarily some $e=((\bar u,i),(\bar v,j))\in P$ so that $(\bar u,i)\in V[G]$ and brutally bound 
\begin{equ}[e:L0]
\frac{1}{\sum_{\bv\in V[G]}|p_{\bv}|^2}\leq \frac{1}{|p_{(\bar u,i)}|}\,.
\end{equ}
Then, we estimate the sums over all variables whose indices are vertices of edges in $U\setminus\{e\}$ 
as in~\eqref{e:semplice} while for the sum over $p_{(\bar u,i)},\,p_{(\bar v,j)}$ we use the bound 
\begin{equs}[e:wye]
\sum_{p_{(\bar u,i)}, p_{(\bar v,j)}} \frac{\1_{p_{(\bar u,i)}=-p_{(\bar v,j)}}}{|p_{(\bar u,i)}|} f_{\bar u}(p_{(\bar u,i)})&f_{\bar v}(p_{(\bar v,j)})=\sum_{p_{(\bar u,i)}, p_{(\bar v,j)}} \frac{\1_{p_{(\bar u,i)}=-p_{(\bar v,j)}}}{\sqrt{|p_{(\bar u,i)}||p_{(\bar v,j)}|}} f_{\bar u}(p_{(\bar u,i)})f_{\bar v}(p_{(\bar v,j)})\\
&\leq \Big(\sum_{p_{(\bar u,i)}} \frac{ f_{\bar u}(p_{(\bar u,i)})^2}{|p_{(\bar u,i)}|}\Big)^{\half} \Big(\sum_{p_{(\bar v,j)}} \frac{ f_{\bar v}(p_{(\bar v,j)})^2}{|p_{(\bar v,j)}|}\Big)^{\half}\,.
\end{equs}
\begin{remark}
  \label{rem:j21}
Note that necessarily $j^{\bar u},j^{\bar v}>2$, otherwise the edge $e\in P$ could not connect rows $\bar u$, $\bar v$, since by definition it is has no endpoint in the rectangles of those rows.  
\end{remark}
We now turn to the sums over the variables whose indices are vertices
of edges in $D$.  Adopting the same conventions as above, we endow
each of these edges with a direction and a colour.  Given the
structure of the Feynman graphs $\gamma$ in~\eqref{e:Feynman}, either
$|D|=4$ and therefore all of its edges have one vertex in a rectangle
and one outside, or $|D|=2$ so that both edges have both vertices
lying in two different rectangles.  In the first case, we direct the
edges in such a way that the tip is in the rectangle (and the base
outside) and assign them the colour blue, in the other we arbitrarily
choose a direction and paint it orange.  See Figure \ref{fig:Feyn5}.
\begin{figure}
  \begin{center}\includegraphics[width=4cm]{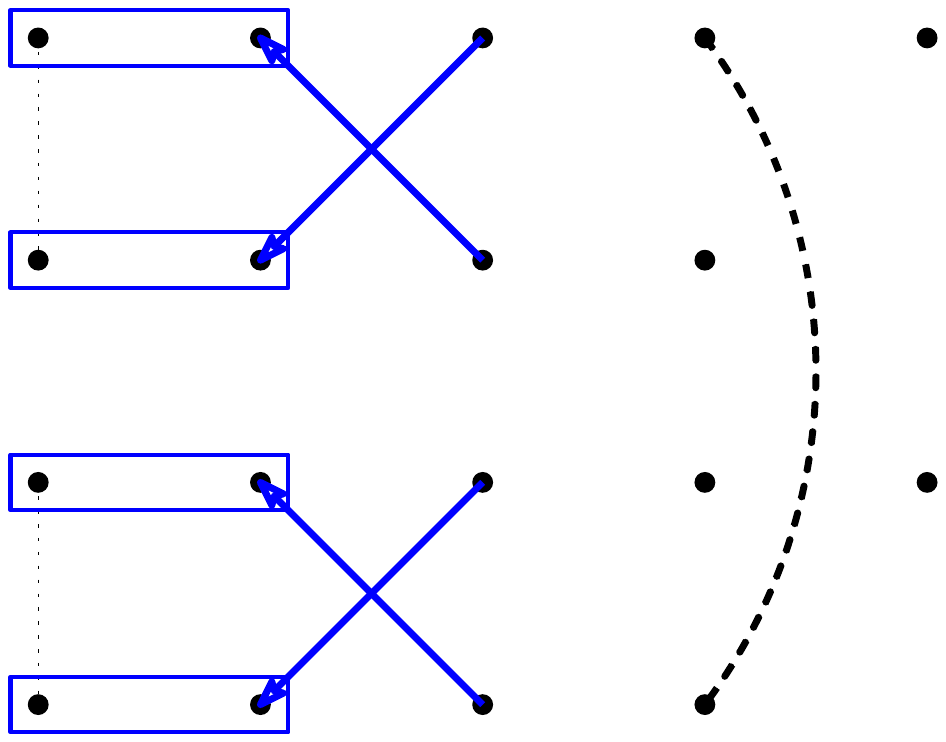}\qquad\qquad\qquad\includegraphics[width=4cm]{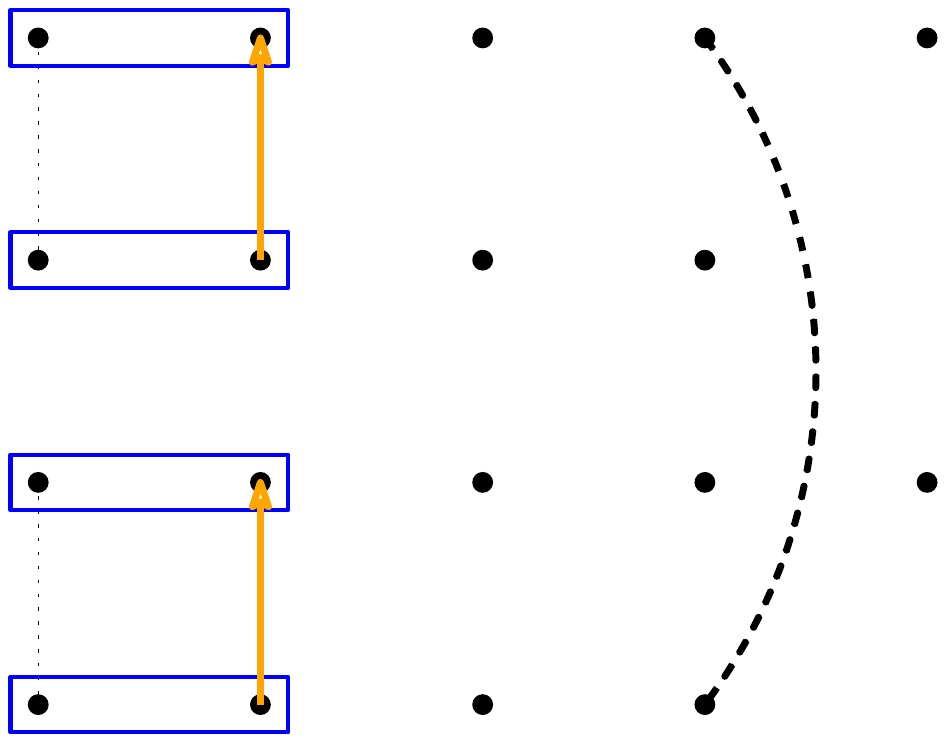}
    \caption{The case $\kappa=4$, when $D\cup S$ has two components  and $P\cap D=\emptyset$. The thick dashed edge is the selected edge in $P$. In the left (resp. right) drawing the case where $|D|=4$ (resp. $|D|=2$).}
\label{fig:Feyn5}
\end{center}
\end{figure}

As usual, we upper bound~\eqref{e:trem4} by removing from
the denominator all the variables whose index is either a vertex of
an orange or blue edge. As a consequence, the denominators become
\begin{equ}[e:prefactor2]
\prod_{u=1}^4 \frac{1}{|p_{(u,0)}|^2}
\end{equ}
Then, we estimate the sums over variables whose indices are vertices of blue edges by~\eqref{e:yde}, 
while those over variables whose indices are vertices of orange edges by~\eqref{e:ode}. 
\medskip 

Before summing over the remaining variables (that are those with
labels on the two special edges) we consider the case in which
$P\cap D\neq \emptyset$.  As was done in all other cases apart from the
previous one, we upper bound~\eqref{e:trem4} by removing from the
denominator all the variables indexed by vertices in $V_U$ and
applying~\eqref{e:semplice}.  Next, let us pick one edge
$e=((\bar u,i), (v,i_v^1))\in P\cap D$ which is necessarily such that
one between $\bar u$ and $v$ belongs to $\{1,2\}$ and the other is in
$\{3,4\}$, and without loss of generality we assume that
$\bar u\in \{1,2\}$.
\begin{remark}
  \label{rem:j22}
  Note that necessarily $j^{\bar u}>2$, otherwise $e$ would have both endpoints in a rectangle and $D\cup S$ would have a single connected component.
\end{remark}
We bound the sum over $p_{(\bar u,i)}$ and $p_{(v,i_v^1)}$, by 
removing the second variable from the denominators and brutally estimating 
\begin{equ}
\frac{1}{\sum_{\bv\in V[G]}|p_{\bv}|^2}\leq \frac{1}{\sqrt{|p_{(\bar u,i)}|}}\
\end{equ}
so that, recalling~\eqref{e:scheck},  we are left to control
\begin{equs}[e:wdye]
\sum_{p_{(\bar u,i)},\,p_{(v,i_v^1)}}\frac{\1_{p_{(\bar u,i)}=-p_{(v,i_v^1)}}}{\sqrt{|p_{(\bar u,i)}|}}& f_{\bar u}(p_{(\bar u,i)})|p_{(v,0)} + p_{(v,i_v^1)}|\check{\fs}^v(p_{(v,0)} + p_{(v,i_v^1)})\\
&\leq \Big(\sum_{p_{(\bar u,i)}} \frac{ f_{\bar u}(p_{(\bar u,i)})^2}{|p_{(\bar u,i)}|}\Big)^{\half}\Big(\sum_{p_{(v,a_v^1)}} |p_{(v,i_v^1)}|^2\check{\fs}^v(p_{(v,i_v^1)})^2\Big)^{\half}
\end{equs}
where $f_{\bar u}$ is given by the $\bar u$-th term at the right hand side of~\eqref{e:semplice} 
and above, we only highlighted the dependence on the summed variable $p_{(\bar u,i)}$. 

We now turn to the sums over the variables whose indices are vertices of edges in $D\setminus\{e\}$. 
Adopting the same conventions as above, 
we endow each of these edges with a direction and a colour. 
If an edge has only one vertex in a rectangle then we orient it in such a way that its tip 
belongs to the rectangle and assign it the colour blue. If instead both of its vertices lie in 
different rectangles then we choose one of the two possible directions and paint it orange. 
We then bound the remaining sums by removing from the denominators all 
variables whose indices are vertices of a blue/orange edge (thus obtaining the same factor 
as in~\eqref{e:prefactor2}) and exploiting~\eqref{e:yde} and~\eqref{e:ode}. 

In conclusion, in the case $\kappa=4,D\cup S$ having two components and $P\cap D$ being empty or not, we have summed over all variables except to the ones ($p_{(u,0)},u\le 4$) corresponding to special edges and  it is easily checked that:
\begin{itemize}[leftmargin=*]
\item the factors containing the variables $p_{(u,0)}$, $u=1,\dots, 4$ overall give a contribution of the form
\begin{equ}
\prod_{i=1}^4 \frac{1}{|p_{(u,0)}|}\,,
\end{equ}
which is the combined effect of 
% To see this, we can argue as in the case $\kappa=2$, using though~
\eqref{e:prefactor2} and of the factor $\prod_{u=1}^4|p_{(u,0)}|$ in \eqref{e:trem4}.
\item the expression~\eqref{e:trem4} is upper bounded by
  \begin{equs}
    \label{e:k4emptyset}
    \frac1{(\log N)^2} \|\cS^\half \ffNn_{a^{\bar u},j^{\bar u}-1}\|
    \|\cS^\half \ffNn_{a^{\bar v},j^{\bar v}-1}\|\Big(
    \prod_{u\in\{1,\dots,4\}\setminus\{\bar u,\bar v\}} \|
    (-\gensy)^{1/2}\ffNn_{a^u,j^{u}-1}\|\Big)\\\times
    \sum_{\substack{|p_{(u,0)}|\leq N\\u=1,\dots 4}}
    \1_{\substack{p_{(1,0)}=-p_{(2,0)}\\p_{(3,0)}=-p_{(4,0)}}}\prod_{u=1}^4\frac{1}{|p_{(u,0)}|}
  \end{equs}
if $D\cap P=\emptyset$  (recall that in this case $\bar u,\bar v$ are the rows containing the endpoints of the dashed line in Figure \ref{fig:Feyn4}) and  by 
  \begin{equs}
    \label{e:k4notemptyset}
  \frac1{(\log N)^2}   \|\cS^\half \ffNn_{a^{\bar u},j^{\bar u}-1}\| \Big( \prod_{u\in\{1,\dots,4\}\setminus\{\bar u\}} \| (-\gensy)^{1/2}\ffNn_{a^u,j^{u}-1}\|\Big)\\\times \sum_{\substack{|p_{(u,0)}|\leq N\\u=1,\dots 4}} \1_{\substack{p_{(1,0)}=-p_{(2,0)}\\p_{(3,0)}=-p_{(4,0)}}}\prod_{u=1}^4\frac{1}{|p_{(u,0)}|}
  \end{equs}
  if instead $D\cap P\neq\emptyset$ (in this case, $\bar u$ is the
  same index that appears in \eqref{e:wdye}).
  Here, $\cS$ is the operator defined in \eqref{e:S}. Its origin  is in the denominators $|p_{(\bar u,i)}|,|p_{(\bar v,j)}|$ in \eqref{e:wye}, \eqref{e:wdye} and the blue/orange edge estimate that is applied afterwards. 
  Note also, recalling Remarks \ref{rem:j21}, \ref{rem:j22} that $j^{\bar u}-1,j^{\bar v}-1$ in 
  \eqref{e:k4emptyset} and $j^{\bar u}-1$ in
  \eqref{e:k4notemptyset} are all strictly larger than $1$, so indeed we can 
use definition \eqref{e:S}.\end{itemize}
It remains  to estimate the sum over $p_{(u,0)},u=1,\dots,4$ in \eqref{e:k4emptyset}, which equals
\begin{equ}
\Big(\sum_{0<|\ell|\leq N}\frac{1}{|\ell|^2}\Big)^2\lesssim (\log N)^2.
\end{equ}
Altogether, the bounds \eqref{e:esaurito}, \eqref{e:k4emptyset} and \eqref{e:k4notemptyset}, which apply to the three subcases into which we have split the case $\kappa=4$, immediately imply the statement of Proposition \ref{prop:K0} in the case $\kappa=4$.

\begin{appendix}

\section{Some technical results on the replacement argument}\label{a:Approx}

The present appendix is devoted to the proof of the following
proposition that allows us to replace the complicated sums appearing
in the expression for the diagonal terms of the operators we consider
with simpler integrals.

 In the whole section, $\fc$ is given by~\eqref{e:nueff} and $H$ and $H^+$ 
 satisfy the assumptions of Lemma~\ref{l:GenBound}. Since the argument of $H,H^+$ is always $\Ll(x)$ for some $x\ge 1/2$ and $\Ll(x)\in [0,2\fc]$ for $x\ge 1/2$ (see \eqref{e:L}),
 we can assume that $H,H^+$ are defined just on $[0,2\fc]$ and that \eqref{e:KAPPA} holds for some constant $K$ depending only on $\fc$.

\begin{proposition}\label{p:Approx}
 For $N\in\N$, $\mu\geq 0$, $n\in\N$ and $k_{1:n}\in\Z^{2n}\setminus\{0\}$, define  $P^N$ as 
\begin{equs}[e:J]
P^N&(\mu,k_{1:n})\eqdef\frac{4\hat{\lambda}^2}{\log N} \sum_{\ell + m=k_1}(\nonlin_{\ell,m})^2\times\\
& \times \frac{\mu+\tfrac12(|\ell|^2+|m|^2+|k_{2:n}|^2)H^+(\Ll(\mu+\tfrac12(|\ell|^2+|m|^2+|k_{2:n}|^2)))}{[\mu +\tfrac12(|\ell|^2+|m|^2+|k_{2:n}|^2)H(\Ll(\mu+\tfrac12(|\ell|^2+|m|^2+|k_{2:n}|^2)))]^2}
\end{equs}
where $\nonlin$ and  $\Ll$ are given as in~\eqref{e:nonlinCoefficient} and~\eqref{e:L}, respectively.
Then, there exists a constant $C=C(n,\fc)$ such that 
\begin{equ}[e:Approx]
\sup_{\substack{k_{1:n}\in\Z^{2n}\setminus\{0\}\\ \mu\geq 0}}\Big|P^N(\mu,k_{1:n})-\int_0^{\Ll\left(\mu+\tfrac12|k_{1:n}|^2\right)}\frac{H^+(y)}{H(y)^2}\dd y  \Big| \leq C\eps_N
\end{equ}
where $\eps_N$ goes to $0$ as $N\to\infty$ uniformly over $n$, $\fc$ and $K$ as in Lemma~\ref{l:GenBound}.
\end{proposition}

In order to lighten the exposition, let us introduce some notations. 
Set
\begin{equ}[e:conv]
\mu_N\eqdef \mu/N^2\,,\qquad \beta_N\eqdef |k_{2:n}/N|^2\,,\qquad \alpha_N\eqdef \mu_N+\tfrac12 |k_{1:n}/N|^2
\end{equ}
and, for $x\in\R^2$ such that $|x|\leq 1$, 
\begin{equ}[e:Gammas]
\Gamma(x)\eqdef \tfrac12 (|x|^2+|k_1/N-x|^2+\beta_N)\,,\quad\Gamma_1(x)\eqdef \mu_N+\Gamma(x)\,,\quad \Gamma_2(x)\eqdef |x|^2 +\alpha_N\,.
\end{equ}
For future reference, we point out that $\Gamma_1$ and $\Gamma_2$ are comparable in that, by the triangular inequality, 
we have
\begin{equ}[e:lm]
\Gamma_2(x)\lesssim \Gamma_1(x)\lesssim \Gamma_2(x)\,
\end{equ}
for all $x\in\R^2$ and $k_{1:n}\in\Z^{2n}$.  The bulk of the proof of
Proposition~\ref{p:Approx} is contained in the following technical
lemma, which allows us to substitute the Riemann-sum
approximation of $P^N$ with a more tractable integral.

\begin{lemma}\label{l:MainRepl}
Under the assumptions of Proposition~\ref{p:Approx}, the quantity
\begin{equs}[e:MainRepl]
\Big|\int_{\R^2}  \dd x\,\,&(\nonlin_{xN,k_1-xN})^2 \frac{\mu_N+\Gamma(x)H^+(\Ll(N^2\Gamma_1(x)))}{[\mu_N +\Gamma(x)H(\Ll(N^2\Gamma_1(x)))]^2}\\
&-\int_{\R^2}  \dd x\,\,(\nonlin_{Nx,-Nx})^2 \frac{H^+(\Ll(N^2\Gamma_1(x)))}{\Gamma_2(x)(\Gamma_2(x)+1)H(\Ll(N^2\Gamma_2(x)))^2}\Big|
\end{equs} 
is bounded uniformly over $N$, $\mu$ and $k_{1:n}\in\Z^{2n}\setminus\{0\}$. 
\end{lemma}
\begin{proof}
In order to prove the lemma, we will add and subtract a number of terms and 
show that the resulting integrals are bounded uniformly over 
$N$, $\mu$ and $k_{1:n}\in\Z^{2n}\setminus\{0\}$. 
More precisely, we will subsequently replace each factor in the first integral of~\eqref{e:MainRepl}, 
with the corresponding term of the second and we will do so in the following order. 
We begin by turning every $\Gamma$ into a $\Gamma_1$
\begin{align}
\begin{split}\label{e:Repl1}
&\int_{\R^2}  \dd x\,\,(\nonlin_{xN,k_1-xN})^2 \Big|\frac{\mu_N+\Gamma(x)H^+(\Ll(N^2\Gamma_1(x)))}{[\mu_N +\Gamma(x)H(\Ll(N^2\Gamma_1(x)))]^2} -\frac{H^+(\Ll(N^2\Gamma_1(x)))}{\Gamma_1(x)H(\Ll(N^2\Gamma_1(x)))^2}\Big|\,,
\end{split}
\intertext{then every $\Gamma_1$ into $\Gamma_2$}
\begin{split}\label{e:Repl2}
&\int_{\R^2}  \dd x\,\,(\nonlin_{xN,k_1-xN})^2 \Big|\frac{H^+(\Ll(N^2\Gamma_1(x)))}{\Gamma_1(x)H(\Ll(N^2\Gamma_1(x)))^2} -\frac{H^+(\Ll(N^2\Gamma_2(x)))}{\Gamma_2(x)H(\Ll(N^2\Gamma_2(x)))^2}\Big|
\end{split}
\intertext{we insert a factor which will be needed to apply a change of variables (see~\eqref{e:CoV}) }
\begin{split}\label{e:Repl3}
&\int_{\R^2}  \dd x\,\,(\nonlin_{xN,k_1-xN})^2 \frac{H^+(\Ll(N^2\Gamma_2(x)))}{\Gamma_2(x)H(\Ll(N^2\Gamma_2(x)))^2}\Big|1-\frac{1}{\Gamma_2(x)+1}\Big|
\end{split}
\intertext{and at last we modify the arguments of $\nonlin$}
\begin{split}\label{e:Repl4}
&\int_{\R^2}  \dd x\,\,\Big|(\nonlin_{xN,k_1-xN})^2-(\nonlin_{xN,-xN})^2\Big| \frac{H^+(\Ll(N^2\Gamma_2(x)))}{\Gamma_2(x)(\Gamma_2(x)+1)H(\Ll(N^2\Gamma_2(x)))^2}\,.
\end{split}
\end{align}
The proof consists of showing that each of~\eqref{e:Repl1}-\eqref{e:Repl4} are bounded uniformly over
$N$, $\mu$ and $k_{1:n}\in\Z^{2n}\setminus\{0\}$. 
Throughout, in order to lighten the notation we will omit the arguments of 
the functions involved unless needed. Moreover, for~\eqref{e:Repl1},~\eqref{e:Repl2} and~\eqref{e:Repl3}, 
$\nonlin$ won't play any role, therefore we directly bound it by $1$. 
\medskip

Let us begin with~\eqref{e:Repl1}. 
We add and subtract a term from the integrand so that we obtain
\begin{equs}[e:Repl1+]
&\int_{\R^2} \dd x\,\, \Big|\frac{\mu_N+\Gamma H^+}{[\mu_N +\Gamma H]^2}-\frac{\mu_N+\Gamma H^+}{[\Gamma_1 H]^2}+\frac{\mu_N+\Gamma H^+}{[\Gamma_1 H]^2} -\frac{H^+}{\Gamma_1 H^2}\Big|\\
&\leq \int_{\R^2} \dd x\,\, (\mu_N+\Gamma H^+) \frac{|[\Gamma_1 H]^2-[\mu_N +\Gamma H]^2|}{[\mu_N +\Gamma H]^2[\Gamma_1 H]^2}+\frac{1}{[\Gamma_1 H]^2}|\mu_N+\Gamma H^+ -\Gamma_1 H^+|\,.
\end{equs}
Notice that, for the first summand, since $\Gamma_1=\mu_N+\Gamma$ and $H$ is bounded, the numerator 
can be controlled as
\[
[\Gamma_1 H]^2-[\mu_N +\Gamma H]^2=\mu_N[H-1][\mu_N(H+1)+2\Gamma H]\lesssim \mu_N \Gamma_1
\]
for the second instead, we use that also $H^+$ is bounded, so that 
\[
|\mu_N+\Gamma H^+ -\Gamma_1 H^+|=\mu_N|H^+-1|\lesssim \mu_N
\]
As further $H\geq 1$,~\eqref{e:Repl1+} can be bounded above by 
\begin{equ}[e:Basic]
\mu_N\int_{\R^2}\frac{\Gamma_1(x)^2}{\Gamma_1(x)^4} +\frac{1}{\Gamma_1(x)^2}\lesssim \mu_N \int \frac{\dd x}{(\mu_N +|x|^2)^2}\lesssim \mu_N\int_0^\infty\frac{\dd r}{(r+\mu_N)^2}\lesssim 1\,,
\end{equ}
which completes the argument for~\eqref{e:Repl1}. 
\medskip

We turn to~\eqref{e:Repl2}. Similarly to~\eqref{e:Repl1+}, we add and subtract a term whose numerator 
is that of the first summand while the denominator is that of the second, so that, by the triangular inequality we 
need to bound
\begin{equs}[e:Repl2.2]
\int_{\R^2}  \dd x\,\, &H^+(\Ll(N^2\Gamma_1(x)))\frac{|\Gamma_2(x)H(\Ll(N^2\Gamma_2(x)))^2-\Gamma_1(x)H(\Ll(N^2\Gamma_1(x)))^2|}{\Gamma_1(x)H(\Ll(N^2\Gamma_1(x)))^2\Gamma_2(x)H(\Ll(N^2\Gamma_2(x)))^2}
\\
&+\frac{1}{\Gamma_2(x)H(\Ll(N^2\Gamma_2(x)))^2}|H^+(\Ll(N^2\Gamma_1(x)))- H^+(\Ll(N^2\Gamma_2(x)))|
\end{equs}
For the first summand, we introduce the function $G(z)=z H(\Ll(N^2 z))^2$ and note that its derivative is 
such that
\begin{equ}
|G'(z)|=\left|H(\Ll(N^2 z))^2 -\frac{\fc}{\log N^2} \frac{2H(\Ll(N^2 z))H'(\Ll(N^2 z))}{z+1}\right|\lesssim 1\,,\qquad z\in\R_+
\end{equ}
thanks to the definition of $\Ll$ in~\eqref{e:L} and the assumptions on $H$. 
Now, in terms of $G$, the numerator equals
\begin{equ}
|G(\Gamma_2(x))-G(\Gamma_2(x))|\lesssim |\Gamma_2(x)-\Gamma_1(x)|
\end{equ}
and the bound follows by mean value theorem. 
For the second summand in~\eqref{e:Repl2.2}, we note that 
\begin{equ}
|H^+(\Ll(N^2 z))'|=\frac{\fc}{\log N} \left|\frac{(H^+)'(\Ll(N^2 z))}{z(z+1)}\right|\lesssim \frac{1}{z}\,.
\end{equ}
which, arguing as above, implies
\begin{equs}
|H^+(\Ll(N^2\Gamma_1))-H^+(&\Ll(N^2\Gamma_2))|\leq \Big(\sup_{z}\frac{1}{z}\Big) |\Gamma_1-\Gamma_2|=\frac{|\Gamma_1-\Gamma_2| }{\Gamma_1\wedge\Gamma_2}
\end{equs}
where the supremum is over the interval $[\Gamma_1\wedge\Gamma_2,\Gamma_1\vee\Gamma_2 ]$. 
Now, by~\eqref{e:Gammas} we have 
\begin{equs}[e:Diff]
|\Gamma_2(x)-\Gamma_1(x)|&=||x|^2 +\tfrac12 |k_1/N|^2-\tfrac12 |x|^2 -\tfrac12 |k_1/N-x|^2|\\
&=|\langle k_1/N, x\rangle|\leq |k_1/N||x|\,.
\end{equs}
Since further $H^+$ is bounded, $H\geq 1$ and $\Gamma_1$ and $\Gamma_2$ are 
comparable by~\eqref{e:lm},~\eqref{e:Repl2.2} is bounded above by 
\begin{equs}
\int_{\R^2}\frac{|\Gamma_2(x)-\Gamma_1(x)|}{\Gamma_2(x)^2}&\lesssim |k_1/N|\int\frac{|x| \dd x}{(|x|^2 +\alpha_N)^2}\\
&\lesssim |k_1/N|\int_0^{\infty} \frac{r^2\dd r}{(r^2+\alpha_N)^2} =\frac{\pi}{4}\frac{|k_1/N|}{\sqrt{\alpha_N}}\leq 1\,,
\end{equs}
which concludes the proof of~\eqref{e:Repl2}. 
\medskip

The analysis of~\eqref{e:Repl3} is immediate and therefore we omit it. Hence we are left with~\eqref{e:Repl4}. 
At first, we upper bound the ratio  $H^+/H^2$ by $K$. 
To treat what is left, we split the integral over two regions corresponding to 
$|k_1/N-x|\leq \tfrac12|k_1/N|$ and $|k_1/N-x|> \tfrac12|k_1/N|$ and 
begin by the former. 
Using the definition of $\nonlin$ in~\eqref{e:nonlinCoefficient}, we  bound the difference by $1$, 
exploit the fact that $\Gamma_1(x)\gtrsim |x|^2$ and 
note that since $|k_1/N-x|\leq \tfrac12|k_1/N|$, $|x|\in[\tfrac12|k_1/N|, \tfrac32|k_1/N|]$. 
Then, the integral is bounded above by 
\begin{equ}
\int_{|x|\in[\tfrac12|k_1/N|, \tfrac32|k_1/N|]}\frac{\dd x}{|x|^2}\lesssim 1\,.
\end{equ}
Instead, for $x$ such that $|k_1/N-x|> |k_1/N|/2$, we use the bilinearity of $c$ 
in~\eqref{e:nonlinCoefficient} to get
\begin{equ}
c(x,k_1/N-x)^2=c(x,-x)^2+c(x,k_1/N) c(x,k_1/N-2x)
\end{equ}
so that
\begin{equs}
\Big|&\frac{c(x,k_1/N-x)^2}{|x|^2|k_1/N-x|^2}-\frac{c(x,-x)^2}{|x|^4}\Big|\\
&\leq \frac{c(x,-x)^2}{|x|^2}\left|\frac{1}{|k_1/N-x|^2}-\frac{1}{|x|^2}\right|+\frac{|c(x,k_1/N) c(x,k_1/N-2x)|}{|x|^2|k_1/N-x|^2}\\
&\lesssim \frac{|k_1/N||k_1/N-2x|}{|k_1/N-x|^2} \lesssim \frac{|k_1/N|}{|k_1/N-x|}\,.
\end{equs}
Here, we used that $c(x,-x)\leq |x|^2$.
%where in the last passage we used that $|k_1/N-x|> |k_1/N|/2$. 
Therefore, we obtain 
\begin{equs}
\int_{|k_1/N-x|> \tfrac12|k_1/N|}\dd x&\left|\frac{c(x,k_1/N-x)^2}{|x|^2|k_1/N-x|^2}-\frac{c(x,x)^2}{|x|^4}\right|\frac{1}{\Gamma_1(x)}\\
&\lesssim |k_1/N|\int_{|k_1/N-x|> \tfrac12|k_1/N|}\frac{\dd x}{|k_1/N-x|^3}\lesssim 1\,,
\end{equs}
and the proof is concluded. 
\end{proof}

\begin{remark}
Arguing as in the proof of~\eqref{e:Repl1}, one can also show 
\begin{equ}[e:Repl1.1]
\int_{\R^2}  \dd x\,(\nonlin_{xN,k_1-xN})^2 \left|\frac{1}{\mu_N +\Gamma(x)H(\Ll(N^2\Gamma_1(x)))}-\frac{1}{\Gamma_1(x)H(\Ll(N^2\Gamma_1(x)))}\right|
\end{equ}
is uniformly bounded
which will be used in the proof of Proposition~\ref{e:Dbulkn}. 
\end{remark}

We are now ready to prove Proposition~\ref{p:Approx}. 
\begin{proof}[of Proposition~\ref{p:Approx}]
As a first step, we note that, performing a Riemann sum approximation, 
the sum in~\eqref{e:J} can be written as 
\begin{align}
&\sum_{\ell + m=k_1}\frac{1}{N^2}(\nonlin_{N(\ell/N),N(m/N)})^2\times\nonumber\\
&\,\,\,\times \frac{\mu_N +\tfrac12(|\tfrac{\ell}{N}|^2+|\tfrac{m}{N}|^2+\beta_N)H^+(\Ll(N^2(\mu_N+\tfrac12(|\tfrac{\ell}{N}|^2+|\tfrac{m}{N}|^2+\beta_N))))}{[\mu_N +\tfrac12(|\tfrac{\ell}{N}|^2+|\tfrac{m}{N}|^2+\beta_N)H(\Ll(N^2(\mu_N+\tfrac12(|\tfrac{\ell}{N}|^2+|\tfrac{m}{N}|^2+\beta_N))))]^2}\label{e:RS}\\
&= \int \dd x (\nonlin_{Nx,k_1-N x})^2\frac{\mu_N +\Gamma(x) H^+(\Ll(N^2\Gamma_1(x)))}{[\mu_N +\Gamma(x)H(\Ll(N^2\Gamma_1(x)))]^2}+o(1)\nonumber
\end{align}
where we adopted the conventions in~\eqref{e:conv} and~\eqref{e:Gammas}, and the $o(1)$ vanishes 
for $N$ large as a consequence of the Riemann summation. 
Thanks to the previous and Lemma~\ref{l:MainRepl}, we have 
\begin{equ}
  P^N(\mu,k_{1:n})=\frac{4\hat{\lambda}^2}{\log N}\int \dd
  x\,\,(\nonlin_{Nx,-Nx})^2
  \frac{H^+(\Ll(N^2\Gamma_2(x)))}{\Gamma_2(x)(\Gamma_2(x)+1)H(\Ll(N^2\Gamma_2(x)))^2}
  +o(1)
\end{equ}
where $o(1)$ converges to $0$ as $(\log N)^{-1}$ uniformly over $\mu\ge0$ and $k_{1:n}\in\Z^{2n}\setminus\{0\}$. 
For the latter integral, we pass to polar coordinates and exploit the fact that $c(x,-x)^2=r^4\cos(2\theta)^2$, 
which gives
\begin{equs}
% &\frac{4\hat{\lambda}^2}{\log N}\int  \dd x\,\,(\nonlin_{x,x})^2\frac{H^+(\Ll(N^2\Gamma_2(x)))}{\Gamma_2(x)(\Gamma_2(x)+1)H(\Ll(N^2\Gamma_2(x)))}\\
&\frac{4\hat{\lambda}^2}{\log N^2}\frac{1}{(2\pi)^2}\int_{|x|\in[1/N,1]} \frac{c(x,-x)^2}{|x|^4}\frac{H^+(\Ll(N^2\Gamma_2(x)))}{\Gamma_2(x)(\Gamma_2(x)+1)H(\Ll(N^2\Gamma_2(x)))^2}\\
&=\frac{\hat{\lambda}^2}{\log N^2}\frac{1}{\pi^2}\int_0^{2\pi} \cos(2\theta)^2\int_{1/N}^1 \frac{H^+(\Ll(N^2(r^2+\alpha_N))) \,r\dd r}{(r^2+\alpha_N)(r^2+\alpha_N+1)H(\Ll(N^2(r^2+\alpha_N)))^2} \\
%&=\frac{\fc}{\log N^2}\int_{1/N^2+\alpha_N}^{1+\alpha_N} \frac{H^+(\Ll(N^2\rho)) \dd \rho}{\rho(\rho+1)H(\Ll(N^2\rho))^2}\\
&=\frac{\fc}{\log N^2}\int_{\alpha_N}^{1+\alpha_N} \frac{H^+(\Ll(N^2\rho)) \dd \rho}{\rho(\rho+1)H(\Ll(N^2\rho))^2}+o(1)
\end{equs}
where $\fc$ is defined in~\eqref{e:nueff} and the last step follows by
\begin{equation}\label{e:CoV}
\int_{\alpha_N}^{1/N^2+\alpha_N}\frac{H^+(\Ll(N^2\rho)) \dd \rho}{\rho(\rho+1)H(\Ll(N^2\rho))^2}\lesssim \int_{\alpha_N}^{1/N^2+\alpha_N}\frac{\dd \rho}{\rho}=\log\left(1+\frac{1}{\mu+\tfrac12|k_{1:n}|^2}\right)
\end{equation}
and the latter is uniformly bounded by $1$ as $|k_{1:n}|\geq 1$. By the definition of $\Ll$ in~\eqref{e:L}, we see that 
\begin{equs}
\frac{\fc}{\log N^2}\int_{\alpha_N}^{1+\alpha_N} \frac{H^+(\Ll(N^2\rho)) \dd \rho}{\rho(\rho+1)H(\Ll(N^2\rho))^2} %\int_{\Ll(N^2(1+\alpha_N))}^{\Ll(N^2\alpha_N)}\frac{H^+(y)}{H(y)^2}\dd y\\
=\int_0^{\Ll(N^2\alpha_N)}\frac{H^+(y)}{H(y)^2}\dd y+o(1)
\end{equs}
since 
\begin{equ}
\int_0^{\Ll(N^2(\alpha_N+1))}\frac{H^+(y)}{H(y)^2}\dd y\lesssim \Ll(N^2(\alpha_N+1))=\frac{\fc}{\log N^2}\log\left(1+\frac{1}{\alpha_N+1}\right)\lesssim \frac{1}{\log N^2}\,,
\end{equ}
where we bounded the ratio of $H^+$ and $H$ by $K$, and the proof is concluded. 
\end{proof}

\section{The operator $\cS$}
\label{app:S}
\begin{lemma}\label{l:OpS}
Let $\cS$ be the diagonal operator on $L^2(\eta)$ whose Fourier multiplier is defined as in~\eqref{e:S}. 
Let $j\in\N$ and  $\fg\in\fock_{j-1}$. Define $\ft^N\in\fock_{j}$ according to
\begin{equ}[e:RecGen3]
\ft^N\eqdef (-\gensy)^{-1}\genap\fg\,.
\end{equ}
Then, there exists a constant $C=C(j)>0$ such that 
\begin{equ}[e:SRec]
\|\cS^\half \ft^N\|^2\leq C\Big( \|\cS^\half \fg\|^2\mathds{1}_{j>2} +\frac{1}{\log N}\|(-\gensy)^\half \fg\|^2\Big)\,.
\end{equ}
\end{lemma}
\begin{proof}
First, we decompose the square of the left hand side of~\eqref{e:SRec} in diagonal and off-diagonal parts 
as in Lemma~\ref{l:DiagOffDiag} and bound, via Cauchy-Schwarz, the off-diagonals with the diagonal up 
to a constant that only depends on $j$. 
Therefore, we are left to bound
\begin{equs}[e:SRec1]
\langle \cS (-\gensy)^{-1}&\genap\fg, (-\gensy)^{-1}\genap\fg\rangle_{\Di}\\
&\lesssim \lambda_N^2 \sum_{a\ne b\in\{1,\dots, j\}} \sum_{k_{1:j}}\frac{|k_a|^2}{|k_b|}\frac{|k_1+k_2|^2}{|k_{1:j}|^4} (\nonlin_{k_1,k_2})^2 
|\hat\fg(k_1+k_2,k_{3:j})|^2\\
&\lesssim \lambda_N^2 \sum_{a\ne b\in\{1,\dots, j\}}\sum_{k_{1:j}}\frac{1}{|k_b|}\frac{|k_1+k_2|^2}{|k_{1:j}|^2} 
|\hat\fg(k_1+k_2,k_{3:j})|^2\mathds{1}_{|k_1|,|k_2|\le N}.
\end{equs}
We need to distinguish a few cases. If $\{a,b\}\cap\{1,2\}=\emptyset$, or $a\in\{1,2\}$ and $b>2$ 
(which is possible only if $j>2$), 
then the inner sum in~\eqref{e:SRec1} is bounded above by
\begin{equs}
\hat\lambda^2\sum_{k,\,k_{3:j}}\frac{|k|^2}{|k_b|} |\hat\fg(k,k_{3:j})|^2\Big(\frac{1}{\log N} \sum_{k_1+k_2=k}\frac{1}{|k_1|^2+|k_2|^2}\Big)\lesssim \|\cS^\half \fg\|^2
\end{equs} 
as the quantity in parenthesis is bounded above by a constant and $\hat\fg$ is symmetric by definition. 
If instead  either $b\in\{1,2\}$ and $a>2$, or $|\{a,b\}\cap\{1,2\}|=2$, then 
we control~\eqref{e:SRec1} as 
\begin{equs}
\const^2\sum_{k,\,k_{3:j}}|k|^2 |\hat\fg(k,k_{3:j})|^2\sum_{k_1+k_2=k}\frac{1}{|k_b|(|k_1|^2+|k_2|^2)}\lesssim \const^2\|(-\gensy)^\half \fg\|^2
\end{equs}
as the inner sum can be easily checked to be bounded. Since $\const^2\sim 1/\log N$, the result follows at once. 
\end{proof}

\begin{proof}[of Proposition \ref{prop:Spiccolo}]
  We can assume $j\ge2$ since otherwise the norm we want to estimate is just $0$.
  If $j>2$ then we recall that $\ffNn_{a,j}$ is defined by the second line in \eqref{e:Fjn} and in the computation of the norm
  $\|\cS^\half \ffNn_{a,j}\|$ we can replace $G_{n+2-j}$ (which is positive) by zero.
  We have then
  \begin{equs}
    \|\cS^\half \ffNn_{a,j}\|\le \|\cS^\half (-\gensy)^{-1}\genap \ffNn_{a,j-1}\|\le C(j)\left(\|\cS^\half \ffNn_{a,j-1}\|^2+\frac1{\log N}\|(-\gensy)^\half \ffNn_{a,j-1}\|^2
      \right)
    \end{equs}
    where the last inequality follows from Lemma \ref{l:OpS}. 
    The norm $\|(-\gensy)^\half \ffNn_{a,j-1}\|$ is bounded as $N\to\infty$ by 
    Proposition \ref{l:PrelimBounds} so that, by iteratively applying Lemma~\ref{l:OpS} over $j$ it is enough to show
    \begin{equs}
      \label{e:afa}
    \lim_{N\to\infty}\|\cS^\half \ffNn_{a,2}\|=0.
    \end{equs}
     Recall \eqref{e:Fjn}. If $a=1$ (so that $\ffNn_a=\fbNn$) then $\ffNn_{a,2}$ is defined by the second line 
     of \eqref{e:Fjn} and \eqref{e:afa} follows again by Lemma \ref{l:OpS}. If instead $a=2$ 
     (so that $\ffNn_a=\fhNn$) then $\ffNn_{a,2}$ is defined by the first line of \eqref{e:Fjn} with 
     $\fg_2={\mathfrak n}^N_0$ (see \eqref{eq:explFock}). Then,
     \begin{equs}
       \|\cS^\half \ffNn_{2,2}\|^2\le         \|\cS^\half(-\gensy)^{-1}{\mathfrak n}^N_0\|^2\lesssim\frac{1}{\log N}\sum_{k_1+k_2=0}\frac1{(|k_1|^2+|k_2|^2)^2}\frac{|k_1|^2}{|k_2|}(\nonlin_{k_1,k_2})^2\\
       \lesssim \frac{1}{\log N}\sum_{k_1\ne 0} \frac1{|k_1|^3}  \lesssim \frac{1}{\log N}\,
     \end{equs}
     from which the result follows at once.
 \end{proof}

\section{The variance of the quadratic variation: a concrete example}
\label{sec:example}

\begin{figure}\begin{center}
\includegraphics[width=4cm]{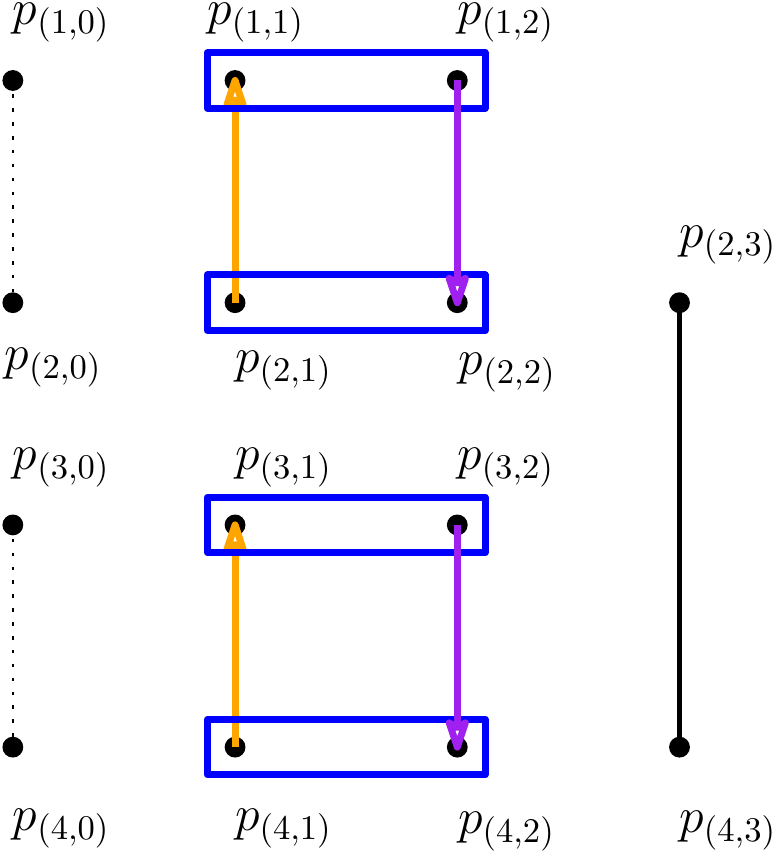}%\qquad\qquad\qquad\includegraphics[width=4cm]{k2fig2.pdf}
\caption{An example of a diagram involved in the estimation of the variance in \eqref{e:Var}.
}
\label{example}
\end{center}
\end{figure}

In this section we provide a concrete example to illustrate 
how to apply the estimates associated to the various colours outlined in Section~\ref{sec:tourdeforce}. 
We focus on the simplest case, namely $\kappa=0$. 
The graph we consider is given in Figure~\ref{example}, and our starting point is~\eqref{e:trem5}. 
We first note that, with the notation introduced in Section~\ref{sec:tourdeforce}, 
the only edge in $U$ is $((2,3), (4,3))$. 
We therefore first sum over $p_{(2,3)}$ and $p_{(4,3)}$. Hence, we estimate
\begin{equs}
&\sum_{p_{(2,3)}, p_{(4,3)}}\1_{p_{(2,3)}=-p_{(4,3)}}	|\fs^2(p_{(2,2)}+p_{(2,1)}, p_{(2,0)}, p_{(2,3)})||\fs^4(p_{(4,2)}+p_{(4,1)}, p_{(4,0)}, p_{(4,3)})|\\
&\leq \Big(\sum_{p_{(2,3)}}|\fs^2(p_{(2,2)}+p_{(2,1)}, p_{(2,0)}, p_{(2,3)})|^2\Big)^{1/2}\Big(\sum_{p_{(4,3)}} |\fs^4(p_{(4,2)}+p_{(4,1)}, p_{(4,0)}, p_{(4,3)})|^2\Big)^{1/2}\,.
	\end{equs}
It remains to sum over the remaining vertices. 
Observe that the remaining edges in $D$ form two closed loops, 
and we colour the edges purple and orange as in Figure~\ref{example}. 
Since in our graphical notation there are no other edges between the first two rows and the last two rows, 
we see that the remaining sum can be factorised and 
we can independently sum over $\{p_{(u,i)}:\, u\in\{1,2\},\, i\ge1\}$ and $\{p_{(u,i)}:\, u\in\{3,4\},\, i\ge1\}$. 
Since these two sums can be treated in the exact same way we only sum over $\{p_{(u,i)}:\, u\in\{1,2\},\, i\ge1\}$.
Note that in that case $(1,1)$ and $(2,1)$ are the vertices on the orange edge, 
so that by removing $|p_{(1,1)}|^2$ and $|p_{(2,1)}|^2$ from the denominator we are left to estimate
\begin{equs}[e:appC]
\sum_{p_{(u,i)}} &\1_{p_{(1,0)}=-p_{(2,0)}}\1_{p_{(1,1)}=-p_{(2,1)}}\1_{p_{(1,2)}=-p_{(2,2)}}\times\\
&\times\frac{\fs^1(p_{(1,1)}+p_{(1,2)}, p_{(1,0)})|f_2(p_{(2,1)}+p_{(2,2)}, p_{(2,0)})\prod_{u=1}^{2}|p_{(u,0)}||p_{(u,1)}+p_{(u,2)}|}{(|p_{(1,0)}|^2+|p_{(1,2)}|^2)(|p_{(2,0)}|^2+|p_{(2,2)}|^2)}\,,
	\end{equs}
where $u=1,2$ and $i=1,2,3$ and we defined
\begin{equ}
f_2(p_{(2,1)}+p_{(2,2)}, p_{(2,0)})= \Big(\sum_{p_{(2,3)}}|\fs^2(p_{(2,1)}+p_{(2,2)}, p_{(2,0)}, p_{(2,3)})|^2\Big)^{1/2}\,. 	
\end{equ}
We first perform the sum with respect to the vertices of the orange edge, i.e., over $p_{(1,1)}$ and $p_{(2,1)}$. 
We bound, using~\eqref{e:ode},
\begin{equs}
&\sum_{p_{(1,1)}, p_{(2,1)}} \1_{p_{(1,1)}=-p_{(2,1)}}	 \prod_{u=1}^{2}|p_{(u,1)}+p_{(u,2)}|\fs^1(p_{(1,1)}+p_{(1,2)}, p_{(1,0)})|f_2(p_{(2,1)}+p_{(2,2)}, p_{(2,0)})\\
&\quad\leq \Big(\sum_{p_{(1,1)}}|p_{(1,1)}|^2|\fs^1(p_{(1,1)}, p_{(1,0)})|^2\Big)^{1/2}\Big(\sum_{p_{(2,1)}}|p_{(2,1)}|^2 f_2(p_{(2,1)}, p_{(2,0)})^2\Big)^{1/2}\\
&\quad= \Big(\sum_{p_{(1,1)}}|p_{(1,1)}|^2|\fs^1(p_{(1,1)}, p_{(1,0)})|^2\Big)^{1/2}\Big(\sum_{p_{(2,1)}, p_{(2,3)}}|p_{(2,1)}|^2 |\fs^2(p_{(2,1)}, p_{(2,0)}, p_{(2,3)})|^2\Big)^{1/2}\,.
	\end{equs}
The next sum we perform is over the purple edge, i.e. over $p_{(1,2)}$ and $p_{(2,2)}$. 
To that end, note that the right hand side above does not depend on those two summation variables. 
Hence, the sum over $p_{(1,2)}$ and $p_{(2,2)}$ is given by
\begin{equs}\label{e:purpleappendix}
\sum_{p_{(1,2)}, p_{(2,2)}}\1_{p_{(1,2)}=-p_{(2,2)}}\frac{1}{(|p_{(1,0)}|^2+|p_{(1,2)}|^2)(|p_{(2,0)}|^2+|p_{(2,2)}|^2)}	\lesssim\frac{1}{|p_{(1,0)}||p_{(2,0)}|}\,,
	\end{equs}
where we made use of~\eqref{e:pde}.
Combining the above estimates, we are left with
\begin{equ}
\sum_{\substack{p_{(1,0)}, p_{(2,0)}\\ p_{(1,0)}=-p_{(2,0)}}}\Big(\sum_{p_{(1,1)}}|p_{(1,1)}|^2|\fs^1(p_{(1,1)}, p_{(1,0)})|^2\Big)^{1/2}\Big(\sum_{p_{(2,1)}, p_{(2,3)}}|p_{(2,1)}|^2 |\fs^2(p_{(2,1)}, p_{(2,0)}, p_{(2,3)})|^2\Big)^{1/2}.
	\end{equ}
Here, we used that the factor on the right hand side of~\eqref{e:purpleappendix} cancels the term $\prod_{u=1,2}|p_{(u,0)}|$ in \eqref{e:appC}. The above however equals the expression in~\eqref{e:kappa0almostfinal} and can be bounded in the same way as in~\eqref{e:kappa0final}. 
\end{appendix}

\section*{Acknowledgements}
The authors would like to thank Massimiliano Gubinelli for helpful discussions and suggestions. 
G. C. gratefully acknowledges financial support via the EPSRC grant EP/S012524/1. 
D. E. gratefully acknowledges financial support 
from the National Council for Scientific and Technological Development - CNPq via a 
Universal grant 409259/2018-7, and a Bolsa de Produtividade 303520/2019-1. F. T. gratefully acknowledges financial support of Agence Nationale de la Recherche via the
ANR-15-CE40-0020-03 Grant LSD.

\bibliography{bibtex}
\bibliographystyle{Martin}

\end{document}